\documentclass[letterpaper, 11pt,  reqno]{amsart}
\usepackage{package_rapport}
\usepackage{float}
\usepackage{xpatch}
\usepackage[pagewise]{lineno}
\usetikzlibrary{arrows}
\usetikzlibrary{calc, decorations.pathreplacing, positioning}

\tikzset{
	treenode/.style = {align=center, inner sep=0pt, text centered,
		font=\sffamily},
	arn_n/.style = {treenode, circle, white, font=\sffamily\bfseries, draw=black,
		fill=black, text width=1.5em},
	arn_r/.style = {treenode, circle, red, draw=red, 
		text width=1.5em, very thick},
	arn_b/.style ={treenode, circle, blue, draw=blue, 
		text width=1.5em, very thick},
	arn_x/.style = {treenode, rectangle, draw=black,
		minimum width=1em, minimum height=1em}
}

\makeatletter
\@namedef{subjclassname@2025}{
	\textup{2025} Mathematics Subject Classification}

\makeatother

\begin{document}
	\baselineskip = 13.5pt

	\title[]
	{Long time well-posedness and wave turbulence limit for a quadratic non-linear Schrödinger equation with well prepared random initial data}

	\author{Enguerrand Brun \textsuperscript{1} \vspace{1cm} \\ 
		\textsuperscript{1} UMPA, ENS \MakeLowercase{de }L\MakeLowercase{yon, }L\MakeLowercase{yon, }F\MakeLowercase{rance,} \\
		E\MakeLowercase{mail: enguerrand.brun@ens-lyon.fr}}
	\thanks{\textsuperscript{1}Corresponding author: Enguerrand Brun}

	\subjclass[date]{Primary: 35Q55, 35A01, 35A02; Secondary: 35Q82}
	
	
	\vspace*{-8mm}
	\begin{abstract}
		We prove a long time well-posedness result for a quadratic non-linear Schrödinger equation with random well prepared initial conditions in dimension $d\ge 1$. In some regimes, we derive the Wave Kinetic equation associated with the equation in dimension $d\ge 5$, almost up to the kinetic time to the power $d/(d+1)$. Our analysis relies on the ideas introduced by Deng and Hani.
	\end{abstract}
	
	\date{\today}
	\maketitle

	\tableofcontents
	
	\noi \textbf{AI declaration.}  No AI tools were used in this work.
	
	\section{Introduction}
	\subsection{Setting and main results}\label{subsection1}
	We aim to apply the ideas of the recent work of Deng and Hani, specifically, \cite{Deng_2021} to a new context.
	\noi We are interested in showing a long time well-posedness result in dimensions $d\ge1$ and deriving a wave turbulence result in dimensions $d\ge5$ for a quadratic non-linear Schrödinger equation. 
	\subsubsection{Long time well-posedness}
	The point of our well-posedness result is that it uses randomness in the initial data to guaranty the existence of the solution for longer time intervals than the deterministic theory ; similarly to how randomness can help improve the regularity threshold for which we obtain local well-posedness such as in \cite{bourgain1996invariant}, \cite{burq2008random}, \cite{burq2013probabilistic} and \cite{colliander2012almost}. In these papers, frequencies of all sizes matter, while in our case, only high frequencies (of size $\sim L$) contribute significantly. The offset is that our set of possible initial data is not dense in Sobolev spaces, unlike in the low regularity local well-posedness theory.  \\
	
	\noi More precisely, we will show local well-posedness for equation \eqref{nls} below for well prepared initial data on the flat torus up to a longer time than what is known by the deterministic theory on a set of high probability. We also provide a propagation of $L^\infty$ bounds thanks to randomness, much like in \cite{dengwavetwo}. This result is presented in Theorems \ref{theorem1} and \ref{theorem2}.  \\
	
	\noi We study the following quadratic nonlinear Schrödinger equation :
	
	\begin{equation}
		\begin{cases} \label{nls}
			i\dt u - \Dl u =  u^2  \\
			u(0,x)=u_0(x) 
		\end{cases}
	\end{equation}
	where $\Dl:=\sum_{i=1}^d \partial_{x_i}^2$, $(t,x)\in \R\times \T_L^d$, $ \T_L^d:=\R^d/\Z_L^d$, $\Z_L^d := \frac{1}{L} \Z^d$.  \\
	
	\noi First, we recall previously known results of local well-posedness for \eqref{nls}, which rely on Bourgain's analysis :
	\begin{proposition} \label{basiclwp}
		\renewcommand{\theenumi}{\roman{enumi}}%
		\item $ \ \forall  \, s>d/2-1$, $\exists \, \, C >0$, $\forall  \, L \ge 1,$ $\forall \, \, u_0\in H^s(\T_L^d)$, $\exists ! \, u \in C^0([-T,T],H^s(\T_L^d))$ solution of~~\eqref{nls} where $T=C \|u_0\|_{H^s}^{-1}$. 
	\end{proposition}
	\begin{proof}
		See Section \ref{basicprooflwp} of the appendix.
	\end{proof}
	\noi We want to improve the time of existence given by Proposition \ref{basiclwp}.
	\noi This will be done with our main result, which is a long time well-posedness result on the torus $\T_L^d$. We split it into two theorems, the first using only trivial counting estimates and the other divisor bounds. \\
	
	\noi In the Theorems below, the $h_{T_{max}}^{s,b}$ spaces are Bourgain spaces defined in Section~\ref{section4}. We recall that $C^\infty\big([0,T_{max}]\times\T_L^d\big)\subset h_{T_{max}}^{s,b} \subset C^0\big([0,T_{max}],H^s(\T_L^d)\big) $,  $b>1/2$. \\
	
	\begin{theorem}\label{theorem1}
		Let $d\ge 1$, $\delta >0$ arbitrarily small, $\gamma > \delta$ arbitrarily large, $L^{-\gamma}\le \al \le L^{-\delta}$ and the well prepared initial data
		\begin{align*}
			u_0(x)=\frac{\al}{L^{d/2}}\sum_{k\in\Z^d}\sqrt{n_{in}(k/L)}g_k(\omega) e^{i2\pi (k/L)\cdot x}
		\end{align*} with $n_{in}\in \mathcal{S}(\R^d)$ non-negative, $(g_k(\omega))_{k\in \Z^d}$ family of i.i.d. centered normalized complex Gaussians. Then we have that there exists $\theta>0$, $K:=K(\theta,\gamma ),c:=c(\theta,\gamma)>0$, $s>d/2$ and $b>1/2$ close to $1/2$ such that for all $L$ large enough, there exists an event $E_L\subset \O$ verifying  $\P(E_L^c)\le Ke^{-cL^\theta}$, such that the solution to \eqref{nls} is defined up to time $T_{max}$ on $E_L$ in $h_{T_{max}}^{s,b}$ where :
		\begin{equation*}
			T_{max} =\al^{-1}L^{-\delta} \, .
		\end{equation*}
		\noi Furthermore, there exists $C>0$ such that for all $L\ge 1$ :
		\begin{equation} \label{boundpropagation}
			\|u\|_{L_t^\infty((0,T_{max}), L_x^\infty)} \le  \al L^{\theta+C(b-1/2)} \, .
		\end{equation}
	\end{theorem}

	\noi For $\al$ decaying faster than any negative power of $L$, we do not get such results but we will see later that for such regimes, the factor we gain for the time of existence compared to what we had with Proposition \ref{basiclwp} is negligible. \\

	
	\noi It is important to note that in Theorem \ref{theorem1} we only make use of trivial counting estimates (in Sections \ref{proofbound1} and \ref{proofbound2}) in the analysis of Feynman tree expansions. Is is worth mentioning that this result can be achieved without using the fine structure of Picard iterates by relying only on known Bourgain analysis (writing the Duhammel formula in the Fourier space) coupled with probabilistic cancelations. This already gives a longer time than the one given by Proposition \ref{basiclwp} as we shall see below. As we only use trivial counting estimates, this result could be extended to the case where we only have fractional dispersion i.e. the Laplacian is replaced with the fractional Laplacian $(-\Dl)^\beta$ where $0< \beta \le1$.\\
	
	\noi Now we state a stronger result but for shorter range of $\al$ where we make use of number theory, i.e. use divisor bounds.
	\begin{theorem}\label{theorem2}
		Let $d\ge 1$, $\delta >0$, $L^{-\frac{d(d+1)}{4d-2}-\delta}\le\al \le L^{-\delta}$ and the well prepared initial data
		\begin{align*}
			u_0(x)=\frac{\al}{L^{d/2}}\sum_{k\in\Z^d}\sqrt{n_{in}(k/L)}g_k(\omega) e^{i2\pi (k/L)\cdot x}
		\end{align*} with $n_{in}\in \mathcal{S}(\R^d)$ non-negative, $(g_k(\omega))_{k\in \Z^d}$ family of i.i.d centered normalized complex Gaussians. Then we have that there exists $\theta>0$, $K(\theta ),c(\theta)>0$, $s>d/2$ and $b>1/2$ close to $1/2$ such that for all $L$ large enough, there exists an event $E_L\subset \O$ verifying  $\P(E_L^c)\le Ke^{-cL^\theta}$, such that the solution to \eqref{nls} is defined up to time $T_{max}$ on $E_L$ in $h_{T_{max}}^{s,b}$ where :
		\begin{equation*}
			T_{max} =\al^{-\frac{2d}{d+1}}L^{-\frac{2\delta d}{d+1}} \, .
		\end{equation*}
		\noi Furthermore, there exists $C>0$ such that for all $L\ge 1$ :
		\begin{equation} \label{boundpropagation}
			\|u\|_{L_t^\infty((0,T_{max}), L_x^\infty)} \le  \al L^{\theta+C(b-1/2)} \, .
		\end{equation}
	\end{theorem}
	
	\noi We have an additional restriction on $\al$ which is $\al \ge L^{-\frac{d(d+1)}{4d-2}-\delta}$. This restriction comes from the fact that estimates in Section \ref{proofbound2} degenerate when $\al$ is smaller. \\

	\noi We want to compare the time of existence  given by Theorems \ref{theorem1} and \ref{theorem2} and Proposition~\ref{basiclwp}.
	We first notice that $T_{max}$ is greater in Theorem \ref{theorem2} than in Theorem \ref{theorem1}, which is to be expected as its proof uses more refined counting arguments. \\
	To compare with Proposition \ref{basiclwp}, we need the following lemma to assess the size of the initial data :

	\begin{lemma}\label{time of existence}
		Let the setting be the same as in Theorem \ref{theorem1} or \ref{theorem2}. Then for $L$ large enough, on an event of probability $\ges 1-\frac{4c_s}{C_s^2}L^{-d}$, we have that $$\|u_0\|_{H^s}\asymp \al C_s^{1/2}L^{d/2}$$ where
		$$C_s:=\int_\R n_{in}(x)\jb{x}^{2s}dx, \ c_s:=\int_\R \jb{x}^{4s}n_{in}(x)^2dx$$ are independent of $L$. 
	\end{lemma}
	\begin{proof}
		See Section \ref{proofoflemma} of the Appendix.
	\end{proof}
	
	\noi This implies that the time of existence given by Proposition \ref{basiclwp} on $\T_L^d$ on an event of probability $\ges 1-\frac{4c_s}{C_s^2}L^{-d}$ is of order $\al^{-1}L^{-d/2}C_s^{-1/2}$ whereas the time of existence given by Theorems~\ref{theorem1} and \ref{theorem2} is of order 
	$$\begin{cases}
		\al^{-1}L^{-\delta} \ \ \ \text{if $L^{-\gamma}\le \al \le L^{-\delta}$} \\
		\al^{-\frac{2d}{d+1}}L^{-\frac{2\delta}{1+1/d}} \ \ \ \text{if $L^{-\frac{d(d+1)}{4d-2}-\delta}\le \al  \le L^{-\delta}$}
	\end{cases}$$
	on an event of probability $\ge 1-Ke^{-cL^\theta}$. \\

	\noi Compared to Proposition~\ref{basiclwp}, we notice that Theorem \ref{theorem1} gives a better time of existence $T_{max}$ by a factor $L^{d/2}C_s^{1/2}$, which is polynomial in $L$. The timescales in Theorem~\ref{theorem2} are always longer than those in Theorem~\ref{theorem1} and are essentially longer by a factor of $\al^{-\frac{1-1/d}{1+1/d}}L^{d/2}$ than the time of Proposition~\ref{basiclwp}.
	We see that, in contrast to Theorem \ref{theorem1} and \ref{theorem2}, the timescales in Proposition \ref{basiclwp} degenerate as $s\rightarrow +\infty$ because $C_s \xrightarrow[s\rightarrow +\infty]{}+\infty$. \\
	
	\noi Let us observe that the condition $\al \ge L^{- \gamma}$ in Theorem~\ref{theorem1} is not quite restrictive. Indeed, if $\al$ is smaller than $L^{-\gamma}$ for $\gamma \gg 1$, the gain of $L^{d/2}$ is essentially negligible. \\ 
	
	\noi We note that the time of existence given by Theorems \ref{theorem1} and \ref{theorem2} is longer than the blow-up time for the ODE $i\dt u=u^2$ with initial condition in $i\R_+$. This means that our result shows any blow up time for equation \eqref{nls} for the well prepared initial data as in Theorem~\ref{theorem1} or \ref{theorem2} will be longer than the trivial blow up time one can find by taking the ODE $\dt u = u^2$. \\
	
	\noi Equation \eqref{nls} is not a Hamiltonian system and we have no phase invariance. This has an impact, as phase invariance implies the $L^2$ norm conservation. \\
	
	\noi The main difference between our work and \cite{Deng_2021} is that we cannot have a Wick renormalization in our equation \eqref{nls}. This means that degenerate terms will appear in the Feynman tree expansion of the solution. However, the simpler non-linear interaction will show that there are fewer of them. Furthermore, we cannot have self-coupling as our non-linearity contains no conjugates. This simplifies our counting algorithms for the first bound (see Section~~\ref{proofbound1} below). A divisor bound (used in the proof of Proposition \ref{bound2}.) will allow us to gain some power of $\al$ in the time of existence for the regime in Theorem \ref{theorem2}. However, the use of the $TT^*$ argument in Section \ref{proofbound2} causes conjugations to appear, leading to some degenerate affine counting problems related to the conjugated equation of \eqref{nls}, i.e. $i\dt u -\Dl u =|u|^2$. These affine counting problems are what prevents Theorem~\ref{theorem2} from reaching $T_{max}=\al^{-2}$ up to $\eps$. Treating them requires a finer argument to exploit the divisor bounds associated with the other counting problems (see Remark \ref{remarkk-1} below). This difficulty does not seem to appear in the cubic case.  \\
	
	\noi Instead, one could study the spectral radius : $\rho(T)=\lim_{N \rightarrow + \infty} \| T^N\|^{1/N}$. This would be preferable as this would prevent any issues with counting relating to the  $T^*$ operator but this implies tending $N$ to infinity. Thus one cannot take $N$ large enough independent of $L$. Instead $N$ has to depend on $L$ and this introduces many issues we will have to tackle to hopefully reach the kinetic time. \\

	\noi The smaller $\delta $ in Theorems \ref{theorem1} and \ref{theorem2}, the longer $T_{max}$ gets. \\

	\noi Furthermore, the choice of the quadratic non linearity as $u^2$ is specific in the sense that the coupling of the Feynman trees we use in Section \ref{proofbound1} and Section \ref{proofbound2} will not work the same way for $|u|^2$. The dispersion relation will also lead to worse counting problems, meaning worse bounds for the terms in Sections \ref{proofbound1} and \ref{proofbound2}.\\

	\noi The sign of the non-linearity does not play a role in this set of problems because the linear part dominates the nonlinear one. In particular, in the case of the cubic non-linear Schrödinger equation, this domination implies that the potential energy is dominated by the kinetic energy.\\
	
	\noi The bound \eqref{boundpropagation} provides the propagation of the smallness of the $L^\infty$ norm of the initial data thanks to randomness. \\

	
	\noi See Section \ref{rescaledtorus} in the Appendix to see what Theorems \ref{theorem1} and \ref{theorem2} imply on the rescaled torus $\T^d$.

	\subsubsection{Wave turbulence limit} In some regimes, we derive a limit Kinetic equation for equation \eqref{nls} (see \eqref{limitequation} for the limit equation) up to the kinetic (or Van Hove) time to the power $d/(d+1)$ with an $\eps$ loss. This derivation will be weak in the sense that we only get the limiting equation in a Taylor expansion, contrary to stronger statements such as \eqref{strongerconvergence} below. \\
	
	\noi First, we remark that if $u$ is the solution to \eqref{nls} given by Theorem \ref{theorem1} or \ref{theorem2}, then $u/\al$ solves the equation with weak non-linearity :
	\begin{equation}
		\begin{cases} \label{weaknls}
			i\dt u - \Dl u =  \al u^2  \\
			u(0,x)=\frac{1}{L^{d/2}}\sum_{k\in\Z^d}\sqrt{n_{in}(k/L)}g_k(\omega) e^{i2\pi (k/L)\cdot x}
		\end{cases}
	\end{equation}
	where $(t,x)\in \R\times \T_L^d$, $ \T_L^d:=\R^d/\Z_L^d$, $\Z_L^d := \frac{1}{L} \Z^d$. \\
	
	\noi  We recall finding the wave turbulence limit means taking $L \xrightarrow[]{}+\infty$ and $\al \xrightarrow[]{} 0$ (large box and weak non-linearity limit). This derivation only works when the non-linearity is not too weak however (see Remark \ref{CRcondition} below for details). \\
	
	\noi Here is the statement that we obtain in our context :

	\begin{theorem} \label{wavetheorem}
		Under the same assumptions and with the same notations as Theorem \ref{theorem2}, if furthermore $d\ge 5$, $L^{-1-1/d-\delta}\le \al \le L^{- \delta}$ and $T_{kin}:=\al^{-2}$, the solution $u$ verifies for $L^{0+}~\le~t~\le~T_{max}$, for $L$ large enough and for all $ k \in \Z^d$~: 
		\begin{equation} \label{oui}
			\E\big[\left| \Ft_x(u)(t,k/L)\right|^2 \ind_{E_L}\big]=n_{in}(k/L)+\frac{t}{T_{kin}}C(n_{in})(k/L) + r\Big(k,\frac{t}{T_{kin}},L\Big)
		\end{equation}
		where 
		\begin{align*}
			C(f)=C(f,f)(x):=2\int_{\R^d}f(y)f(x-y)\delta(y\cdot(x-y))dy
		\end{align*} 
		and there exists $\beta > 0$ such that for all $t\in[L^{0+},T_{max}]$, $$\lim_{L\xrightarrow[]{}+\infty}\sup_{k \in \Z^d}\left|r\Big(k,\frac{t}{T_{kin}},L\Big)\right|L^\beta =0 \, .$$
	\end{theorem}
	\noi In the above Theorem and in the rest of the paper, $\ft{u}_k$, $\ft{u}(k)$ or $\Ft_x u(t,k)$ denote the Fourier transform in space as defined in Section \ref{section4}. \\
	
	\noi For an explanation on the definition of the kinetic time, see Remark \ref{kinetictimeexplanation} below in Section \ref{heuristic}. \\
	

	\noi We notice that this result gives :
	\begin{equation*}
		T_{max} =(T_{kin}L^{-2\delta})^{\frac{d}{d+1}} \ \ \ \text{if $L^{-1-1/d-\delta}\le \al  \le L^{-\delta}$}
	\end{equation*}
	where we recall that $T_{kin}=\al^{-2}$.
	\\
	
	\noi It follows from Theorem \ref{wavetheorem} that the limit equation is :
	\begin{equation}\label{limitequation}
		\begin{cases}
			\dt f(t,x)=\frac{1}{T_{kin}}C(f,f)(x)=\frac{2}{T_{kin}}\int_{\R^d}f(t,y)f(t,x-y)\delta(y\cdot(x-y))dy \\
			f(0,x)=n_{in}(x)
		\end{cases}
	\end{equation}
	which we will show to be locally well-posed in some functional space in Proposition \ref{LWP} below for $d\ge 4$. \\
	
	\noi The smaller $\delta $ in Theorem \ref{wavetheorem}, the closer $T_{max}$ is to $T_{kin}^{\frac{d}{d+1}}$. \\
	
	\noi The conditions $d\ge 5$ and $\al \ge L^{-1-1/d-\delta}$ are used to justify the limit from the Riemann sum to the integral up to a small enough error in Lemma \ref{sumtointegral} below. \\
	
	\noi The following properties of the i.i.d. standard complex Gaussian random variables are important in our analysis below : $\E\big[g_{k}\cj{g_{k'}}\big]=\delta_{k=k'}$, $\E\big[g_{k}^n\big]=0$ for any integer $n\ge 1$. For $k\neq k'$, this de-correlation property is propagated by the flow thanks to the translation invariance, namely, we have that for $t$ in the interval of the existence of solution, $\E\big[\ft{u}(t,k)\cj{\ft{u}(t,k')}\big]=0$ if $k\neq k'$. On the other hand, for $k=k'$, the property is not propagated, which justifies the need to study $\E\big[\ft{u}(t,k)\cj{\ft{u}(t,k)}\big]$. \\

	\noi  One ideally hopes for convergence in a stronger sense, i.e. a result of the type :
	\begin{equation}\label{strongerconvergence}
		\lim_{L\xrightarrow{}+\infty}\sup_{t\in[0,\delta T_{kin}]}\sup_{k\in \Z_L}\Big|\E\left[|\ft{u}(t,k)|^2\right]-f(t/T_{kin},k)\Big| =0
	\end{equation}
	for some  $\delta>0$ where f is the solution to the limit equation \eqref{limitequation} whose existence we prove locally in time in Section~~\ref{limitequationsection} below. We hope to address this in a future work. \\
	
	\noi One can see that the previous statement \eqref{strongerconvergence} implies the result \eqref{oui} by doing a Taylor expansion~:
	\begin{align*}
		& \Big|\E\left[|\ft{u}(t,k)|^2\right]-n_{in}(k)-\frac{t}{T_{kin}}C(n_{kin},k)\Big|=\Big|\E\left[|\ft{u}(t,k)|^2\right]-n_{in}(k)-\frac{t}{T_{kin}}\dt f(\cdot/T_{kin},k)(0)\Big|  \\
		& \le \Big|\E\left[|\ft{u}(t,k)|^2\right]-f(t/T_{kin},k)+O\Big(\frac{t}{T_{kin}}\Big)^2\Big| \\
		& \le  \sup_{t\in[0,\delta T_{kin}]}\sup_{k\in \Z_L}\Big|\E\left[|\ft{u}(t,k)|^2\right]-f(t/T_{kin},k)\Big| + O\Big(\frac{t}{T_{kin}}\Big)^2
	\end{align*}
	
	\noi One could also try and give an empirical measure interpretation of this result. This is what we do heuristically in Section \ref{empirical measure} of the appendix. We did not find such an empirical measure interpretation in the previous literature on kinetic wave limits. Such an interpretation is always involved in the many-particle case.

	\subsection{Comments}In this section we will first present some physical background of the Wave turbulence derivation. Then, we will give a non-exhaustive presentation of the literature on earlier closely related works. \\
	
	\noi Wave turbulence theory is a field that studies nonequilibrium statistical mechanics for many wave systems. The idea is to extend Boltzmann's kinetic theory for dilute gasses, where we go from microscopic laws to derive new equations on a mesoscopic scale. This kinetic theory was studied in the works \cite{RevModPhys.52.569} and \cite{Spohn1991}. We wish to apply it to systems of nonlinear dispersive waves, as was initiated by the work of \cite{article} and developed in \cite{1992kst..book.....Z} and \cite{2011LNP...825.....N} in Physics. This corresponds formally, for a system of particles, to some convergence of an empirical measure to a new measure verifying Boltzmann's equation. In the case of Boltzmann, the number $N$ of particles will go to infinity as their radius $r$ will go to $0$. Of course, the limit cannot be taken freely in $N$ and $r$ at the same time, as these two parameters must depend on each other for the derivation to work. In the case of nonlinear dispersive waves, the size of the torus $L$  and the strength of the nonlinearity $\al$ will play an analogous role to the number of particles and the radius. The goal of the mathematical approach is to find out which relations the parameters $L$ and $\al$ have to verify for this rigorous derivation to hold. One usually chooses $\al=L^{-\gamma}$ with $\gamma >0$ some constant within an acceptable range. The characteristic time of the limit equation will be $T_{kin}$, i.e. not the same as the initial equation~~\eqref{nls}, where $T_{kin}$ is the kinetic time, dependent on $L$ as such : $T_{kin}=\al^{-2}$ (see Remark \ref{kinetictimeexplanation} as to why this is the case). Finally, one wishes to obtain this derivation for $t\le T_{max}$, where ideally $T_{max}\sim T_{kin}$ and an arbitrary multiple of $T_{kin}$. If one can reach $T_{max} \sim T_{kin}$, we speak of derivation up to the kinetic timescale in short time. If one can furthermore take $T_{max}$ an arbitrary multiple of $T_{kin}$, we talk about long time derivation. \\

	\noi For results in the case of dimension $d\ge 3$ up to the kinetic timescale and a quadratic system, see \cite{desuzzoni2025waveturbulencesemilinearkleingordon} with a study of a half-wave Klein-Gordon system, i.e. a quadratic system where the second order resonances prevail and the authors get a strong convergence like \eqref{strongerconvergence}. See \cite{quadraticreal} for the derivation on $\R^d$ in the inhomogeneous setting for a quadratic system. For an extended study of the case of cubic NLS for $d\ge3$ where the authors derive the Wave Kinetic Equation (WKE) up to the kinetic timescale strongly, see, in order : \cite{buckmaster2021onsetwaveturbulencedescription}, \cite{Deng_2021}, \cite{Deng_2023}, \cite{deng2022derivationtionchaoshigherorder}, \cite{deng2023derivationwavekineticequation}. For the case of cubic NLS with a lower dispersion fractional Laplacian in dimension 1, see \cite{vassilev2024onedimensionalwavekinetictheory}. For a case where the first order resonances prevail in the cubic NLS case, see \cite{faou2013weaklynonlinearlargebox} where the authors derive a CR (Continuous resonant) equation (see Remark \ref{CRcondition} for details). For the case of periodic data embedded in the whole plane $\R^2$, see \cite{faou2023scatteringrandomphasewave}. Finally, \cite{arXiv:2311.10082} provides the first long time result for the cubic non-linear Schrödinger equation. The work \cite{wu2025rigorousderivationwavekinetic} gives a derivation of the kinetic equation where the dispersive relation is non-algebraic. The papers \cite{dengwaveone} and \cite{dengwavetwo} use tools developed in  wave turbulence theory to show well-posedness and long time bound results. See, for instance, \cite{bourgain1996invariant}, \cite{burq2008random}, \cite{burq2013probabilistic} and \cite{colliander2012almost} for a use of randomness not to extend the time of existence but to show local well-posedness for low regularities. Finally, see \cite{dymovkuksin1}, \cite{dymovkuksin2}, \cite{Dymovkuksinzakharov} and \cite{dymov} for wave turbulence derivation where the randomness lies in the forcing. \\
	
	\subsection{Organization of the paper}
	\noi The remainder of this paper is organized as follows. Section 2 presents the heuristic computation that yields the wave turbulence derivation. Section 3 gives a local well-posedness result for the limit equation \eqref{limitequation}. Section 4 states the two bounds we have to show in order to prove Theorems \ref{theorem1}, \ref{theorem2} and \ref{wavetheorem}. These two bounds are proved in Section 6 and 7. Section 5 presents the proof of the local well-posedness i.e. Theorems \ref{theorem1} and \ref{theorem2}. Section 8 presents a lemma that allows us to go from a sum to an integral. Finally, Section~9 shows the wave turbulence limit result in Theorem \ref{wavetheorem}. Appendix A gives an empirical measure interpretation to wave turbulence. Appendix B presents the proof of Lemma \ref{time of existence}. Appendix C interprets the results of Theorems~~\ref{theorem1} and \ref{theorem2} on a rescaled torus. Appendix D compiles technical results used in the paper. Appendix E presents elements of the proof of Proposition \ref{basiclwp}.
	
	\subsection{Miscellaneous Notations}\label{notationss}
	\begin{itemize}
		\item We define the Japanese brackets : $\jb{x}:=\sqrt{1+4\pi^2|x|^2}$.
		\item We note $A\les B$ if there exists a universal constant $C>0$ (universal in the sense that it does not depend on the relevant parameters) such that $A\le CB$.
		\item We note $A\les_\eps B$ if the constant $C$ depends on $\eps$.
		\item We note $A\asymp B$ if $A=O(B)$ and $B=O(A)$.
		\item We note $l^2_k:=\{f:\Z_L^d \rightarrow \C, \sum_{k\in \Z_L^d}|f(k)|^2 <+\infty\}$.
		\item If $\eta_1 , \eta_2$ are two independent centered normal real Gaussian random variables, we define a centered normal complex Gaussian random variable $g$ by setting $$g:=\frac{1}{\sqrt{2}}(\eta_1+i\eta_2)\, .$$
	\end{itemize}
	
	\section{Overview of the proof of Theorem \ref{wavetheorem}} \label{heuristic}

	\subsection{Formal derivation}
	We hereby present the formal computation that motivates the proof of Theorem \ref{wavetheorem}. We start from equation \eqref{weaknls}, factorize the linear evolution by setting $a:=\Ft_x (e^{it\Dl}u)$ and we look at $a_k$ for $k\in \Z_L^d$ the Fourier coefficients of $a$. They verify the following equation :
	\begin{align} \label{firsta_kequation}
		\dt a_k(t)=-i\left( \frac{\al}{L^{d/2}}\right)\sum_{\substack{k_1,k_2 \in \Z_L^d \\k_1+k_2=k}}a_{k_1}(t)a_{k_2}(t)e^{-2\pi it2k_1\cdot k_2} \, .
	\end{align}
	Before we perform our integration by parts, we notice that when taking the expectation $\E\left[a_k\cj{a_k}\right]$, if we expand $a_k$ through successive integration by parts, the only non-zero terms are the one where the expansion is of the same order for both $a_k$ and $\cj{a_k}$ as otherwise, there is not the same amount of conjugate and non conjugate Gaussian random variables in the expectation. Because of that, we already know that the computation of order 2 expansions against order 0 is zero. Therefore, we are only interested at the computation of order 1 against order 1 at first. \\
	
	\noi We integrate and perform one integration by parts which gives us our main terms, and then a second one which makes our error terms appear. We will use the abusive notation $\frac{1-e^{isk}}{k}=-is$ if $k=0$. We write :
	\begin{align*}
		a_k(t)&=a_k(0)-i\left( \frac{\al}{L^{d/2}}\right)\int_0^t\sum_{\substack{k_1,k_2 \in \Z_L^d \\ k_1+k_2=k}}a_{k_1}(s)a_{k_2}(s)e^{-2\pi is2k_1\cdot k_2}ds\\
	\end{align*}
	\noi We perform an integration by part by using the equation verified by $a_k$ \eqref{firsta_kequation} and choose a primitive that cancels at $s=t$ so that the terms $a_k$ that appear are evaluated at $s=0$ where we know their value. We perform a second integration by part with the same primitive and make two additional terms appear : the first one is a trilinear expression in the family $(g_k)_{k\in \Z^d}$ noted $\mathcal{T}$ and the second is a rest we put in a big $O$. Notice that because of equation~~\eqref{firsta_kequation}, each time we perform an IBP, we get and additional power of $\frac{\al}{L^{d/2}}$. Note that because of the choice of our nonlinearity as $u^2$, $\mathcal{T}$ contains no conjugate of Gaussians. We get :
	\begin{align*}
		a_k(t)=&a_k(0)+\left( \frac{\al}{L^{d/2}}\right)\left[\sum_{\substack{k_1,k_2 \in \Z_L^d \\ k_1+k_2=k}}\frac{e^{-2\pi i s2k_1\cdot k_2}-e^{-2\pi i t2k_1\cdot k_2}}{2\pi 2k_1\cdot k_2}a_{k_1}(s)a_{k_2}(s)\right]_0^t +\mathcal{T}\big((g_k)_{k \in \Z^d}\big)+O\left(\frac{\al^3}{L^{3d/2}}\right) \\
		&=a_k(0)-\left( \frac{\al}{L^{d/2}}\right)\sum_{\substack{k_1,k_2 \in \Z_L^d \\ k_1+k_2=k}}\frac{1-e^{-2\pi i t2k_1\cdot k_2}}{2\pi 2k_1\cdot k_2}a_{k_1}(0)a_{k_2}(0)+\mathcal{T}\big((g_k)_{k \in \Z^d}\big) +O\left(\frac{\al^3}{L^{3d/2}}\right) \, .
	\end{align*}
	We then plug in the values for the initial conditions and get :
	\begin{align*}
		a_k(t)&=\sqrt{n_{in}(k)}g_k(\o)-\left( \frac{\al}{L^{d/2}}\right)\sum_{\substack{k_1,k_2 \in \Z_L^d \\ k_1+k_2=k}}\frac{1-e^{-2\pi i t2k_1\cdot k_2}}{2\pi 2k_1\cdot k_2}\sqrt{n_{in}(k_1)n_{in}(k_2)}g_{k_1}(\o)g_{k_2}(\o) \\
		&+\mathcal{T}\big((g_k)_{k \in \Z^d}\big)+O\left(\frac{\al^3}{L^{3d/2}}\right) \, .
	\end{align*}
	So we get :
	\begin{align*}
		&\E\left[a_k\cj{a_k}\right]  =\E[\Lambda\cj{\Lambda}]+O\left(\frac{\al^3}{L^{3d/2}}\right)
	\end{align*}
	where :
	\begin{align*}
		\Lambda&:=\sqrt{n_{in}(k)}g_k(\o)-\left( \frac{\al}{L^{d/2}}\right)\sum_{\substack{k_1,k_2 \in \Z_L^d \\ k_1+k_2=k}}\frac{1-e^{-2\pi i t2k_1\cdot k_2}}{2\pi 2k_1\cdot k_2}\sqrt{n_{in}(k_1)n_{in}(k_2)}g_{k_1}(\o)g_{k_2}(\o)+\mathcal{T}\big((g_k)_{k \in \Z^d} \, .
	\end{align*}
	Now we use the fact there are no conjugate Gaussians in the term $\mathcal{T}\big((g_k)_{k \in \Z^d}\big)$ to justify that
	$$\E[\mathcal{T}\big((g_k)_{k \in \Z^d}\big)\cj{g_k}]=0 \, .$$ \noi This yields :
	\begin{align*}
		&\E\left[a_k\cj{a_k}\right]=n_{in}(k) +\left( \frac{\al}{L^{d/2}}\right)^2\sum_{\substack{k_1,k_2 \in \Z_L^d \\ k_1+k_2=k}}\sum_{\substack{k_1',k_2' \in \Z_L^d \\ k_1'+k_2'=k}}\frac{1-e^{-2\pi i t2k_1\cdot k_2}}{2\pi 2k_1\cdot k_2}\frac{1-e^{+2\pi i t2k_1'\cdot k_2'}}{2\pi 2 k_1'\cdot k_2'} \\
		& \hspace{5.1cm} \times \sqrt{n_{in}(k_1)n_{in}(k_2)}\sqrt{n_{in}(k_1')n_{in}(k_2')}\E\left[g_{k_1}g_{k_2}\cj{g_{k_1'}g_{k_2'}}\right] +O\left(\frac{\al^3}{L^{3d/2}}\right) \, .
	\end{align*}
	We have either $k_1'=k_1$ and $k_2'=k_2$ or $k_2'=k_1$ and $k_1'=k_2$. We notice that if $k_1=k_2=k_1'=k_2'$ and $k=2k_1$, we have to compute $\E[|g_k|^4]=8$. So we get using the abusive notation $\frac{\sin(kx)}{x}=k$ if $x=0$ :
	\begin{align*}
		\E\left[a_k\cj{a_k}\right] & = n_{in}(k)+2\left( \frac{\al}{L^{d/2}}\right)^2\sum_{\substack{k_1,k_2 \in \Z_L^d \\ k_1+k_2=k}}\left|\frac{\sin(\pi t2k_1\cdot k_2)}{\pi 2k_1\cdot k_2}\right|^2n_{in}(k_1)n_{in}(k_2)\\
		&+\ind(\text{k is even})6\left( \frac{\al}{L^{d/2}}\right)^2\left|\frac{\sin(\pi tk\cdot k/2)}{\pi k\cdot k/2}\right|^2n_{in}(k/2)n_{in}(k/2)+O\left(\frac{\al^3}{L^{3d/2}}\right) \, .
	\end{align*}
	However, the term $$\ind(\text{k is even})6\left( \frac{\al}{L^{d/2}}\right)^2\left|\frac{\sin(\pi tk\cdot k/2)}{\pi k\cdot k/2}\right|^2n_{in}(k/2)n_{in}(k/2)$$ will tend to $0$ as $L\rightarrow +\infty$ when go from the sum to an integral. \\
	
	\noi We ignore the higher order terms and go from the sum to an integral :
	\begin{align*}
		\E\left[a_k\cj{a_k}\right] & \sim n_{in}(k)+2\left( \al\right)^2\int_{\substack{k_1,k_2 \in \R^d \\k_1+k_2=k}}\left|\frac{\sin(\pi t2k_1\cdot k_2)}{\pi 2k_1\cdot k_2}\right|^2n_{in}(k_1)n_{in}(k_2)dk_1 \\
		&\sim n_{in}(k)+2T_{kin}^{-1}\int_{\R^d}\left|\frac{\sin(\pi t2x\cdot (k-x)}{\pi 2x\cdot (k-x)}\right|^2n_{in}(x)n_{in}(k-x)dx \, .
	\end{align*}
	Finally we use the approximation of a Dirac $t \left|\frac{\sin(tx)}{tx}\right|^2 \xrightarrow[t\rightarrow +\infty]{}\delta(x)$ to get for $t\gg 1$ :
	\begin{align*}
		\E\left[a_k\cj{a_k}\right]  &\sim n_{in}(k)+2\frac{t}{T_{kin}}\int_{\R^d}\delta(x\cdot (k-x))n_{in}(x)n_{in}(k-x)dx \, .
	\end{align*}
	\begin{remark}\label{kinetictimeexplanation}
		We notice that the condition $T_{kin} =\al^{-2}$ appears naturally in the computation and comes from the fact that the dynamics are led by the expectation of order 1 terms against order 1 terms. The choice of the normalization factor in the initial condition also impacts the kinetic time. Here, we have chosen the normalization $L^{-d/2}$ as it naturally makes $L^{-d}$ appear in the computation, which we use to go from the Riemann sum to an integral. However, a consequence of this choice is that $\|u_0\|_{L^2}\sim C L^{d/2}$ and $\|u\|_{L^\infty}\sim C$ in \eqref{weaknls}, meaning we have not renormalized the $L^2$ norm of the initial condition. Making such a choice would lead to a $L^{-d}$ renormalization factor and the above computation would yield $T_{kin}=L^d \al^{-2}$.
	\end{remark}
	\begin{remark}\label{CRcondition}
		This derivation is justified for a weak nonlinearity, that is $\al \ll 1$. We justify that this nonlinearity should not be too weak. This comes from the condition $T_{max} \le L^2$ which is necessary to justify the sum to integral Lemma \ref{sumtointegral} below. This implies a constraint on $\al$ as $T_{max}$ is a decreasing function in $\al$ (in the case of Theorem \ref{wavetheorem}, $T_{max}=\al^{-\frac{2}{1+1/d}}L^{-\frac{2\delta}{1+1/d}}$). For a generic torus, we have a laxer condition : $T_{max}\le L^d$ (see \cite{Deng_2021}).
	\end{remark}
	\subsection{Main steps of the convergence analysis}
	We expand our solution as $$a=J_0+J_1+...+J_N+R_{N+1}$$ where $J_n$ corresponds to some of the terms obtained through $n$ iterations of the Duhamel formula on the initial data (or integration by parts). We decompose each $J_n$ in a sum of Feynman trees $J_\mathcal{T}$ with each tree tracing the history of where these iterations were performed. We get a priori bounds on these terms in Proposition \ref{bound1}. We use the structure of the equation solved by $R_{N+1}$ to bound an operator in Proposition \ref{bound2} which gives us a priori control on $R_{N+1}$. This shows the existence of the Fourier modes up to $T_{max}$ and allow us to show local well-posedness. We proceed by computing $\E\big[|\ft{u}(t,k)|^2\big]$ whose first two terms give us the limit dynamic but whose other terms we prove negligible in our regime thanks to the aforementioned bounds. We finally go from a sum to the integral using Lemma \ref{sumtointegral} and then to the Dirac limit. 
	

	\section{Well-posedness of the limit equation}\label{limitequationsection}
	In this section, we show that the limit equation \eqref{limitequation} is locally well-posed in some functional space.
	\begin{proposition} \label{LWP}
		Let $d\ge 4$ and $s\ge \frac{d-4}{2}$.
		Consider the following partial differential equation 
		\begin{equation}
			\begin{cases}
				\dt f(t,x)=C(f(t),f(t))(x)=\int_{\R^d}f(t,y)f(t,x-y)\delta(y\cdot(x-y))dy \\
				f(0,x)=f_0(x)
			\end{cases}
		\end{equation}
		where $(t,x) \in \R^+ \times \R^d$. \\
		Let $\|f\|_{L_s^2}:=\|\jb{x}^s f\|_{L^2}$ and $L_s^2:=\{f\in L^2, \ \|f\|_{L_s^2} <+\infty\}$. Then there exists $C_s>0$ such that for every $f_0\in L_s^2$, there exists a solution to the equation above in $C^0([0,T],L^2_s)$ where $T=C_s\|f_0\|_{L_s^2}^{-1}$. 
	\end{proposition}
	\begin{proof} In the rest of the paper, we will denote by $\delta_x(.)$ the Dirac distribution supported on $x\in \R$. When $x=0$, we will use the notation $\delta$. We use the ideas of \cite{germain2018optimallocalwellposednesstheory}.
		We rewrite the equation using the tempered distribution equality
		\begin{align*}
			\delta_0(p)dp=\frac{1}{2\pi}\int_{\R^d} e^{ip\cdot x}dxdp \ \text{in $\mathcal{S'}(\R^d)$}\, .
		\end{align*}
		We then define the operator
		$$C:(f,g)\in \mathcal{S}(\R^d)^2 \mapsto \Big(x \in \R^d \mapsto \int_{\R^d}f(y)g(x-y)\delta(y\cdot(x-y))dy \in \mathcal{S}'(\R^d) \Big) \, .$$
		\noi For $f,g\in \mathcal{S}(\R^d)$ the following equality in the space of tempered distribution $\mathcal{S}'(\R^d)$ holds :
		\begin{align*}
			C(f,g)(x)&=\int_{\R^d}f(y)g(x-y)\delta(y\cdot(x-y))dy \\
			&=c\int_{\R^d\times \R^d \times  \R \times \R^d}f(y)g(z)e^{i\big[|x|^2-|z|^2- |y|^2\big]\tau}e^{i(x-z-y)\cdot u}dydzd\tau du 
		\end{align*}
		where $c\in \R$ is an irrelevant constant. Now we define :
		\begin{align*}
			Tf(p,\tau):=\int_{\R^d}f(q)e^{-ip\cdot q}e^{-i|q|^2\tau}dq =c e^{i\tau \Dl}\left( \mathcal{F}^{-1}(f)\right)(-p)
		\end{align*}
		where $\Ft$ is the Fourier transform on $\R$ as defined in Section \ref{section4}. \\
		We get :
		\begin{align*}
			C(f,g)(x)=c\int_{\R^d \times \R}Tf(u,\tau)Tg(u,\tau)e^{ix\cdot u}e^{i|x|^2\tau}du d \tau \, .
		\end{align*}
		We bound $C(f,g)$ in $L^2_s$ using Hölder inequality and the fractional Leibniz rule established in Theorem A.8 of \cite{KPVfractionalLeibniz}~:
		\begin{align*}
			\big\|\jb{x}^{s} C(f,g)(x)\big\|_{L^2} & \les \Big\|\jb{x}^s\int_\R \mathcal{F}_u^{-1} \left(Tf(u,\tau)Tg(u,\tau)\right)(x,\tau)e^{i\tau|x|^2}d\tau \Big\|_{L^2} \les \int_{\R}\big\|Tf(.,\tau)Tg(.,\tau) \big\|_{H^s} d\tau \\
			&\les \big\|\jb{D_x}^{s}Tf(.,\tau)\big\|_{L_\tau^2L_x^{\frac{2d}{d-2}}}\big\|Tg(.,\tau)\big\|_{L_\tau^2L_x^{d}} +\big\|\jb{D_x}^{s}Tg(.,\tau)\big\|_{L_\tau^2L_x^{\frac{2d}{d-2}}}\big\|Tf(.,\tau)\big\|_{L_\tau^2L_x^{d}} \, .
		\end{align*}
		Because we have a quadratic non-linearity, we have to use the endpoint Strichartz estimate established in \cite{endpointstrichartz}. This is crucial as for lower order non-linearities, we cannot run the same argument. Combining that with a Sobolev injection, we get :
		\begin{align*}
			\big\|\jb{x}^{s} C(f,g)(x)\big\|_{L^2} & \les \big\|\jb{D_x}^sTf(.,0)\big\|_{L_x^{2}}\Big\|\jb{D_x}^{\frac{d-4}{2}}Tg(.,\tau)\Big\|_{L_\tau^2L_x^{\frac{2d}{d-2}}} \\
			&+ \big\|\jb{D_x}^sTg(.,0)\big\|_{L_x^{2}}\Big\|\jb{D_x}^{\frac{d-4}{2}}Tf(.,\tau)\Big\|_{L_\tau^2L_x^{\frac{2d}{d-2}}}  \\
			& \les \|f\|_{L_s^{2}}\|g\|_{L_\frac{d-4}{2}^{2}} + \|g\|_{L_s^{2}}\|f\|_{L_\frac{d-4}{2}^{2}} \les \|f\|_{L_s^{2}}\|g\|_{L_s^{2}} \ \ \ \textbf{for $s\ge d/2-2$ \, .} 
		\end{align*}
		We thus have proved :
		\begin{equation*}
			\| C(f,g)\|_{L_s^2} \les \|f\|_{L_s^{2}}\|g\|_{L_s^{2}}
		\end{equation*}
		for $f$ a Schwartz function. By denseness one can extend the bilinear form $C$ on $(L_s^2)^2$, such that this inequality holds for $f,g\in L_s^2$. \\
		Now, we can run a classical fixed point argument and get local well-posedness.\\
		\noi We define the functional :
		$$\Phi:f\in C^0\big((0,T),L_s^2\big) \mapsto \Big( t\mapsto f_0+ \int_0^tC(f(\tau),f(\tau))d\tau \Big) \, .$$
		Let $Z:=\{f\in C^0\big((0,T)L_s^2\big), \ \|f\|_{C^0((0,T),L_s^2)}\le 2\|f_0\|_{L_s^2}\}$.
		We want to show that $\Phi$ is a contraction on the Banach space  $Z$ for small times.\\
		We first show boundedness. Let $f\in Z$.
		\begin{align*}
			\|\Phi(f)\|_{C^0\big((0,T)L_s^2\big)}& \le \|f_0\|_{L_s^2}+\|\int_0^tC(f(\tau),f(\tau))d\tau \|_{C^0 \big((0,T),L_s^2\big)} \\
			& \le \|f_0\|_{L_s^2}+T\|C(f(t),f(t)) \|_{C^0 \big((0,T),L_s^2\big)} \\
			&\le  \|f_0\|_{L_s^2}+TC_s\|f \|_{C^0 \big((0,T),L_s^2\big)}^2 \\
			& \le 2\|f_0\|_{L_s^2}
		\end{align*}
		if $T\le (4C_s \|f_0\|_{L_s^2})^{-1}$. \\
		Now we show that $\Phi$ is a contraction. Let $f,g \in Z$.
		\begin{align*}
			\|\Phi(f)-\Phi(g)\|_{C^0 \big((0,T),L_s^2\big)}& \le \|\int_0^tC(f(\tau),f(\tau)-g(\tau))d\tau \|_{C^0 \big((0,T),L_s^2\big)} \\
			& \ \ \ \ + \|\int_0^tC(g(\tau),f(\tau)-g(\tau))d\tau \|_{C^0 \big((0,T),L_s^2\big)} \\
			& \le TC_s\Big(\|f \|_{C^0 \big((0,T),L_s^2\big)}+\|g \|_{C^0 \big((0,T),L_s^2\big)}\Big)\|f-g \|_{C^0 \big((0,T),L_s^2\big)} \\
			&\le 4TC_s \|f_0\|_{L_s^2}\|f-g \|_{C^0 \big((0,T),L_s^2\big)} \\
			&\le \frac{1}{2}\|f-g \|_{C^0 \big((0,T),L_s^2\big)}
		\end{align*}
		if $T\le (8C_s \|f_0\|_{L_s^2})^{-1}$.
		So we can set $T:=(8C_s \|f_0\|_{L_s^2})^{-1}$ and apply Banach's fixed point Theorem to $\Phi$.
	\end{proof}
	\begin{remark}
		We can apply the ideas of the proof of Proposition \ref{LWP} to Lemma 4.1 of \cite{germain2018optimallocalwellposednesstheory} and extend this lemma to $s\ge \frac{d-2}{2}$ (i.e. $s\ge1/2$ for $d=3$) instead of $s>\frac{d-2}{2}$.
	\end{remark}

	\section{Expansion of the solution and statement of two bounds in Bourgain spaces}\label{section4}
	
	\subsection{Expansion of the solution}
	
	\noi We define for $u:\T_L^d\rightarrow \C$ its Fourier coefficients by setting for $k \in \Z_L^d$ :
	\begin{align*}
		\Ft_x(u)(k)=\ft{u}_k=\ft{u}(k):=\frac{1}{L^{d/2}}\int_{\T_L^d}u(x)e^{-i2\pi k\cdot x}dx
	\end{align*} 
	and $u$ can be reconstructed from $\ft{u}(k)$ by the formula
	\begin{align*}
		u(x)=\frac{1}{L^{d/2}}\sum_{k\in\Z_L^d}\ft{u}(k)e^{i2\pi k\cdot x} \, .
	\end{align*}
	\noi We define the Fourier transform on $\R$ or temporal Fourier transform and recall its associated inverse :
	\begin{equation*}
		\Ft_t(f)(\xi)=\Tilde{f}(\xi):=\int_{\R}e^{-2\pi i x \xi}f(x)dx, \ \ \ f(x)=\int_{\R}e^{2\pi i x \xi}\Tilde{f}(\xi)d\xi \, .
	\end{equation*}
	We will denote by $\Ft_{t,x}(u)$ or $\Tilde{\ft{u}}$ the space time Fourier transform of $u$. \\
	
	\noi We define the Bourgain spaces $h^{s,b}$ endowed with the norm :
	\begin{equation*}
		\|u\|_{h^{s,b}}^2:=L^{-d}\sum_{k\in \Z_L^d}\int_\R \jb{k}^{2s}\jb{\tau}^{2b}|\Ft_{t,x}u(\tau,k)|^2d\tau \, .
	\end{equation*}
	
	\noi We will omit the dependence on $L$ of all functional spaces. \\
	\noi We define the restriction Bourgain spaces $h_{T}^{s,b}$ for $T>0$ by the norm :
	\begin{align*}
		\|u\|_{h_{T}^{s,b}} := \inf\{\|f\|_{h^{s,b}}, \ f_{\big|]0,T[}=u_{\big|]0,T[}\} \, .
	\end{align*}
	\noi We first do the change of variables 
	$$u:=u/\al \, .$$ 
	This will allow us to treat all three Theorems \ref{theorem1}, \ref{theorem2} and \ref{wavetheorem} at the same time. This change of variable does not impact the time of existence. After this change of variable, equation \eqref{nls} translates into equation \eqref{weaknls}. We will now study equation \eqref{weaknls} and remember we have to multiply its solution by $\al$ to find the solution of equation \eqref{nls}. \\
	
	\noi We factorize the linear evolution of equation \eqref{weaknls}, do a rescaling in time by setting $s=t/T_{max}$ and look at the Fourier modes. Formally, if $u$ is a solution to equation \eqref{weaknls}, let $k \in \Z_L^d$, we define $a:\R \times \Z_L^d \rightarrow \C$ by : $$a_k(s):=\Ft_x (e^{it\Dl}u)(sT_{max},k) \, .$$ 
	Then $a_k$ verifies the following equation, which we will study now  :
	\begin{align}\label{equationak}
		i\dot{a_k}(s)=\left(\frac{\al T_{max}}{L^{d/2}}\right)\sum_{\substack{k_1,k_2 \in \Z_L^d \\ k_1+k_2=k}}a_{k_1}(s)a_{k_2}(s)e^{-2\pi i s T_{max} 2 k_1 \cdot k_2} \, .
	\end{align}
	We remark that all the degenerate terms i.e. the terms where $k_1\cdot k_2 =0$ matter. This is not the case for cubic Schrödinger where a Wick renormalization is done first to eliminate some degenerate terms. In our case however, the degenerate terms will be negligible as $L\xrightarrow[]{}+\infty$. Furthermore, the resonant equation corresponding to these degenerate terms is integrable and does not blow-up almost surely with the initial data as in Theorems \ref{theorem1} and \ref{theorem2}.   \\
	
	\noi We introduce the equivalent of the Bourgain spaces $h^{s,b}$ on the Fourier side. These spaces are subspaces of $L_t^2l_{k}^2$ and are endowed with the norm :
	\begin{equation*}
		\|u\|_{\Tilde{h}^{s,b}}^2:=L^{-d}\sum_{k\in \Z_L^d}\int_\R \jb{k}^{2s}\jb{\tau}^{2b}|\Ft_{t}u(\tau,k)|^2d\tau \, .
	\end{equation*}
	We notice that :
	\begin{equation*}
		\|\Ft_x(u)\|_{\Tilde{h}^{s,b}}^2=\|u\|_{h^{s,b}}^2 \, .
	\end{equation*}
	\noi We define the restriction Bourgain spaces $\Tilde{h}_{T}^{s,b}$ for $T>0$ by the norm :
	\begin{align*}
		\|u\|_{\Tilde{h}_{T}^{s,b}} := \inf\{\|f\|_{\Tilde{h}^{s,b}}, \ f_{\big|]0,T[}=u_{\big|]0,T[}\} \, .
	\end{align*}
	\noi We define the bilinear form that corresponds to the derivative of the Fourier mode $a_k$ in~~\eqref{equationak} by defining for $u, \ v \in \Tilde{h}^{s,b}$ :
	\begin{align*}
		B(u,v)(s,k):=\left(\frac{-i \al T_{max}}{L^{d/2}}\right)\sum_{\substack{k_1,k_2 \in \Z_L^d \\ k_1+k_2=k}}u_{k_1}(s)v_{k_2}(s)e^{-2\pi i s 2T_{max} k_1 \cdot k_2} \in \Tilde{h}^{s,b} \, .
	\end{align*}
	In order to do our fix point Theorem, we will integrate this bilinear form in time :
	$$\int_0^tB(u,v)(s,k)ds \, .$$
	\noi We have to truncate this integral in time as we want to work on $[0,1]$. We use a smooth cutoff function $\chi \in C_c^\infty(\R)$ such that $\chi \equiv 1$ on $[0,1]$. We will need  a technical lemma to compute the Fourier transform in time of the truncated integral. We use a Lemma proved in \cite{hilbertransform} using the Hilbert transform. We shall give another proof inspired by \cite{Ginibre}.
	\begin{lemma} \label{FT}
		Define $IF(t):=\chi(t)\int_0^t\chi(t')F(t')dt'$ for $F\in \mathcal{S}(\R)$. Then we have :
		\begin{align*}
			\Ft \big(IF\big)(\tau)=\int_\R(I_0+I_1)(\tau,\sigma)\Ft(F)(\sigma)d\sigma \, ,
		\end{align*}
		where, if we denote by $p.v.$ the Cauchy principal value distribution,
		\begin{align*}
			I_0(\tau, \s)&:= - \big(\text{p.v.}(1/i{t'}),\Ft{\chi}({t'} - \s)\Ft{\chi}(\tau) \big)_{t'} \\
			I_1(\tau, \s)&:=  \big(\text{p.v.}(1/i{t'}),\Ft{\chi}({t'} - \s)\Ft{\chi}(\tau-t') \big)_{t'} \, .
		\end{align*}
		And we have for $a\in \N^2$, $A>0$ :
		\begin{align*}
			\left|\partial_{\tau,\sigma}^a I_d(\tau,\sigma) \right| \les_{a,A} \frac{1}{\jb{\tau-j\sigma}^A} \frac{1}{\jb{\sigma}} \ \ \text{for } j=0,1. 
		\end{align*}
	\end{lemma}
	\begin{proof}
		We work in the space of tempered distributions and use that for $u\in \mathcal{S}(\R)$
		\begin{align*}
			\int_0^t u(t')dt'= \int_\R \frac{e^{it\tau}-1}{i\tau}\Ft(u)(\tau)d\tau=\lim_{\eps \rightarrow 0}\int_{|\tau|>\eps}\frac{e^{it\tau}-1}{i\tau}\Ft(u)(\tau)d\tau
		\end{align*}
		which is true as both formulas coincide for $t=0$ and have the same derivative in $t$. \\
		\noi We can thus write for $F\in \mathcal{S}(\R)$, using the notation $\Ft{}$ for the Fourier transform on $\mathcal{S}'(\R)$ :
		\begin{align*}
			\Ft\left( IF\right)(\tau) &=\Bigg( \Ft \left( \chi\right)*\Ft_t \Big(t\mapsto \big(\text{p.v.}(1/i{t'}), ( e^{it{t'}}-1 ) \Ft(\chi F)({t'})\big)_{t'}\Big)\Bigg)(\tau) \ \text{as a convolution in $\mathcal{S}$ and $\mathcal{S'}$}
		\end{align*}
		By definition of the convolution, this is equal to
		\begin{align*}
			&\Bigg(t\mapsto \big(\text{p.v.}(1/i{t'}), ( e^{it{t'}}-1 ) \Ft(\chi F)({t'})\big)_{t'},\Ft_t(\Ft(\chi)(\tau-t))\Bigg)_t \\
			&= \Bigg(t\mapsto \big(\text{p.v.}(1/i{t'}), ( e^{it{t'}}-1 ) \Ft(\chi F)({t'})\big)_{t'},\chi(t)e^{-it\tau}\Bigg)_t
		\end{align*}
		Now we use dominated convergence and Fubini to swap the operators and obtain this is equal to :
		\begin{align*}
			&\Bigg(\text{p.v.}(1/i{t'}),\Ft(\chi F)({t'})\big((e^{it{t'}}-1)e^{-it\tau},\chi(t)\big)_t \Bigg)_{t'}  \\
			&=\Bigg(\text{p.v.}(1/i{t'}),\int_{\s \in \R}\Ft{\chi}({t'} - \s)\Ft{F}(\s)d\s\big(\Ft{\chi}(\tau-{t'})-\Ft{\chi}(\tau)\big)\Bigg)_{t'}  
		\end{align*}
		Finally, we use Fubini and get :
		\begin{align*}
			\Ft\left( IF\right)(\tau) &=\int_{\s \in \R}\Ft{F}(\s)\Bigg(\text{p.v.}(1/i{t'}),\Ft{\chi}({t'} - \s)\big(\Ft{\chi}(\tau-{t'})-\Ft{\chi}(\tau)\big) \Bigg)_{t'} d\s \\
			&= \int_{\s \in \R}\Ft(F)(\s)\big(I_1(\tau,\s)+I_0(\tau,\s)\big)d\s \, ,
		\end{align*}
		where
		\begin{align*}
			I_0(\tau, \s)&:= - \big(\text{p.v.}(1/i{t'}),\Ft{\chi}({t'} - \s)\Ft{\chi}(\tau) \big)_{t'} \\
			I_1(\tau, \s)&:=  \big(\text{p.v.}(1/i{t'}),\Ft{\chi}({t'} - \s)\Ft{\chi}(\tau-t') \big)_{t'}  \, .
		\end{align*}
		\noi A Using the fact that $\chi$ is a Schwartz function, we get the desired result.
		Indeed, let $\eps>0$ small enough. Then :
		\begin{align*}
			\Big|\int_{|t|>\eps}\frac{1}{i t} \Ft \chi(\tau)\Ft \chi(t-\s)dt\Big| & \le \Big|\int_{\substack{|t|>\eps \\ |t|\le \s/2}}\frac{1}{i t} \Ft \chi(\tau)\Ft \chi(t-\s)dt\Big| + \Big|\int_{\substack{|t|>\eps \\ |t|> \s/2}}\frac{1}{i t} \Ft \chi(\tau)\Ft \chi(t-\s)dt\Big| \\
			& \le C \jb{\tau}^{-A}\jb{\s}^{-1}+C\jb{\tau}^{-A}\jb{\s}^{-1} \, .
		\end{align*}
		Hence
		\begin{align*}
			I_0(\tau,\s) \les_{A} \frac{1}{\jb{\tau}^A} \frac{1}{\jb{\sigma}} \, .
		\end{align*}
		Similarly, we get the result for $I_1$ and their derivatives. This concludes the proof of Lemma~~\ref{FT}.
	\end{proof}
	\begin{lemma}\label{FT2}
		Let $I_1$ as defined in Lemma \ref{FT}. Define the operator $\mathcal{I}_1$ by :
		\begin{equation} \label{newoperator}
			\Ft_{t}(\mathcal{I}_1F)(\tau) :=\int_\R I_1(\tau,\sigma)\Ft_{t}(F)(\s)d\s \, .
		\end{equation}
		then for $b>1/2$, $$\|IF\|_{h^{b}}\les\|\mathcal{I}_1F\|_{h^{b}} \, . $$
	\end{lemma}
	\begin{proof}
		We justify this by noticing in Lemma \ref{FT} that $$IF(t)=\mathcal{I}_1 F(t)-\Ft^{-1}_{\tau}\Big(\int_{\s \in \R} \big(\text{p.v.}(1/i{t'}),\Ft{\chi}({t'} - \s)\Ft{\chi}(\tau) \big)_{t'}\Ft(F)(\s)d\s\Big)$$
		and we have :
		\begin{align*}
			\int_{\s \in \R} \big(\text{p.v.}(1/i{t'}),\Ft{\chi}({t'} - \s)\Ft{\chi}(\tau) \big)_{t'}\Ft(F)(\s)d\s & = \Ft{\chi}(\tau)\int_{\s \in \R} \big(\text{p.v.}(1/i{t'}),\Ft{\chi}({t'} - \s) \big)_{t'}\Ft(F)(\s)d\s \,.
		\end{align*}
		We note $$M:=\int_{\s \in \R} \big(\text{p.v.}(1/i{t'}),\Ft{\chi}({t'} - \s) \big)_{t'}\Ft(F)(\s)d\s \, .$$
		So we end up with formula :
		$$IF(t)=\mathcal{I}_1 F(t)-\chi(t)M\, .$$ We evaluate at 0 : $$IF(0)=0=\mathcal{I}_1 F(0)-M $$ so $$IF(t)=\mathcal{I}_1 F(t)+\chi(t)\mathcal{I}_1F(0)$$ and, as  $b>1/2$, by Sobolev injection, $$\|IF\|_{h^{b}}\les\|\mathcal{I}_1F\|_{h^{b}} \, .$$ This concludes the proof of Lemma \ref{FT2}.
	\end{proof}

	\noi We will use the operator $I$ instead of the integral operator to write the Duhamel formula, which is equivalent as $\chi$ is a cut-off function and we work for time $0\le s:=t/T_{max} \le 1$. \\

	\noi We recall we wish to iterate the Duhamel formula for the  Fourier mode $a_k$ of the solution $a$. This leads us to defining the following objects.
	We define recursively $$J_n:\R \times \Z_L^d \longrightarrow \C$$  by :
	$$J_{0}(t,k):=\chi(t)a_{in}(k)$$
	where for $k\in \Z_L^d$ $$a_{in}(k)~:=~\sqrt{n_{in}(k)}g_{Lk}(w)$$
	and
	$$J_n(t)~:=~\sum_{n_1+n_2=n-1}IB(J_{n_1},J_{n_2})(t) \, .$$
	We set :
	$$J_{\le N}:=\sum_{n\le N}J_n \, .$$ 
	We have that $J_n$ corresponds to terms in the $n^{th}$-iterate of the Duhamel formula which are $(n+1)-$linear in the initial condition $a_{in}$. $J_{\le N}$ corresponds to the totality of the $N^{\text{th}}$ Duhammel iterate.
	Finally if $a$ is the solution to \eqref{equationak}, we want to set formally :
	$$R_{N+1}=a-J_N \, .$$
	The rest $R_{N+1}$ formally solves the equation :
	\begin{equation}\label{restequation}
		R_{N+1}=J_{\sim N}+\mathcal{L}(R_{N+1})+Q(R_{N+1})
	\end{equation}
	where :
	\begin{align*}
		J_{\sim N}&:=\sum_{\substack{n_1,n_2 \le N \\ n_1+n_2\ge N}}IB(J_{n_1},J_{n_2}) \\
		\mathcal{L}(R_{N+1}) & := 2\sum_{n_1\le N}IB(J_{n_1},R_{N+1}) \\
		Q(R_{N+1}) & := Q(R_{N+1},R_{N+1})=IB(R_{N+1},R_{N+1}) \, .
	\end{align*}
	\noi  We will proceed in the reverse order i.e. first prove by a fix point Theorem that there exists a solution $R_{N+1}$ to equation \eqref{restequation} and then set $a:=J_{\le N}+R_{N+1}$ which will solve equation~~\eqref{equationak} in a distributional sense (and a Duhammel sense). Indeed, we have : 
	\begin{align*}
		\dt a &=\sum_{n=0}^N \dt J_n+\dt R_{N+1} \\
		&= \sum_{\substack{0\le n_1,n_2\le N \\ n_1+n_2 \le N-1}} B(J_{n_1},J_{n_2}) + \sum_{\substack{0\le n_1,n_2\le N \\ n_1+n_2 \ge N}} B(J_{n_1},J_{n_2}) +2 \sum_{n\le N}B(J_n,R_{N+1}) + B(R_{N+1},R_{N+1}) \\
		&=B(\sum_{n=0}^N  J_n+ R_{N+1},\sum_{n=0}^N  J_n+ R_{N+1}) \\
		&= B(a,a) \, .
	\end{align*}
	
	\noi We will now decompose each term $J_n$ in trees representing the history of the $n$ iterations in the Duhamel formula. We define :
	\begin{definition}
		Let $\mathcal{T}$ be a binary tree. We will denote by $\mathcal{L}$ its set of leaves and $\mathcal{N}$ its set of branching nodes. We will also use the notation $l\in \mathcal{L}$ for a leaf, $n\in \mathcal{N}$ for a branching node and $r$ for the root of the tree. We will call scale of the tree its number of branching nodes : $s(\mathcal{T})=|\mathcal{N}|$. We have the formula $|\mathcal{L}|=s(\mathcal{T})+1$. \\
		\noi We will decorate binary trees by associating to each node a wave number $k\in \Z_L^d$ with the following constraints : if a node has wave number $k$ and two children with wave number $k_1$ and $k_2$, then one must have $k=k_1+k_2$. This implies that the decoration of a tree is entirely determined by the decoration of its leaves. These wave numbers will correspond to Fourier modes and the constraint encodes the convolution. We will say $(k_n)_{n\in \mathcal{N}}$ is admissible if it verifies the aforementioned conditions.
	\end{definition}
	
	\begin{remark}
		Later in the paper, we will add more conditions on trees in addition to the admissibility condition. For the clarity of the exposition, we will sometimes omit the notation of these further constraints and simply write "admissible".
	\end{remark}
	\noi In the following, we will use $\mathcal{T}$ for a binary tree, $n$ for its scale (not to confuse with $n\in \mathcal{N}$ its branching nodes), and the same notations as above for its leaves and branching nodes. \\
	\noi Now we define the objects that encode the history of iterations by induction.

	\begin{definition} \label{definitiontrees} Let $\cdot$ be the trivial tree composed of a single root. We set $$J_{\cdot}(t):=\chi(t)a_{in} \, .$$ 
		
		\noi We define for $\mathcal{T}$ a binary tree formed by  two binary trees $\mathcal T_1$ and $\mathcal{T}_2$ attached to the root :
		\begin{align*}
			J_{\mathcal{T}}(t):=IB(J_{\mathcal{T}_1},J_{\mathcal{T}_2})(t) \,.
		\end{align*}
		
	\end{definition}
	
	\begin{definition} \label{definitionomega}
		Let $\mathcal{T}$ be a decorated tree.  We set for $(j_n)_{n\in \mathcal{L}\cup \mathcal{N}} \in \{0,1\}^{\mathcal{T}}$ : \\
		$q_n:=0$ for $n\in \mathcal{L}$, $q_n~~:=~~j_{n_1}q_{n_1}+~~j_{n_2}q_{n_2}+\O_n$ otherwise if $n\in \mathcal{N}$ has $n_1\in \mathcal{T}$ decorated with $k_1$ and $n_2\in \mathcal{T}$ decorated with $k_2$ as children and where $\O_n:=-2k_{n_1}\cdot k_{n_2}=-|k_n|^2+|k_{n_2}|^2+|k_{n_1}|^2$ is the dispersion relation.
	\end{definition}
	
	\begin{lemma} \label{lemmaK} Let $n \in \N$ and $\mathcal{T}$ a binary tree. \\
		
		\noi Then :
		\begin{align*}
			J_n=\sum_{s(\mathcal{T})=n}J_{\mathcal{T}}
		\end{align*}
		and we have the following formula for a certain function $K_\mathcal{T}$ :
		\begin{align*}
			\Ft_{t}(J_{\mathcal{T}})(\tau,k)=\left(\frac{-i\al T_{max}}{L^{d/2}}\right)^n \sum_{k_n \text{ admissible decorations}, \ n\in \mathcal{T}}K_{\mathcal{T}}(\tau,k_n:n\in \mathcal{T})\prod_{l\in \mathcal{L}}\sqrt{n_{in}(k_l)}g_{Lk_l}(\o)
		\end{align*}
		with the following estimate for all $a\in \N$, $A>0$ :
		\begin{align} \label{estimateonK}
			\left|\partial_{\tau}^a K_\mathcal{T}(\tau,k_n \ : \ n\in\mathcal{T}) \right| \les_{a,A} \sum_{(j_n)_{  n\in\mathcal{T}}}\jb{\tau-T_{max}j_rq_r}^{-A}\prod_{n\in \mathcal{N}}\jb{T_{max}q_n}^{-1} \, .
		\end{align}
		The sum on $(j_n)$ comes from the two possible choices at each iteration $I_0$ and $I_1$ in Lemma~\ref{FT}.
	\end{lemma}
	\begin{proof} The first formula comes from an immediate induction, the branching nodes on the tree corresponding to which term we choose to iterate Duhamel on. \\
		
		\noi We will prove both the second formula and bound by induction. \\
		\noi For the trivial tree, we have :
		\begin{align*}
			|J_{\cdot}(t)|=|\chi(t)a_{in}| \, .
		\end{align*}
		So $K_\cdot(\tau) =\Ft_t(\chi)(\tau)$ satisfies the equality and the bound as it is a Schwartz function. \\
		Now we suppose that the result is true for all trees of scale $\le n-1$ and we fix a tree $\mathcal{T}$ of scale $n$, formed by $\mathcal{T}_1$ and $\mathcal{T}_2$. Now we use Lemma \ref{FT} to write :
		\begin{align*}
			\Ft_{t}(J_{\mathcal{T}})(\tau,k)=&\Ft_{t}(IB(J_{\mathcal{T}_1},J_{\mathcal{T}_2}))(\tau,k)  \\
			&=\int_\R(I_0+I_1)(\tau,\sigma)\Ft_{t}(B(J_{\mathcal{T}_1},J_{\mathcal{T}_2}))(\sigma,k)d\sigma \\
			&=\left(\frac{-i \al T_{max}}{L^{d/2}}\right)\int_\R(I_0+I_1)(\tau,\sigma) \\
			& \times \sum_{\substack{k_1,k_2 \in \Z_L^d \\ k_1+k_2=k}}\int_{\R^2}\Ft_{t}(J_{\mathcal{T}_1})(\tau_1,k_1)\Ft_{t}(J_{\mathcal{T}_2})(\tau_2,k_2) \delta(\s -(\tau_1+\tau_2-2T_{max}k_1\cdot k_2))d\tau_1 d\tau_2 d\sigma \, .
		\end{align*}
		Using induction on the formula, this yields after identification :
		\begin{align*}
			K_{\mathcal{T}}(\tau,k_n:n\in \mathcal{T}) =\int_{\R^2}(I_0+I_1)(\tau,\tau_1+\tau_2-2T_{max}k_1\cdot k_2)K_{\mathcal{T}_1}(\tau_1,k_n:n\in \mathcal{T}_1)K_{\mathcal{T}_2}(\tau_2,k_n:n\in \mathcal{T}_2)d \tau_1 d \tau_2 \, .
		\end{align*}
		This proves the formula for $\mathcal{T}$. \\
		\noi Now let us denote by $r_1$ and $r_2$ the roots of $\mathcal{T}_1$ and $\mathcal{T}_2$ respectively. We get using the induction hypothesis on the bounds :
		\begin{multline*}
			\left|\partial_{\tau}^a K_\mathcal{T}(\tau,k_n:n\in \mathcal{T}) \right| \les_{a,A} \\ \sum_{j_r\in \{0,1\}}\sum_{(j_n)_{n\in \mathcal{T}_1}}\sum_{(j_n)_{n\in \mathcal{T}_2}} \int_{\R^2}\left[\frac{1}{\jb{\tau-j_r(\tau_1+\tau_2-2T_{max}k_1\cdot k_2)}^{A_0}}\frac{1}{\jb{\tau_1+\tau_2-2T_{max}k_1\cdot k_2}}\right]\\
			\times \jb{\tau_1-T_{max}j_{r_1}q_{r_1}}^{-A_1}\jb{\tau_2-T_{max}j_{r_2}q_{r_2}}^{-A_2}\prod_{\substack{n\in \mathcal{N} \\ n \neq r}}\jb{T_{max}q_n}^{-1}d \tau_1 d\tau_2 \, .
		\end{multline*}
		We merge the three sums into one as the tree $\mathcal{T}$ is made of the two trees $\mathcal{T}_1$ and $\mathcal{T}_2$ attached to the root $r$. We also translate $\tau_1$ by $T_{max}j_{r_1}q_{r_1}$ and $\tau_2$ by $T_{max}j_{r_2}q_{r_2}$. We obtain :
		\begin{align*}
			\big|\partial_{\tau}^a K_\mathcal{T}&(\tau,k_n:n\in \mathcal{T}) \big|\les_{a,A} \sum_{(j_n)_{n\in \mathcal{T}}}\int_{\R^2}\frac{1}{\jb{\tau-j_r(\tau_1+T_{max}j_{r_1}q_{r_1}+\tau_2+T_{max}j_{r_2}q_{r_2}-2T_{max}k_1\cdot k_2)}^{A_0}}\\
			&\times \frac{1}{\jb{\tau_1+T_{max}j_{r_1}q_{r_1}+\tau_2+T_{max}j_{r_2}q_{r_2}-2T_{max}k_1\cdot k_2}} \times \jb{\tau_1}^{-A_1}\jb{\tau_2}^{-A_2}\prod_{\substack{n\in \mathcal{N} \\ n \neq r}}\jb{T_{max}q_n}^{-1}d \tau_1 d\tau_2 
		\end{align*}
		where $A_0$, $A_1$ and $A_2$ are functions of $A$ large enough and determined later. \\
		Let us note $\al_1 :=-T_{max}j_{r_1}+T_{max}k_1\cdot k_2$, $\al_2: =-T_{max}j_{r_2}+T_{max}k_1\cdot k_2$.
		Let us first consider the part of the integral where $|\al_1+\al_2| \ge 2|\tau_1+\tau_2|$. Then we can bound the integrand up to multiplicative constant by :
		\begin{multline*}
			\frac{1}{\jb{\tau-j_r(T_{max}j_{r_1}q_{r_1}+T_{max}j_{r_2}q_{r_2}-2T_{max}k_1\cdot k_2)}^{A_0}}
			\frac{1}{\jb{T_{max}j_{r_1}q_{r_1}+T_{max}j_{r_2}q_{r_2}-2T_{max}k_1\cdot k_2}} \\
			\times \jb{\tau_1}^{-A_1}\jb{\tau_2}^{-A_2}\prod_{\substack{n\in \mathcal{N} \\ n \neq r}}\jb{T_{max}q_n}^{-1} d\tau_1 d\tau_2 \, .\end{multline*}
		Hence, the result by integrating in $\tau_1$ and $\tau_2$. \\
		\noi Now consider the case where  $|\al_1+\al_2| \le 2|\tau_1+\tau_2|$.\\
		\noi If $j_r=0$, the result is obvious by bounding the integrand by :
		\begin{align*}
			\frac{1}{\jb{\tau}^{A_0}}
			\frac{1}{\jb{0}} \jb{\tau_1}^{-A_1}\jb{\tau_2}^{-A_2}\prod_{\substack{n\in \mathcal{N} \\ n \neq r}}\jb{T_{max}q_n}^{-1}d \tau_1 d\tau_2 
		\end{align*}
		using that $\jb{x+y} \les \jb{x}\jb{y}$ and integrating in $\tau_1$ and $\tau_2$.  \\
		\noi If $j_r=1$, we distinguish two sub-cases. If $|\tau+\al_1+\al_2| \ge 2|\tau_1+\tau_2|$, we bound the integrand up to a multiplicative constant by
		\begin{multline*}
			\frac{1}{\jb{\tau-j_r(T_{max}j_{r_1}q_{r_1}+T_{max}j_{r_2}q_{r_2}-2T_{max}k_1\cdot k_2)}^{A_0}}
			\frac{1}{\jb{0}} 
			\times \jb{\tau_1}^{-A_1}\jb{\tau_2}^{-A_2}\prod_{\substack{n\in \mathcal{N} \\ n \neq r}}\jb{T_{max}q_n}^{-1}d \tau_1 d\tau_2 
		\end{multline*}
		and we integrate in $\tau_1$ and $\tau_2$ using that $\jb{x+y}\les \jb{x}\jb{y}$.\\
		If $|\tau+\al_1+\al_2| \le 2|\tau_1+\tau_2|$, we bound the integrand by
		\begin{equation*}
			\frac{1}{\jb{0}^{A_0}}
			\frac{1}{\jb{0}} \jb{\tau_1}^{-A_1}\jb{\tau_2}^{-A_2}\prod_{\substack{n\in \mathcal{N} \\ n \neq r}}\jb{T_{max}q_n}^{-1}d \tau_1 d\tau_2 
		\end{equation*}
		and we get the result by integrating in $\tau_1$, $\tau_2$ and using that $\jb{x+y}\les \jb{x} \jb{y}$.
		We get our result choosing $A_0$, $A_1$ and $A_2$ large enough.
	\end{proof}
	\noi We now define :
	\begin{equation}\label{rhodefinition}
		\rho := L^{-\delta} \, .
	\end{equation}
	\noi Observe that :
	\begin{equation*}
		\rho = \begin{cases}
			\al T_{max}^{1} \ \ \ \text{in the case of Theorem \ref{theorem1}}\\
			\al T_{max}^{1/2(1+1/d)} \ \ \ \text{in the case of Theorem \ref{theorem2}}
		\end{cases}
	\end{equation*}
	and 
	\begin{align*}
		\al T_{max}^{1/2} \le \rho  \, .
	\end{align*} 
	\noi Our goal will be to bound the series of the $J_n$ by a geometric series $\rho^n$ for some norm. \\
	
	\subsection{Dependence of parameters and negligible sets}
	\noi  We set our parameters for the rest of the paper\label{parameter}.
	We fix $\delta >0$, $\gamma>\delta$ (for Theorem~\ref{theorem1} only). The parameter $b>1/2$ will be chosen close enough to $1/2$. We fix $s>d/2$. The parameter $N$ will be chosen large enough depending on $\delta$. The parameter $\theta $ will be chosen small enough depending on $N$, $s$, $b-1/2$, $\delta$ and $\gamma$. The constant $C$ will vary from line to line but will be independent from previously introduced parameters. \\
	
	\noi We will call a property \textbf{L-certain} if it holds on an event of probability $\ge 1-Ke^{-cL^{\theta}}$ where $c,K$ depend on $\theta$ and $\gamma$. We will call an event \textbf{L-negligible} if it occurs with probability $\le Ke^{-cL^{\theta}}$ where $c,K$ depend on $\theta$ and $\gamma$. \\
	
	\subsection{Statement of the two main bounds}
	We recall we defined trees and decorated trees in Definitions \ref{definitiontrees} and \ref{definitionomega}. \\
	
	\noi We now present the main propositions that we want to show in the settings of Theorems \ref{theorem1}, \ref{theorem2} and \ref{wavetheorem} :
	\begin{proposition}\label{bound1} For any tree $\mathcal{T}$ of scale $0\le s(\mathcal{T})=n\le 2N$, we have L-certainly for $L$ large enough :
		\begin{align*}
			\sup_{k}\jb{k}^{2s}\left|\left|(J_\mathcal{T})_k\right|\right|_{h^b}  \le L^{\theta+C(b-1/2)}\rho^{n} \, .
		\end{align*}
	\end{proposition}
	\begin{proof}
		See Section \ref{proofbound1}.
	\end{proof}
	
	\begin{proposition} \label{bound2}
		For $L$ large enough, we have L-certainly, for all tree $\mathcal{T}$ such that $|\mathcal{T}|~~=~~2n+1$ with $0\le n\le N$ that the linear operator
		\begin{align*}
			\mathcal{P} : v \in \Tilde{h}^{s,b} \mapsto IB(J_\mathcal{T},v) \in \Tilde{h}^{s,b} 
		\end{align*}
		satisfies :
		\begin{align*}
			\|\mathcal{P}\|_{\Tilde{h}^{s,b}\xrightarrow[]{}\Tilde{h}^{s,b}} \le L^{\theta} \rho^{n+1/2} \, .
		\end{align*}
	\end{proposition}
	\begin{proof}
		See Section \ref{proofbound2}.
	\end{proof}
	\begin{remark}
		The dependence on $T_{max}$ in Propositions \ref{bound1} and \ref{bound2} is hidden within the operator $B$ used to define the stochastic objects $J_\mathcal{T}$.
	\end{remark}
	\begin{remark}
		Proposition \ref{bound2} is useful to show the contraction in Section \ref{proofofLWP} below thanks to the fact that $\rho$ is elevated to a power strictly larger than $n$.
	\end{remark}
	\section{Proof of local well-posedness} \label{proofofLWP}
	
	We assume Propositions \ref{bound1} and \ref{bound2} are true. 
	We present the proof of the local well-posedness i.e. Theorems \ref{theorem1} and \ref{theorem2} at the same time as the proof is the same within both contexts. \\
	
	\noi We recall we have written our solution $a:=\Ft_x(e^{it\Dl}u)$ where $u$ solution to \eqref{weaknls} and set formally $$a(t,k)~=~J_0(t,k)~+~J_1(t,k)+...+J_N(t,k)+R_{N+1}(t,k)$$ where $R_{N+1}$ solves equation \eqref{restequation}. With this transformation, solving equation \eqref{weaknls} is equivalent to solving the rest equation \eqref{restequation} and having some control on the iterates $(J_n)_{0\le n \le N}$.
	We see equation \eqref{restequation} as a fixed point problem. We set :
	\begin{align*}
		\Phi:v \in \Tilde{h}_{T_{max}}^{s,b} \mapsto J_{\sim N}+\mathcal{L}(v)+\mathcal{Q}(v) \in \Tilde{h}_{T_{max}}^{s,b} \, .
	\end{align*}
	
	\noi We now present the proof that $\Phi$ is a contraction L-certainly. From now on, we will work on a L-certain set $E_L$ on which the Propositions \ref{bound1} and \ref{bound2} hold. \\
	
	\noi Using Proposition \ref{bound1}, we prove that $L-certainly$, $\forall \, 1 \le n \le 3N$ :
	\begin{align*}
		\sup_{k}\jb{k}^{2s}\left|\left|(J_n)_k\right|\right|_{h^b} & \le (n+1)!L^{\theta+C(b-1/2)}\rho^{n} \\
		&\le L^{\theta+C(b-1/2)}\rho^{n}
	\end{align*}
	as there are at most $C_{2n+1}$ trees of scale $n$ where $C_n$ is the n-th Catalan number for $n\in \N$. We have $C_n=\frac{1}{2n+1}\binom{2n+1}{n}\sim\frac{4^n}{n^{3/2 }\pi}$. In our setting, these asymptotics do not matter as $C_n$ is a constant which only depends on $N$ which will be fixed, hence we can ignore it and bound it by $L^{\theta}$ for $L$ large enough. \\
	
	\noi Now we want to show that 
	\begin{align*}
		\Phi :v\in Z \mapsto J_{\sim N} + \mathcal{L}(v)+\mathcal{Q}(v) \in Z
	\end{align*}
	is a contraction on $Z:=\left\{v, \|v\|_{\Tilde{h}_{T_{max}}^{s,b}} \le \rho^N\right\}$ for $N$ large enough. \\
	To show this rigorously, we should work with a minimizing sequence $v_n$ realizing the infimum of the Bourgain restriction norm, apply the reasoning below, and go to the limit. For simplicity of notation, we work directly with a function $v$ such that $\|v\|_{\Tilde{h}_{}^{s,b}} \le \rho^N$ and do computations in $\Tilde{h}^{s,b}$.
	\\  We first show boundedness. We use Proposition \ref{bound1} to get bounds on the trees which are of scale~$\ge n+1$ :
	\begin{align*}
		\|J_{\sim N}\|_{\Tilde{h}^{s,b}}^2 &\le L^{-d}\sum_{k\in\Z_L^d}\jb{k}^{2s}\|J_{\sim N}\|_{h^{b}}^2 \les \left(L^{\theta+C(b-1/2)}\rho^{N+1}\right)^2L^{-d}\sum_{k\in\Z_L^d}\jb{k}^{-2s} \\
		& \le \rho^{2N}\rho^2\left(L^{\theta+C(b-1/2)}\right)^2L^{-d}(L^d+L^{2s}\sum_{\substack{k\in\Z^d-\{0\}\\ |k|\ge L}}\jb{k}^{-2s} ) \\
		& \le \rho^{2N}L^{-2\delta}\left(L^{\theta+C(b-1/2)}\right)^2L^{-d}(L^d+L^{2s}L^{-2s+d} ) \\
		& \le \rho^{2N}L^{-2\delta}\left(L^{\theta+C(b-1/2)}\right)^2 \le  \rho^{2N}/9 \, .
	\end{align*}
	for $\theta$ and $b-1/2$ small enough, $L$ large enough. \\
	Now using the operator bound in Proposition~\ref{bound2} we get by summing on all trees of scale $\le N$ :
	\begin{align*}
		\|\mathcal{L}(v)\|_{\tilde{h}^{s,b}} \le \sum_{n=0}^{N} \frac{n!}{2^n} L^\theta \rho^{n+1/2}\|v\|_{\Tilde{h}^{s,b}} \le L^\theta \rho^{1/2}\|v\|_{\tilde{h}^{s,b}} \le \rho^N/3
	\end{align*}
	for $L$ and $N$ large enough. \\
	
	\noi Finally, using $b < 1$ and Proposition \ref{duhamel} below (see \cite{Ginibre} for reference) with $T=1$, we get for some $0<b'<1/2$ such that $b+b'\le 1$ :
	\begin{align*}
		\|Q(v)\|_{\Tilde{h}^{s,b}}&= \|IB(v,v)\|_{\Tilde{h}^{s,b}} \\
		&\les \|\chi B(v,v)\|_{\Tilde{h}^{s,-b'}} \\
		& \le \|\chi B(v,v)\|_{\Tilde{h}^{s,0}}
	\end{align*}
	%
	We now denote :
	$$\|u\|_{l_k^{2,s}}:=\|\jb{k}^su\|_{l_k^2} \, .$$
	We have for $u,v \in \Tilde{h}^{s,b}$ :
	\begin{align*}
		\|\Ft_{t}(\chi(t) u*_k v)\|_{L_\tau^2 l_k^{2,s}} &= \|\chi(t)( u*_k v)\|_{L_t^2 l_k^{2,s}} \, .
	\end{align*}
	We use the fact that $\|\chi\|_{L^\infty} \le 1$ :
	\begin{align*}
		\|\Ft_{t}(\chi(t) u*_k v)\|_{L_\tau^2 l_k^{2,s}}&\le \| u*_k v\|_{L_t^2 l_k^{2,s}} \\
		&=\|\jb{k}^su*_k v\|_{L_t^2 l_k^{2}} \, .
	\end{align*}
	\noi We use Young's inequality and get :
	\begin{align*}
		\|\Ft_{t}(\chi(t)u*_k v)\|_{L_\tau^2 l_k^{2,s}} & \le C \|u \|_{L_t^2 l_k^{2,s}}\|v\|_{L_t^\infty l_k^{1}} +\|v \|_{L_t^2 l_k^{2,s}}\|u\|_{L_t^\infty l_k^{1}} \\
		& = C\|\Ft_{t}(u) \|_{L_\tau^2 l_k^{2,s}}\|v\|_{L_t^\infty l_k^{1}}+C\|\Ft_{t}(v) \|_{L_\tau^2 l_k^{2,s}}\|u\|_{L_t^\infty l_k^{1}} \, .
	\end{align*}
	Finally, using $b> 1/2$, $s>d/2$ and Cauchy-Schwarz inequality :
	\begin{align*}
		\|\Ft_{t}(u*_k v)\|_{L_\tau^2 l_k^{2,s}}&\le C\|u\|_{\Tilde{h}^{s,b}}\|v\|_{\Tilde{h}^{s,b}}L^{3d/2} \, .
	\end{align*}
	Applying this to our computation and using that $T_{max}\le L^{C(d,\gamma)}$, we get :
	\begin{align*}
		\|\chi B(v,v)\|_{\tilde{h}^{s,0}} \le L^{d} \left(\frac{\al T_{max}}{L^{d/2}} \right)\rho^{2N} \le \rho L^{d/2+C(d,\gamma)} \rho^{2N} \le L^{-N\delta+C(d,\gamma)}\rho^N \le \rho^N/3
	\end{align*}
	for $N$ large enough. \\
	Hence the boundedness. \\
	
	\noi For the contraction, let $v,w \in Z$  :
	\begin{align*}
		\|\Phi(v-w)\|_{\Tilde{h}^{s,b}}&\le\|\mathcal{L}(v-w)\|_{\Tilde{h}^{s,b}}+\|\mathcal{Q}(v,v-w)\|_{\Tilde{h}^{s,b}}+\|\mathcal{Q}(w,v-w)\|_{\Tilde{h}^{s,b}} \\
		& \le L^\theta \rho^{1/2}\|v-w\|_{\Tilde{h}^{s,b}}+L^{C(d)}\al T_{max}\rho^N(\|v-w\|_{\Tilde{h}^{s,b}}) \\
		& \le \frac{1}{2}\|v-w\|_{\Tilde{h}^{s,b}}
	\end{align*}
	for $N$ and $L$ large enough. \\
	Hence the contraction and thus we can apply the fixed point theorem on $Z$ which is a complete space and get the existence of $v\in Z$ solving \eqref{restequation}.\\
	
	\noi Now, one can use this result to finish the proof of local well-posedness in Theorems \ref{theorem1} and \ref{theorem2}. First we apply the contraction result and find a unique fix point we will call $R_{N+1}$. Then we set :
	\begin{equation*}
		a:=\sum_{n=0}^N J_n +R_{N+1}
	\end{equation*}
	The sequence $a_k(t)$ belongs to the $\Tilde{h}^{s,b}$ space as showed by the computation :
	\begin{align*}
		\|a\|_{\Tilde{h}^{s,b}}\le \sum_{n=0}^N \|J_n\|_{\tilde{h}^{s,b}}+\|R_{N+1}\|_{\Tilde{h}^{s,b}} \le L^{\theta+C(b-1/2)}\sum_{n=0}^N\rho^n+\rho^N =L^{\theta+C(b-1/2)}\frac{1-\rho^{N+1}}{1-\rho}+\rho^N < + \infty \, .
	\end{align*}
	Function $a$ is thus a mild solution to equation \eqref{equationak}. Indeed :
	\begin{align*}
		a_k(t)&=J_0 + \sum_{n=1}^N J_n +R_{N+1} \\
		&=a_k(0)+\sum_{n=1}^N \sum_{\substack{0 \le n_1,n_2\le N-1\\n_1+n_2=n-1}}IB(J_{n_1},J_{n_2})+J_{\sim N} +\mathcal{L}(R_{N+1})+Q(R_{N+1},R_{N+1}) \, .
	\end{align*}
	Then, identifying the terms in equation \eqref{restequation}, we obtain for $0 \le t \le T_{max}$ :
	\begin{align*}
		a_k(t)&=a_k(0) + \sum_{n=1}^N \sum_{\substack{0 \le n_1,n_2\le N-1\\n_1+n_2=n-1}}IB(J_{n_1},J_{n_2}) + \sum_{\substack{n_1,n_2 \le N \\ n_1+n_2\ge N}}IB(J_{n_1},J_{n_2}) \\
		&+ 2\sum_{n_1\le N}IB(J_{n_1},R_{N+1}) + IB(R_{N+1},R_{N+1}) \\
		&=a_k(0)+IB(a,a) \, .
	\end{align*}
	We also have in a distributional sense in $(0,T_{max})$ :
	\begin{align*}
		\dt a &=\sum_{n=0}^N \dt J_n+\dt R_{N+1} \\
		&= \sum_{\substack{0\le n_1,n_2\le N \\ n_1+n_2 \le N-1}} B(J_{n_1},J_{n_2}) + \sum_{\substack{0\le n_1,n_2\le N \\ n_1+n_2 \ge N}} B(J_{n_1},J_{n_2}) +2 \sum_{n\le N}B(J_n,R_{N+1}) + B(R_{N+1},R_{N+1}) \\
		&=B(\sum_{n=0}^N  J_n+ R_{N+1},\sum_{n=0}^N  J_n+ R_{N+1}) \\
		&= B(a,a) \, ,
	\end{align*}
	i.e. $a$ is a solution to equation \eqref{equationak}. \\
	Then rescaling in time and defining $u:=\Ft_x^{-1}(e^{it|k|^2}a)$, we have that u is a mild solution of equation \eqref{nls}.  \\
	
	\noi Furthermore,
	\begin{align*}
		\|u\|_{L_t^\infty L_x^\infty} & \le \|u\|_{L_t^\infty h^s}\le \|u\|_{h^{s,b}}\le \|a\|_{\Tilde{h}^{s,b}} \le L^{\theta+C(b-1/2)}\frac{1-\rho^{N+1}}{1-\rho}+\rho^N \\
		&\le 2 L^{\theta+C(b-1/2)}\frac{1}{1-\rho}=2 L^{\theta+C(b-1/2)}\frac{1}{1-L^{-\delta}} 
	\end{align*}
	for $N$ large enough. This bound is valid for $u$ solution to \eqref{weaknls}. We multiply by $\al$ to get a solution to \eqref{nls} and get the desired bound.

	\section{Proof of first bound} \label{proofbound1}
	We want to prove the first bound i.e. Proposition \ref{bound1}. We first prove a lemma that allows us to go from a bound on the $L_\o^2(\O)$ norm to a $L^\infty$ bound with high probability.
	\begin{lemma} \label{hypercontractivity} Let $g:=(g_{i})_{i\in \Z^d}$ be i.i.d. centered normal complex Gaussian random variables. Define for $(a_{i})_{i\in \Z^d}$ constants in $\C$ : 
		\begin{align*}
			F(w):=\sum_{k_1,...,k_n\in \Z^d}a_{k_1}...a_{k_n}g_{k_1}...g_{k_n} \,.
		\end{align*}
		Then if $S_n$ denotes the set of permutations of finite sets of size $n$ :
		\begin{align*}
			\E\big[|F(w)|^2\big] = \sum_{k_1,...,k_n\in \Z^d}\sum_{\sigma \in S_n}a_{k_1}...a_{k_n} \cj{a_{k_{\sigma(1)}}...a_{k_{\sigma(n)}}} 
		\end{align*}
		and there exists $C,c >0$ independent of $n$ such that :
		\begin{align*}
			\P\Big(|F(\o)|\ge A \E\big[|F(w)|^2\big]^{1/2} \Big)\le Ce^{-cA^{\frac{2}{n}}}
		\end{align*}
		for $A \ge 0 $.
	\end{lemma}
	\begin{proof}
		The first inequality comes from Isserlis' theorem for complex standard Gaussians (see~~\cite{Isserlistheorem}) which states that if $(X_1,...,X_n)$ is a zero-mean multivariate normal random vector, then~~:
		\begin{align*}
			\E[\prod_{i=1}^n X_i] =\sum_{p\in P_n}\prod_{\{i,j\}\in p}\text{Cov}(X_i,X_j)
		\end{align*}
		where $P_n$ is the set of pairings in $\{1,...,n\}$. \\
		In our case, this gives :
		\begin{align*}
			\E\big[|F(w)|^2\big] &=\E\Big[ \sum_{k_1,...,k_n\in \Z^d}a_{k_1}...a_{k_n} g_{k_1}...g_{k_n}\cj{\sum_{k_1',...,k_n'\in \Z^d}a_{k_1'}...a_{k_n'} g_{k_1'}...g_{k_n'}}\Big] \\
			&= \sum_{k_1,...,k_n\in \Z^d}\sum_{k_1',...,k_n'\in \Z^d}a_{k_1}...a_{k_n} \cj{a_{k_1'}...a_{k_n'}}\E\Big[g_{k_1}...g_{k_n} \cj{g_{k_1'}}...\cj{g_{k_n'}}\Big] \, .
		\end{align*}
		So , if we note $X:=(g_{k_1},...,g_{k_n},\cj{g_{k_1'}},...,\cj{g_{k_n'}})$ indexed by $1\le i \le 2n$, we get :
		\begin{align*}
			\E\big[|F(w)|^2\big] &=\sum_{k_1,...,k_n\in \Z^d}\sum_{k_1',...,k_n'\in \Z^d}a_{k_1}...a_{k_n} \cj{a_{k_1'}...a_{k_n'}}\sum_{p\in P_{2n}}\prod_{\{i,j\}\in p}\text{Cov}(X_i,X_j) \\ 
			&=\sum_{k_1,...,k_n\in \Z^d}\sum_{k_1',...,k_n'\in \Z^d}a_{k_1}...a_{k_n} \cj{a_{k_1'}...a_{k_n'}}\sum_{p\in P_{2n}}\prod_{\{i,j\}\in p}\delta_{k_i=k'_{j-n}} \\
			&=\sum_{k_1,...,k_n\in \Z^d}\sum_{\sigma \in S_n}a_{k_1}...a_{k_n} \cj{a_{k_{\sigma(1)}}...a_{k_{\sigma(n)}}} \, .
		\end{align*}
		
		\noi For the second inequality we use hypercontractivity of the Ornstein-Uhlenbeck semigroup as in \cite{hypercontractivity}. We recall hypercontractivity states that for $(g_i)_{i\in \N}$ i.i.d centered normal Gaussians (real or complex valued), for $J$ a finite set and $\big(P_n^{(j)}\big)_{j\in J}$ finite sequence of polynomials of degree less than $n$, for $p\ge 2$, we have (see for example \cite{tzvetkovhypercontract} for a proof):
		\begin{align}\label{boundhyper}
			\big\|\sum_{j\in \N}P_n^{(j)}(g)\big\|_{L^p(\O)} \le (p-1)^{n/2}\big\|\sum_{j\in \N}P_n^{(j)}(g)\big\|_{L^2(\O)}
		\end{align}
		where the valuation in $g$ of $P_n^{(j)}$ means one can choose any $n$-tuple of Gaussians in the family $g$ to evaluate each $P_n^{(j)}$. \\
		
		\noi Now to prove Lemma \ref{hypercontractivity}, we can first assume that $\E\big[|F|^2\big]<+\infty$ as otherwise, there is nothing to show. \\
		\noi Now set $(F_k)_{k\in \N}$ a sequence of finite sums of polynomials of degree $\le n$ in $g$ such that $F_k(g(\o))\xrightarrow[k\rightarrow +\infty]{a.s} F(\omega)$. As $F_k$ is a finite sum of polynomials, $F_k(g)$ has finite moments. On can assume up to renormalization that $\E\big[|F_k(g)|^2\big]=1$ for all $k\in \N$. Note that by Fatou's lemma, $\E\big[|F|^2\big]\le \lim \inf_{k\in \N} \E\big[|F_k(g)|^2\big] = 1$. \\
		\noi This approximation will work because bound \eqref{boundhyper} does not depend on the number of different i.i.d Gaussians involved, just on the maximum degree of the monomials.\\
		\noi We then write by Chebyshev inequality :
		\begin{align*}
			\P\Big(|F_k(g(\o))|\ge A \E\big[|F_k(g(\o))|^2\big]^{1/2} \Big)&\le \frac{\|F_k(g)\|_{L_\omega^p}^p}{A^p} \\
			& \le (p-1)^{pn/2}/A^p
		\end{align*}
		Choose $p$ so that $\frac{(p-1)^{n/2}}{A}=e^{-1}$ i.e. $p=1+(e^{-1}A)^{2/n}$ which is larger than 2 for $A\ge e$. Then :
		\begin{align*}
			\P\Big(|F_k(g(\o))|\ge A \E\big[|F_k(g(\o))|^2\big]^{1/2} \Big) &\le e^{-(1+(e^{-1}A)^{2/n})} \\
			& \le e^{1}e^{-e^{-2}A^{2/n}} \, .
		\end{align*}
		Now we only have to treat the case where $0\le A \le e$. In that case for all $n\ge 1$ :
		\begin{align*}
			e^{-A^{2/n}} \ge e^{-e^{2/n}} \ge e^{-e^2} \ge C 1 \ge C\P\Big(|F_k(g(\o))|\ge A \E\big[|F_k(g(\o))|^2\big]^{1/2} \Big) \, .
		\end{align*}
		Hence the result for $F_k(g)$ for $A\ge 0$. \\

		\noi Then we take $k\rightarrow +\infty$ and get :
		\begin{align*}
			\P\Big(|F(\o)|\ge A \E\big[|F(\o)|^2\big]^{1/2} \Big) &= \P\Big(\liminf_{k \in \N}|F_k(g(\o))|\ge A \E\big[|F(\o)|^2\big]^{1/2} \Big) \\
			& \le  \P\Big(\liminf_{k \in \N}|F_k(g(\o))|\ge \liminf_{k \in \N}A \E\big[|F_k(g(\o))|^2\big]^{1/2} \Big) \, .
		\end{align*}
		Then we use the fact convergence almost surely implies convergence in probability and the inequality right above to get :
		\begin{align*}
			\P\Big(|F(\o)|\ge A \E\big[|F(\o)|^2\big]^{1/2} \Big) 
			& \le \P\Big(\liminf_{k \in \N}|F_k(g(\o))|\ge A \Big) \\
			&=\lim_{k\in \N} \P\Big(|F_k(g(\o))|\ge A \Big) \\
			& \le Ce^{-cA^{2/n}} \, .
		\end{align*}
		Hence the result.
	\end{proof}
	
	\noi We now prove Proposition \ref{bound1}. 
	
	\subsection{Case of Theorem \ref{theorem1}}\label{subsection-1} In this case, we can prove Proposition \ref{bound1} by using the method in Subsection \ref{subsection0} below and applying only trivial counting estimates. \\

	\subsection{Case of Theorem \ref{theorem2}} \label{subsection0}

	Let $\mathcal{T}$ be a tree of scale $0\le n\le 2N$. We will denote by $k_l : l \in \mathcal{L}$ admissible decorations of $\mathcal{T}$. \\

	\noi For $\tau$ fixed, we apply Lemma \ref{hypercontractivity} with $A=L^{\theta}$. \\
	\noi We will denote by $\mathcal{T}'$ a copy of the tree $\mathcal{T}$ where the wave numbers of the leaves have been permuted by some $\s \in S_{n+1}$ (we recall the entirety of the decoration of $\mathcal{T'}$ is determined by the decoration of its leaves) ; we will use the abusive notation for the tree $\mathcal{T}'$ : $ \mathcal{N}' ~=:~\{\s(n), \ n~\in~\mathcal{N}~\}$, $\mathcal{L}'=:\{\s(l), \ l\in \mathcal{L}\}$, $\s(r)$ its root. This tree will be used to represent the conjugate term of $\Ft_{t}(J_{\mathcal{T}})(\tau,k)$ when we compute its $L^2(\O)$ norm. \\
	\noi We get $L-$certainly for a tree $\mathcal{T}$ of scale $n$ :
	\begin{align*}
		\jb{k}^{4s}\left|\Ft_{t}(J_{\mathcal{T}})(\tau,k)\right|^2 &\les \sum_{\s \in S_{n+1}}L^\theta \left(\frac{\al T_{max}}{L^{d/2}}\right)^{2n}\sum_{\substack{k_l, \ l \in \mathcal{L}\\}}\left|K_\mathcal{T}(\tau, \ k_n:n\in \mathcal{T})\right|\left|K_\mathcal{T}(\tau, \ k_{\s(n)}:n\in \mathcal{T})\right| \\
		& \times \sqrt{n_{in}(k_l)n_{in}(k_{\s(l)})} \, .
	\end{align*}
	\noi We can assume that the Fourier modes of $u_0$ satisfy $|k_l|\le L^\theta$ thanks to the decay of $\sqrt{n_{in}}$. Indeed, otherwise, if there exists $l_0 \in \mathcal{L}$ such that $|k_{l_0}| \ge L^\theta$, then for any arbitrary $r>0$, we can bound the sum using rough bounds and Lemma \ref{lemmaK} :
	\begin{align*}
		\sum_{\substack{k_l, \ l \in \mathcal{L}\\|k_{l_0 }|\ge L^\theta}}\left|K_\mathcal{T}(\tau, \ k_n:n\in \mathcal{T})\right|&\left|K_\mathcal{T}(\tau, \ k_{\s(n)}:n\in \mathcal{T})\right| \sqrt{n_{in}(k_l)n_{in}(k_{\s(l)})} \\ &\les_r L^{-\theta r} 4^n \sum_{\substack{k_l, \ l \in \mathcal{L}-\{l_0\}\\}}\sqrt{n_{in}(k_l)n_{in}(k_{\s(l)})} \\
		& \les_r L^\theta L^{dn}L^{-\theta r} \\
		& \les_r L^{\theta+dN-r\theta} \, .
	\end{align*}
	This implies, for an admissible decoration, that we can assume that $|k_n|\le L^\theta $ for all $n\in \mathcal{N}$ as $L^\theta$ absorbs the constant depending on $n$. \\
	
	\noi We get $L-$certainly for a tree $\mathcal{T}$ of scale $n$ :
	\begin{align*}
		\jb{k}^{4s}\left|\Ft_{t}(J_{\mathcal{T}})(\tau,k)\right|^2 \les \sum_{\s \in S_{n+1}}L^\theta \left(\frac{\al T_{max}}{L^{d/2}}\right)^{2n}\sum_{\substack{k_l, \ l \in \mathcal{L}\\ |k_l|\le L^\theta}}\left|K_\mathcal{T}(\tau, \ k_n:n\in \mathcal{T})\right|\left|K_\mathcal{T}(\tau, \ k_{\s(n)}:n\in \mathcal{T})\right| \, .
	\end{align*}
	\noi We take the supremum on the permutations $\s$, we bound the cardinal of permutations i.e. $\#S_{n+1}=(n+1)!$ by $L^\theta$ and get 
	\begin{align}\label{boundfirst}
		\jb{k}^{4s}\left|\Ft_{t}(J_{\mathcal{T}})(\tau,k)\right|^2 \les L^\theta \sup_{\s \in S_{n+1}}\left(\frac{\al T_{max}}{L^{d/2}}\right)^{2n}\sum_{\substack{k_l, \ l \in \mathcal{L}\\ |k_l|\le L^\theta}}\left|K_\mathcal{T}(\tau, \ k_n:n\in \mathcal{T})\right|\left|K_\mathcal{T}(\tau, \ k_{\s(n)}:n\in \mathcal{T})\right| \, .
	\end{align}
	We now fix $\s \in S_{n+1}$. We want to integrate in $\tau$ and bound the quantity :
	\begin{align*}
		\int_\R \jb{\tau}^{2b}\sum_{\substack{k_l, \ l \in \mathcal{L}\\ |k_l|\le L^\theta}}\left|K_\mathcal{T}(\tau, \ k_n:n\in \mathcal{T})\right|\left|K_\mathcal{T}(\tau, \ k_{\sigma(n)}:n\in \mathcal{T})\right|d\tau \, .
	\end{align*}
	\noi The issue is that we only have a bound on $K_\mathcal{T}$ for a fixed $\tau$ $L$-certainly. We will use the decay of $K_\mathcal{T}$ to get from a bound on a fixed $\tau$ to abound on all $\tau$ $L$-certainly. \\

	\noi We recall that $(q_n)_{n\in \mathcal{N}}$ is defined in Definition \ref{definitionomega}. Because for an admissible decoration, for all $n\in \mathcal{N}$, $|k_n|\le L^\theta$, we have that $|q_n| \le L^{2\theta}$ for all $n \in \mathcal{N}$. We then argue that if $\tau > L^{d+\theta}T_{max}$ (we choose this arbitrarily but any condition $\tau \gg L^{2\theta}T_{max}$ would yield the same type of computation), then $|\tau-T_{max}j_rq_r| \ges \tau/2$ for $L$ large enough and so using Lemma~~\ref{lemmaK}, we get :
	
	\begin{align*}
		\int_{|\tau|>L^{d+\theta}T_{max}}\sum_{\substack{k_l, \ l \in \mathcal{L}\\ |k_l|\le L^\theta}}\left|K_\mathcal{T}(\tau, \ k_n:n\in \mathcal{T})\right|\left|K_\mathcal{T}(\tau, \ k_{\s(n)}:n\in \mathcal{T})\right|d\tau &\les_A L^{(n+1)d(1+\theta)}\int_{|\tau|>L^{d+\theta}T_{max}} \jb{\tau}^{-A}d\tau  \\
		&\les L^{(n+1)d(1+\theta)+(1-A)(d+\theta)}T_{max}^{1-A}
	\end{align*}
	for any $A>0$. Thus we can consider only $|\tau| \le L^{d+\theta}T_{max}$ up to a negligible error. \\
	
	\noi  Then we use the bound of Lemma \ref{lemmaK} on the differential of $K$ to get :
	\begin{align*}
		\left|K_\mathcal{T}(\tau, \ k_n:n\in \mathcal{T})-K_\mathcal{T}(\tau', \ k_n:n\in \mathcal{T})\right| \les L^\theta|\tau-\tau'| \, .
	\end{align*}
	We also notice that if we denote by $g(\tau,k_n  :  n\in \mathcal{T})$ the right hand side of inequality \eqref{estimateonK} in Lemma~~\ref{lemmaK}, then we have that :
	\begin{equation} \label{estimateong}
		|g(\tau,k_n  :  n\in \mathcal{T})|, \ |\partial_\tau g(\tau,k_n  :  n\in \mathcal{T})| \le L^\theta \, .
	\end{equation}
	Now we take the interval $[-L^{d+\theta}T_{max},L^{d+\theta}T_{max}]$ and split it into $2L^{d+\theta + D}T_{max}$ intervals of length $L^{-D}$ for some $D>0$. We fix one $\tau_j$ for each of these intervals and apply the bound~~$\eqref{boundfirst}$. Now for some $\tau \in [-L^{d+\theta}T_{max},L^{d+\theta}T_{max}]$, we pick $\tau_j$ in the same interval as $\tau$ and write~~:
	\begin{align*}
		\jb{k}^{4s}\left|(\Tilde{J_{\mathcal{T}}})_k(\tau)\right|^2 &\les L^\theta \left(\frac{\al T_{max}}{L^{d/2}}\right)^{2n}\sum_{\substack{k_l, \ l \in \mathcal{L}\\ |k_l|\le L^\theta}}\left|K_\mathcal{T}(\tau, \ k_n:n\in \mathcal{T})\right|\left|K_\mathcal{T}(\tau, \ k_{\s(n)}:n\in \mathcal{T})\right| \\
		& \le L^\theta \left(\frac{\al T_{max}}{L^{d/2}}\right)^{2n}\sum_{\substack{k_l, \ l \in \mathcal{L}\\ |k_l|\le L^\theta}}\Big(\left|K_\mathcal{T}(\tau_j, \ k_n:n\in \mathcal{T})\right|\left|K_\mathcal{T}(\tau_j, \ k_{\s(n)}:n\in \mathcal{T})\right| +L^\theta L^{-D}\Big) \, .
	\end{align*}
	We use the bounds \eqref{estimateong} on $g$ and get :
	\begin{align*}
		\jb{k}^{4s}\left|(\Tilde{J_{\mathcal{T}}})_k(\tau)\right|^2 & \les  L^\theta \left(\frac{\al T_{max}}{L^{d/2}}\right)^{2n}\sum_{\substack{k_l, \ l \in \mathcal{L}\\ |k_l|\le L^\theta}}\Big(|g(\tau_j,k_n  :  n\in \mathcal{T})||g(\tau_j,k_{\s(n)}  :  n\in \mathcal{T})| +L^\theta L^{-D}\Big) \\
		& \les  L^\theta \left(\frac{\al T_{max}}{L^{d/2}}\right)^{2n}\sum_{\substack{k_l, \ l \in \mathcal{L}\\ |k_l|\le L^\theta}}\Big(|g(\tau,k_n  :  n\in \mathcal{T})||g(\tau,k_{\s(n)}  :  n\in \mathcal{T})| +L^\theta L^{-D}\Big) \, .
	\end{align*}
	Now we recall that $T_{max} \le L^{C(\gamma,d)}$. By the preceding computations, we know that for $D$ fixed large enough, if we show that inequality 
	\begin{equation}\label{boundforallt}
		\jb{k}^{4s}\left|(\Tilde{J_{\mathcal{T}}})_k(\tau)\right|^2  \les  L^\theta \left(\frac{\al T_{max}}{L^{d/2}}\right)^{2n}\sum_{\substack{k_l, \ l \in \mathcal{L}\\ |k_l|\le L^\theta}}|g(\tau,k_n  :  n\in \mathcal{T})||g(\tau,k_{\s(n)}  :  n\in \mathcal{T})|
	\end{equation}| 
	is true for all $\tau_j$ L-certainly on an event $E_j \subset \O$, then it will be true for all $\tau\in [-L^{d+\theta}T_{max},L^{d+\theta}T_{max}]$ up to a negligible error on the event $\underset{j}{\bigcap}E_j$. We notice that for $D(\gamma,d)$ large, there are at most $L^{2D}T_{max}$ intervals and so the bound \eqref{boundforallt} is true for all $\tau_j$ on an event of probability $\ge 1-L^{2D}T_{max}Ke^{-cL^{\theta}}\ge 1-Ke^{-cL^\theta}$ by modifying the constants $\theta$ and $K$. \\
	
	\noi We can now bound the quantity :
	\begin{align*}
		\Lambda:=\int_\R \jb{\tau}^{2b}\sum_{\substack{k_l, \ l \in \mathcal{L}\\ |k_l|\le L^\theta}}\left|K_\mathcal{T}(\tau, \ k_n:n\in \mathcal{T})\right|\left|K_\mathcal{T}(\tau, \ k_{\sigma(n)}:n\in \mathcal{T})\right|d\tau 
	\end{align*}
	by integrating in $\tau \in \R$, $|\tau|\le L^{d+\theta}T_{max}$ the expression :
	\begin{multline*}
		\jb{\tau}^{2b}\sum_{\substack{k_l, \ l \in \mathcal{L}\\ |k_l|\le L^\theta}}|g(\tau,k_n  :  n\in \mathcal{T})||g(\tau,k_{\s(n)}  :  n\in \mathcal{T})| \\ \les \sum_{j_r,j_{\s(r)}\in \{0,1\}}\Big|\left\{\text{choices for $j_n,j_{\s(n)}$}, \ n\in \mathcal{N}-\{r\}\right\}\Big| 
		\sup_{j_n,j_{\s(n)}, \ n\in \mathcal{N}-\{r\}}\sum_{\substack{k_l, \ l \in \mathcal{L}\\ |k_l|\le L^\theta}}\jb{\tau}^{2b} \\ \times \jb{\tau - T_{max}j_rq_r}^{-10}\jb{\tau-T_{max}j_{\s(r)}q_{\s(r)}}^{-10}\prod_{n\in \mathcal{N}}\jb{T_{max}q_n}^{-1}\jb{T_{max}q_{\sigma(n)}}^{-1} \, .
	\end{multline*}
	So once we integrate, we get up to a negligible error :
	\begin{align}
		\Lambda &\les \int_{|\tau|\le L^{d+\theta}T_{max}} \jb{\tau}^{2b}\sum_{\substack{k_l, \ l \in \mathcal{L}\\ |k_l|\le L^\theta}}|g(\tau,k_n  :  n\in \mathcal{T})||g(\tau,k_{\s(n)}  :  n\in \mathcal{T})|d\tau \notag \\
		& \les L^\theta \int_\R \sum_{j_r,j_{\s(r)}\in \{0,1\}}\sup_{j_n,j_{\s(n)}, \ n\in \mathcal{N}-\{r\}}\sum_{\substack{k_l, \ l \in \mathcal{L}\\ |k_l|\le L^\theta}}\frac{\jb{\tau}^{2b}}{\jb{T_{max}q_r}\jb{T_{max}q_{\s(r)}}} \jb{\tau - T_{max}j_rq_r}^{-10} \notag \\
		& \hspace{5 cm} \times \jb{\tau-T_{max}j_{\s(r)}q_{\s(r)}}^{-10}\prod_{\substack{n\in \mathcal{N} \\ n \neq r}}\jb{T_{max}q_n}^{-1}\jb{T_{max}q_{\sigma(n)}}^{-1} d\tau \, .  \label{technicaluse}
	\end{align}
	We use a technique similar to the proof of Lemma \ref{lemmaK}. More precisely, we use Lemma \ref{technical} below with $b>1/2$, $n=10$, $\al=T_{max}j_rq_r$, $\beta=T_{max}j_{\s(r)}q_{\s(r)}$. We obtain, using additionally the triangular inequality  :
	\begin{align}
		\Lambda
		& \les L^\theta \sum_{j_r,j_{\s(r)}\in \{0,1\}} \sup_{j_n,j_{\s(n)}, \ n\in \mathcal{N}-\{r\}}\sum_{\substack{k_l, \ l \in \mathcal{L}\\ |k_l|\le L^\theta}} \left[\max \left(\jb{T_{max}q_r},\jb{T_{max}q_{\s(r)}}\right)^{-2+2b}\jb{T_{max}(j_rq_r-j_{\s(r)}q_{\s(r)})}^{-5}\right] \notag \\
		& \hspace{9 cm} \times \prod_{\substack{n\in \mathcal{N} \\ n \neq r}}\jb{T_{max}q_n}^{-1}\jb{T_{max}q_{\sigma(n)}}^{-1} \, .
	\end{align}
	\noi Now we have to bound :
	\begin{multline*}
		\sup_{j_n,j_{\s(n)}, \ n\in \mathcal{N}-\{r\}} \sum_{j_r,j_{\s(r)}\in \{0,1\}}\sum_{\substack{k_l, \ l \in \mathcal{L}\\ |k_l|\le L^\theta}}\left[\max \left(\jb{T_{max}q_r},\jb{T_{max}q_{\s(r)}}\right)^{-2+2b}\jb{T_{max}(j_rq_r-j_{\s(r)}q_{\s(r)})}^{-5}\right]\\
		\times \prod_{\substack{n\in \mathcal{N} \\ n \neq r}}\jb{T_{max}q_n}^{-1}\jb{T_{max}q_{\sigma(n)}}^{-1}
	\end{multline*}
	where we recall $T_{max}\le L^{C(\gamma,d)}$ and $\forall \, n\in \mathcal{N}, \ q_n \le L^{2\theta}$. \\
	\noi Let us denote by $[.]$ the integer part function. Now we fix the integer part of $T_{max}q_n, T_{max}q_{\s(n)}$ for all $n\in \mathcal{N}$ at the price of $L^{C(b-1/2)}$. This is justified by the following computation :
	\begin{align*}
		&\sup_{j_n,j_{\s(n)}, \ n\in \mathcal{N}-\{r\}}\sum_{j_r,j_{\s(r)}\in \{0,1\}} \sum_{\substack{k_l, \ l \in \mathcal{L}\\ |k_l|\le L^\theta}} \left[\max \left(\jb{T_{max}q_r},\jb{T_{max}q_{\s(r)}}\right)^{-2+2b}\jb{T_{max}(j_rq_r-j_{\s(r)}q_{\s(r)})}^{-5}\right] \\
		& \hspace{10 cm} \times \prod_{\substack{n\in \mathcal{N} \\ n \neq r}}\jb{T_{max}q_n}^{-1}\jb{T_{max}q_{\sigma(n)}}^{-1} \\
		&\les \sup_{j_n,j_{\s(n)}, \ n\in \mathcal{N}-\{r\}} \sum_{\substack{(u_n)_{n\in \mathcal{N}},(v_n)_{n\in \mathcal{N}} \in \Z^{\mathcal{N}}\\\forall n \in \mathcal{N}, |u_n|,|v_n|\le L^\theta L^{C(\gamma,d)}}}\sum_{\substack{k_l, \ l \in \mathcal{L}\\ |k_l|\le L^\theta}} \sum_{j_r,j_{\s(r)}\in \{0,1\}}\ind([T_{max}q_n]=u_n)\ind([T_{max}q_{\s(n)}]=v_n)\\
		& \hspace{5cm}\times \left[\max \left(\jb{u_r},\jb{v_r}\right)^{-2+2b}\jb{j_ru_r-j_{\s(r)}v_r}^{-5}\right]  \prod_{\substack{n\in \mathcal{N} \\ n \neq r}}\jb{u_n}^{-1}\jb{v_n}^{-1} 
	\end{align*}
	which is smaller by using the elementary property $\sum_j a_jb_j \le \sup_ja_j \sum_jb_j$ than
	\begin{multline*}
		L^\theta  \sup_{\substack{(u_n)_{n\in \mathcal{N}},(v_n)_{n\in \mathcal{N}} \in \Z^\mathcal{N}\\ \forall \, n \in \mathcal{N},|u_n|,|v_n|\le L^{C(\gamma,d)+\theta}}} \sum_{\substack{k_l, \ l \in \mathcal{L}\\ |k_l|\le L^\theta}} \ind([T_{max}q_n]=u_n)\ind([T_{max}q_{\s(n)}]=v_n)\\
		\times \sum_{\substack{(u_n)_{n\in \mathcal{N}},(v_n)_{n\in \mathcal{N}} \in \Z^{\mathcal{N}}\\\forall n \in \mathcal{N}, |u_n|,|v_n|\le L^\theta L^{C(\gamma,d)}}}\left[\max \left(\jb{u_r},\jb{v_r}\right)^{-2+2b}\jb{j_ru_r-j_{\s(r)}v_r}^{-5}\right]  \prod_{\substack{n\in \mathcal{N} \\ n \neq r}}\jb{u_n}^{-1}\jb{v_n}^{-1}  \, .
	\end{multline*}
	We compute the sum and obtain it to be smaller than :
	\begin{align*}
		&\les \Big(\log\big(L^\theta L^{C(\gamma,d)}\big)\Big)^{2n}L^\theta L^{ C(\gamma,d)(1-2b)}\sup_{\substack{(u_n)_{n\in \mathcal{N}},(v_n)_{n\in \mathcal{N}} \in \Z^\mathcal{N}\\ \forall \, n \in \mathcal{N},|u_n|,|v_n|\le L^\theta L^{C(\gamma,d)(\gamma,d)}}}\sum_{\substack{k_l, \ l \in \mathcal{L}\\ |k_l|\le L^\theta}} \ind([T_{max}q_n]=u_n)\ind([T_{max}q_{\s(n)}]=v_n) \\
		&\les L^{C(\gamma,d)(b-1/2)}\sup_{\substack{(u_n)_{n\in \mathcal{N}},(v_n)_{n\in \mathcal{N}} \in \Z^\mathcal{N}\\ \forall \, n \in \mathcal{N},|u_n|,|v_n|\le L^\theta L^{C(\gamma,d)}}}\sum_{n\in \mathcal{N}}\sum_{\substack{k_l, \ l \in \mathcal{L}\\ |k_l|\le L^\theta}} \ind([T_{max}q_n]=u_n)
	\end{align*}
	where we have ignored the admissible conditions relative to $\mathcal{T'}$. However, we lose nothing by ignoring them in the general case as one possible scenario is that $\mathcal{T}$ and $\mathcal{T'}$ have the same decorations and thus all conditions are the same for both trees. \\
	
	\noi Now we give a bound to the counting problem :
	\begin{equation} \label{initialcountingbound}
		\sup_{\substack{(u_n)_{n\in \mathcal{N}},(v_n)_{n\in \mathcal{N}} \in \Z^\mathcal{N}\\ \forall \, n \in \mathcal{N},|u_n|,|v_n|\le L^{C(\gamma,d)+\theta}}}\sum_{n\in \mathcal{N}}\sum_{\substack{k_l, \ l \in \mathcal{L}\\ |k_l|\le L^\theta}} \ind([T_{max}q_n]=u_n) \, .
	\end{equation}
	If we now denote by $M_{\mathcal{T}}$ the result of the counting \eqref{initialcountingbound} for a decorated tree $\mathcal{T}$ (we omit the dependence on the decoration here), we will get, putting all we have done so far together :
	\begin{align*}
		\int_\R \jb{\tau}^{2b}\jb{k}^{4s}\left|\Ft_{t}(J_{\mathcal{T}})(\tau,k)\right|^2 d\tau \le L^{\theta + C(b-1/2)}\left(\frac{\al T_{max}}{L^{d/2}}\right)^{2n}M_{\mathcal{T}} \, .
	\end{align*}
	Now we compute a bound on $M_\mathcal{T}$.
	\noi Let $(u_n)_{n \in \mathcal{N}} \in \Z^\mathcal{N}$ be fixed. We notice that fixing every integer part of $T_{max}q_n$ for $n \in \mathcal
	N$ is equivalent to fixing the integer part of every $T_{max}\O_n$ as $\O_n$ and $q_n$ are linear combinations one of another. Thus fixing $(u_n)_{n \in \mathcal{N}}$ is equivalent to fixing $(\s_n)_{n\in \mathcal{N}} \in \R^\mathcal{N}$ and looking at the situation where :
	\begin{align*}
		\forall \, n \in \mathcal{N}, \ |\sigma_n-\O_n|\le T_{max}^{-1} \, .
	\end{align*}
	We thus have that :
	\begin{align} \label{secondboundstatement}
		M_{\mathcal{T}}= \# \left\{(k_n)_{n\in \mathcal{N}} \ \text{admissible}, \, |k_l|\le L^\theta, \ \forall \, n \in \mathcal{N}, |\O_n-\sigma_n|\le T_{max}^{-1}, \ k_r=k  \right\}
	\end{align}
	where $\#A$ denotes the cardinal of a set $A$. \\
	
	\noi Our goal is to show that
	\begin{equation*}
		M_{\mathcal{T}} \le \left(L^dL^\theta T_{max}^{-1} \right)^n \, .
	\end{equation*}

	\noi We consider the simpler problem, corresponding to the counting for a single node :
	\begin{align*}
		m:=\# \left\{ k_1, \ k_2 \in \Z_L^d, \ k_1+k_2=k, \ |2k_1\cdot k_2-\sigma|\le T_{max}^{-1}, \ |k_i|\le L^\theta  \right\} \ \ \ \text{for $\sigma \in \R$, $k\in \Z_L^d$ fixed.}
	\end{align*}
	and we have
	\begin{align*}
		m = \# \left\{ k_1 \in \Z^d, \ |2k_1\cdot (k-k_1)-\sigma L^{2}|\le L^{2}T_{max}^{-1}, \ |k_1|,|k|\le L^{1+\theta}  \right\} \, .
	\end{align*}
	
	\noi \textbf{We first treat the more complicated case of Theorem \ref{theorem2} where we use divisor bounds.} Skip to \eqref{skip1} for the case of Theorem \ref{theorem1}. \\
	
	\noi We fix the first $d-2$ coordinates and get :
	\begin{align*} 
		& m  \le L^{(d-2)(1+\theta)}\# \big\{ N_1, N_2 \in \Z, \ |2N_1 (k_1-N_1)+2N_2 (k_2-N_2)-\sigma L^{2}|\le L^{2}T_{max}^{-1}, \\
		& \hspace{11cm}  |N_1|,|N_2|,|k_1|,|k_2|\le L^{1+\theta}  \big\} \,.
	\end{align*}
	For $N_1, N_2,k_1,k_2 \in \Z$ such that $|N_1|,|N_2|,|k_1|,|k_2|\le L^{1+\theta}$, we rearrange the expression to make a sum of squares appear :
	\begin{align*}
		|2N_1 (k_1-N_1)+2N_2 (k_2-N_2)-\sigma L^{2}| & =   |-2(N_1-\frac{k_1}{2})^2+k_1^2/2-2(N_2-\frac{k_2}{2})^2+k_2^2/2-\sigma L^{2}| \\
		&=  \frac{1}{2}|-(2N_1-k_1)^2+k_1^2-(2N_2-k_2)^2+k_2^2-2\sigma L^{2}| \, .
	\end{align*}
	We have thus :
	\begin{align*}
		m &\le  L^{(d-2)(1+\theta)} \# \big\{ N_1, N_2 \in \Z, \ |-(2N_1-k_1)^2+k_1^2-(2N_2-k_2)^2+k_2^2-2\sigma L^{2}|\le 2L^{2}T_{max}^{-1}, \\
		& \hspace{11 cm}|N_1|,|N_2|,|k_1|,|k_2|\le L^{1+\theta} \big\} \, .
	\end{align*}
	Now, we will use a divisor bound argument (see the bound of line (4.6) in \cite{countingTzvetkov} for reference). \\
	\noi Set $r\in \left(-k_1^2-k_2^2+2\sigma L^{2}+[-2L^2T_{max}^{-1},2L^2T_{max}^{-1}]\right)\cap \Z$. We consider :
	\begin{align*}
		\# \left\{ N_1, N_2 \in \Z, \ -(2N_1-k_1)^2-(2N_2-k_2)^2=r, \ |N_1|,|N_2|,|k_1|,|k_2|\le L^{1+\theta} \right\} \, .
	\end{align*}
	For $r=0$, it is bounded by $1$. \\
	
	\begin{itemize}
		\item  First we treat the case where $0<|r|\le L^{(1+\theta)40000}$. \\
		\noi The usual bound by divisors in $\Z[i]$ is : $$\les_\eps L^{(1+\theta)\eps} \, .$$ 
		
		\item Now we treat the case when $|r|> L^{(1+\theta)40000}$. We can suppose without loss of generality that $|2N_1-k_1|\le |2N_2-k_2|$. \\
		\noi First we have that $|k_1|\le L^{1+\theta}$ so $|2N_1-k_1|\le CL^{1+\theta}$. \\
		\noi If $2N_2-k_2 \ge 0$, then we have : $2N_2 \in k_2 +\left[\sqrt{|r|-C(L^{1+\theta})^2},\sqrt{|r|}\right]$ and $\#\left(\left[\sqrt{|r|-C(L^{1+\theta})^2},\sqrt{|r|}\right]\cap \Z\right) \les \sqrt{|r|}-\sqrt{|r|-C(L^{1+\theta})^2} \les \frac{L^{2(1+\theta)}}{L^{(1+\theta)40000}} $ and so we get the bound $$L^{1+\theta}\frac{L^{2(1+\theta)}}{L^{(1+\theta)40000}}\les 1 \, .$$ \\
		\noi Then we take the case where $2N_2-k_2 <0$ and count the $N_2$ such that $2N_2 \in k_2 +\left[-\sqrt{|r|},-\sqrt{|r|-c(L^{1+\theta})^2}\right]$ which yields the same bound. \\
	\end{itemize}
	\noi In summary, we get :
	\begin{align*}
		\# \left\{ N_1, N_2 \in \Z, \ -(2N_1-k_1)^2-(2N_2-k_2)^2=r, \ |N_1|,|N_2|,|k_1|,|k_2|\le L^{1+\theta}  \right\} \le L^{(1+\theta)\eps}
	\end{align*}
	which yields by summing in $r$ :
	\begin{align*}
		m\les L^{(d-2)(1+\theta)}L^2 T_{max}^{-1}L^{0+} \les L^\theta L^dT_{max}^{-1} \, .
	\end{align*}
	
	\noi Now we can bound $M_{\mathcal{T}}$ by finite induction. \\
	If a tree $\mathcal{T}'$ with root decorated by $k$ has $ \mathcal{T}_1$ and $\mathcal{T}_2$ as children with respective root decorated with $k_1$ and $k_2$ and if we denote by $M_{\mathcal{T}'}$ the counting bound \eqref{secondboundstatement} corresponding to a decorated tree $\mathcal{T}'$, then :
	\begin{align*}
		M_{\mathcal{T}'} \les \sum_{\substack{k_1,k_2\in \Z_L^d \\ k_1+k_2=k \\ |k_1|,|k_2|\les L^\theta}}M_{\mathcal{T}_1}M_{\mathcal{T}_2} \le \left(L^dL^\theta T_{max}^{-1} \right)\sup_{\substack{k_1,k_2\in \Z_L^d \\  |k_1|,|k_2|\les L^\theta}} M_{\mathcal{T}_1}M_{\mathcal{T}_2} \, .
	\end{align*}
	By iterating this at each branching nodes, that is to say $n$ times, we thus get :
	\begin{align*}
		M_\mathcal{T} \le \left(L^dL^\theta T_{max}^{-1} \right)^n \, .
	\end{align*}
	\textbf{In the context of Theorem \ref{theorem1}, one would only use the trivial bound :}
	\begin{equation} \label{skip1}
		m \le L^dL^\theta
	\end{equation}
	which implies by the same iteration
	\begin{equation*}
		M_{\mathcal{T}} \le (L^d L^\theta)^n \, .
	\end{equation*}
	In both cases we get the result with $\rho$ defined in \eqref{rhodefinition} :
	\begin{equation*}
		\sup_{k\in \Z_L^d}\jb{k}^{4s}\|J_\mathcal{T}\|_{h^b}^2\le L^{\theta + C(b-1/2)}\rho^{2n} \, .
	\end{equation*}

	\noi Note that this implies for $s>d/2$ :
	\begin{align*}
		\|J_{\mathcal{T}}\|_{\tilde{h}^{s,b}}^2 &\le L^{\theta + C(b-1/2)}\left(\frac{\al T_{max}}{L^{d/2}}\right)^{2n}\left(L^dL^\theta T_{max}^{-1} \right)^n \\
		&\le L^{\theta + C(b-1/2)}\rho^{2n } \, .
	\end{align*}

	\section{Proof of second bound} \label{proofbound2}
	\noi We prove the operator bound stated in Proposition \ref{bound2}. Let $\mathcal{T}$ be a tree of scale $0\le n\le N$. The dependence on $\mathcal{T}$ of all the operators below will be omitted to get simpler notations. \\
	
	\noi Let $v \in \tilde{h}^{s,b}$. We have :
	\begin{equation*}
		\mathcal{P}(v)(t,k)=\chi(t)\int_0^t \chi(s)\left(\frac{-i \al T_{max}}{L^{d/2}}\right)\sum_{\substack{k_1,k_2 \in \Z_L^d \\ k_1+k_2=k}} (J_\mathcal{T})_{k_1}(s)v_{k_2}(s)e^{-2\pi i s 2T_{max} k_1 \cdot k_2} ds \, .
	\end{equation*}
	\subsection{Case of Theorem \ref{theorem1}}
	We obtain this result when we use the method in Subsection \ref{subsection2} below and apply only trivial counting estimates. \\
	
	\subsection{Case of Theorem \ref{theorem2}} \label{subsection2}
	
	\noi First we do reductions as in Section~\ref{proofbound1}. By decay of $n_{in}$, one may assume $|k_l|\le L^\theta$ for all $l \in \mathcal{L}$. This implies $ |k_1|\le L^\theta$ . Hence $$L^{-\theta} \les \frac{\jb{k}^s}{\jb{k_2}^s}\les L^\theta$$ as $k=k_1+k_2$. This allows us to consider only the $\Tilde{h}^{0,b}$ norms. \\
	
	\noi Let $\mathcal{I}_1$ as defined in Lemma \ref{FT2}. 
	We wish to use this operator instead of $I$ in Lemma \ref{FT} as the formula is simpler since it does not involve choices of $(j_n)_{n\in \mathcal{N}}$. This is possible thanks to Lemma \ref{FT2} which gives us the bound : $$\|IF\|_{h^{b}}\les\|\mathcal{I}_1F\|_{h^{b}} $$
	for $b>1/2$. \\
	\noi From now on, we will consider we have replaced the operator $I$ with the operator $\mathcal{I}_1$ i.e. we consider we have replaced the operator $\mathcal{P}$ by :
	\begin{align*}
		v \in \Tilde{h}^{s,b} \mapsto \mathcal{I}_1B(J_\mathcal{T},v) \in \Tilde{h}^{s,b} 
	\end{align*}
	to compute the $\|.\|_{\Tilde{h}^{s,b}\xrightarrow[]{}\Tilde{h}^{s,b}}$ norm. \\
	
	\noi Now by Proposition \ref{duhamel} below (see \cite{Ginibre} for a reference) applied with $b=1$, $b'=0$ and $T=1$  and Young's inequality, we have for $v \in \Tilde{h}^{0,b}$ :
	\begin{align*}
		\|\mathcal{P}(v)\|_{\Tilde{h}^{0,1}} & \les \|B(J_\mathcal{T,}v)\|_{\Tilde{h}^{0,0}} \\
		& \les \|v\|_{l_k^2L_t^2}L^{-d/2}\frac{\al T_{max}}{L^{d/2}}\sum_{k_1\in \Z_L^d}\|(J_\mathcal{T})(t,k_1)\|_{L_t^\infty} \, .
	\end{align*}
	\noi So we get as $b> 1/2$ and using Proposition \ref{bound1} :
	\begin{align*}
		\|\mathcal{P}(v)\|_{\Tilde{h}^{0,1}} &\les \|v\|_{l_k^2L_t^2}L^{-d/2}\frac{\al T_{max}}{L^{d/2}}\sum_{k_1 \in \Z_L^d}\| (J_\mathcal{T})_{k_1}\|_{h^b} \\
		&\les \|v\|_{\tilde{h}^{0,b}} \frac{\al T_{max}}{L^{d/2}}\sum_{k_1\in \Z_L^d}L^{\theta +C(b-1/2)}\rho^n \jb{k_1}^{-2s} \\
		& \les \|v\|_{\Tilde{h}^{0,b}} \frac{\al T_{max}}{L^{d/2}}L^dL^{\theta +C(b-1/2)}\rho^n \, .
	\end{align*}
	We just bounded $\|\mathcal{P}\|_{\Tilde{h}^{0,b}\rightarrow \Tilde{h}^{0,1}}$. Now we want to bound another operator norm of $\mathcal{P}$ and use interpolation. More precisely, \textbf{we want to show that} :
	\begin{align} \label{boundsecond}
		\|\mathcal{P}\|_{\Tilde{h}^{0,b}\xrightarrow[]{}\Tilde{h}^{0,1-b}} \les L^\theta \rho^{n+1} \, .
	\end{align}
	First, assume that we have proved inequality \eqref{boundsecond}. \\
	\noi We use interpolation, writing $b=(1-t)1 +t(1-b)$ with $t=1/b-1$. We then get  for $b$ close to $1/2$ using $\rho \les L^{-\delta}$ and $T_{max}\le L^{C(\gamma,d)}$ :
	\begin{align*}
		\|\mathcal{P}\|_{\tilde{h}^{0,b}\xrightarrow[]{}\Tilde{h}^{0,b}} & \les L^\theta \rho^{(n+1)(1/b-1)}(\rho^n)^{2-1/b} L^{C(b-1/2)(2-1/b)}\left(\frac{\al T_{max}L^d}{L^{d/2}}\right)^{2-1/b} \\
		& \les L^{\theta +C(b-1/2)(2-1/b)} \rho^{n+1/b-1}\left(\frac{\rho L^d T_{max}^{1/2}}{L^{d/2}}\right)^{2-1/b} \\
		&\les L^\theta \rho^{n+1/2}
	\end{align*}
	which is the desired result. \\
	
	\noi We now return to the proof of bound \eqref{boundsecond}. \\
	\noi Define $H(\tau,\s):=\jb{\tau}I_1(\tau,\s)$. We write using the definition of $\mathcal{P}$ in Proposition \ref{bound2} :
	\begin{align}\label{defoperatorP}
		\Ft_{t}(\mathcal{P}v)(\tau)=\frac{-i\al T_{max}}{L^{d/2}}\jb{\tau}^{-1}\sum_{\substack{k_1,k_2 \in \Z_L^d \\ k_1+k_2=k}}\int_{\R^2}H(\tau,\s_1+\s_2-2T_{max}k_1\cdot k_2)\Ft_{t}(J_\mathcal{T})(\s_1,k_1)\Ft_{t}(v)(\s_2,k_2)d\s_1 d\s_2
	\end{align}
	where $H(\tau,\eta)$ has all its derivatives bounded by $\jb{\tau-\eta}^{-10}$ by Lemma \ref{FT}. \\
	\noi Now recall that :
	\begin{align}
		&\cdot \ \ \ \|f\|_{L_\tau^1l_k^2} \les \|\jb{\tau}^{b}f\|_{L_\tau^2l_k^2} \ \ \ \text{by Cauchy-Schwarz inequality as $b>1/2$} \label{ligne1} \\
		&\cdot \ \ \ \|\jb{\tau}^{1-b} \jb{\tau}^{-1}\Ft_{t}(f)\|_{l_k^2L_\tau^2} \les \|\Ft_{t}(f)\|_{L_\tau^\infty l_k^2} \ \ \ \text{as $b>1/2$.} \label{ligne2}
	\end{align}
	We bound using inequality \eqref{ligne2} :
	\begin{align*}
		\| \left(\mathcal{P}v\right)\|_{\tilde{h}^{0,1-b}}& \les \frac{\al T_{max}}{L^{d/2}} L^{-d/2} \\
		& \times \Bigg\| \sum_{\substack{k_1,k_2 \in \Z_L^d \\ k_1+k_2=k}} \int_{\R^2}H(\tau,\s_1+\s_2-2T_{max}k_1\cdot k_2)\Ft_{t}(J_\mathcal{T})(\s_1,k_1)\Ft_{t}(v)(\s_2,k_2)d\s_1 d\s_2\Bigg\|_{L_\tau^\infty l_k^2} \\
		&\les \sup_{\tau, \s_2 \in \R} \|\chi(\tau,\s_2)  \|_{ l_k^2\xrightarrow[]{}l_k^2} L^{-d/2}\|\Ft_{t}(v)\|_{L_\tau^1 l_k^2} 
	\end{align*}
	which is smaller by inequality \eqref{ligne1} than
	\begin{align*}
		\sup_{\tau, \s_2 \in \R} \|\chi(\tau,\s_2)  \|_{ l_k^2\xrightarrow[]{}l_k^2} \|v\|_{\Tilde{h}^{0,b}}
	\end{align*}
	where we defined the family of operators $(\tau,\s_2)\in \R^2 \mapsto \chi(\tau,\s_2)\in \mathcal{L}(l_k^2,l_k^2)$ by
	\begin{equation}\label{formulaofchi}
		\big(\chi(\tau,\s_2) v\big)_k:= \frac{\al T_{max}}{L^{d/2}}\sum_{\substack{k_1,k_2 \in \Z_L^d \\ k_1+k_2=k}}v(k_2) \int_{\R}H(\tau,\s_1+\s_2-2T_{max}k_1\cdot k_2)\Ft_{t}(J_\mathcal{T})(\s_1,k_1)d\s_1 \, .
	\end{equation}
	We want to write the operator $\chi$ as a tree expansion and apply the same counting arguments as in Section \ref{proofbound1}. In order to do this, we see the $\Ft_{t}(J_\mathcal{T})$ term in equality \eqref{formulaofchi} as a tree.  The leaves will correspond to the Gaussians of our initial data. We still need to represent the argument $v$ that the operator is evaluated at in \eqref{formulaofchi}. We will represent the empty space for the argument by a special leaf called "placeholder leaf". Now to sum up, $\chi$ is an operator that sticks a placeholder leaf $r'$ (corresponding to where the argument of $\chi$ will be put) and the tree $\mathcal{T}$ to the root $r$ of wave number $k$ for any decoration. It can be represented by :
	\begin{figure}[H]
		\centering
		
		\begin{tikzpicture}[->,>=stealth',level/.style={sibling distance = 5cm/#1,
				level distance = 1.5cm}] 
			\node[arn_n, label=90:{$k$}] {$r$} 
			child{ node [arn_r, label=90:{$k_1$}] {$\mathcal{T}$} } 
			child{node [arn_r,label=90:{$k_2$}, label=270:{placeholder for $v$}] {$r'$}}
			;
		\end{tikzpicture}
		\caption{Tree representation of the operator $\chi$. The leaf $r'$ is a placeholder that will receive the argument of the operator. The tree $\mathcal{T}$ contains the rest of the information of the operator}
		\label{fig2}
		
	\end{figure}
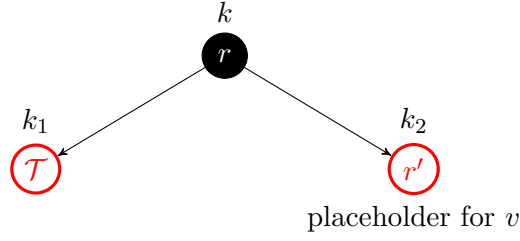
	
	\noi We want to write this operator as :
	\begin{align*}
		&(\chi(\tau,\s_2) v)_k  = \sum_{k'\in \Z_L^d}\chi_{kk'}(\tau,\s_2)v_{k'} \,.
	\end{align*}
	By setting :
	$$q_r=q_{\text{root of $\mathcal{T}$}}-2k_1\cdot k_2\, ,$$
	we get that the kernel is equal to :
	\begin{align}
		& \chi_{kk'}(\tau,\s_2)= \left(\frac{\al T_{max}}{L^{d/2}} \right)^{n+1} \sum_{k_n \ \text{admissible}}\mathcal{K}_\mathcal{T}(\tau,\s_2,(k-k'\cdot k'), \{k_l, l\neq r'\})\frac{1}{\jb{T_{max}q_r--\zeta}^5} \notag \\
		& \hspace{11 cm}\times \prod_{n\in \mathcal{N}-\{r\}} \frac{1}{\jb{T_{max}q_n}}\prod_{l\in \mathcal{L}-\{r'\}}g_{Lk_l} \label{kernelofchi}
	\end{align}
	for some function $\mathcal{K}_\mathcal{T}$, where $(k_n)_n$ admissible if $k_r=k$, $k_{r'}=k'$ and $|k_n|\le L^\theta$ for $n\neq r, r'$. \\
	\noi By Lemma \ref{lemmaK} we have that :
	\begin{align*}
		\Ft_{t}(J_{\mathcal{T}})(\tau,k)=\left(\frac{-i\al T_{max}}{L^{d/2}}\right)^n \sum_{k_n, \ n\in \mathcal{T}}K_{\mathcal{T}}(\tau,k_n:n\in \mathcal{T})\prod_{l\in \mathcal{L}}\sqrt{n_{in}(k_l)}g_{Lk_l}(\o) \, .
	\end{align*}
	This yields~:
	\begin{multline*}
		\mathcal{K}_\mathcal{T} (\tau,\s_2,(k-k'\cdot k'), \{k_l, l\neq r'\})\\ = i^n \jb{T_{max}q_r-\zeta}^5\prod_{n\in \mathcal{N}-\{r\}} \jb{T_{max}q_n} \int_{\R}H(\tau,\tau_1+\s_2-2T_{max}k_1\cdot k_2) K_\mathcal{T}(\tau_1,k_n \ : \ n \in \mathcal{T})d \tau_1 \, .
	\end{multline*}
	First, we make reductions as in the proof of the first bound in Section \ref{proofbound1}. \\
	
	\noi We can assume for the rest of the proof that $|\tau|,|\s_2| \le L^{\theta^{-1}}$ as otherwise, one can gain some power of $L$ thanks to inequalities \eqref{ligne1} and \eqref{ligne2} and close the estimate. \\
	\noi Indeed, if we localize in $|\tau|,|\s_2| >L^{\theta^{-1}}$, we can write :
	\begin{align*}
		\|\ind_{|\tau|\ge L^{\theta^{-1}}} (\mathcal{P}v\ind_{|\s_2|\ge L^{\theta^{-1}}})\|_{\Tilde{h}^{0,1-b}} &= L^{-d/2}\Bigg\|\ind_{|\tau|\ge L^{\theta^{-1}}}\jb{\tau}^{-b} \jb{\tau}\Ft_{t}\left(\mathcal{P}v\ind_{|\s_2|\ge L^{\theta^{-1}}}\right)(\tau,k)\bigg\|_{l_k^2L_\tau^2} \\
		&\le L^{-d/2} \Bigg\| \jb{\tau}\Ft_{t}\left(\mathcal{P}v\ind_{|\s_2|\ge L^{\theta^{-1}}}\right)(\tau,k)\Bigg\|_{L_\tau^\infty l_k^2}\Big\|\jb{\tau}^{-b}\ind_{|\tau|\ge L^{\theta^{-1}}}\Big\|_{L_\tau^2} \, .
	\end{align*}
	We integrate in $\tau$, inject the expression of the Fourier transform of the operator $\mathcal{P}$ \eqref{defoperatorP} and use the formula of operator $\chi$ \eqref{formulaofchi} to get :
	\begin{align*}
		\| \ind_{|\tau|\ge L^{\theta^{-1}}}&(\mathcal{P}v\ind_{|\s_2|\ge L^{\theta^{-1}}})\|_{\Tilde{h}^{0,1-b}} \les \frac{\al T_{max}}{L^{d/2}} L^{-\theta^{-1}(b-1/2)} \\
		&\times \Bigg\| \sum_{\substack{k_1,k_2 \in \Z_L^d \\ k_1+k_2=k}} \int_{\R^2}H(\tau,\s_1+\s_2-2T_{max}k_1\cdot k_2)\Ft_{t}(J_\mathcal{T})(\s_1,k_1)\Ft_{t}(v)(\s_2,k_2)\ind_{|\s_2|\ge L^{\theta^{-1}}}d\s_1 d\s_2\Bigg\|_{L_\tau^\infty l_k^2} \\
		&\les  \sup_{\tau,\s_2 \in \R}\|\chi (\tau,\s_2)  \|_{l_k^2\xrightarrow[]{}l_k^2} \|\Ft_{t}(v)(\s_2,k)\ind_{|\s_2|\ge L^{\theta^{-1}}}\|_{L_{\s_2}^1 l_k^2}L^{-\theta^{-1}(b-1/2)} \\
		& \les L^{-2\theta^{-1}(b-1/2)}  \sup_{\tau,\s_2 \in \R}\|\chi (\tau,\s_2)  \|_{l_k^2\xrightarrow[]{}l_k^2} \|v\|_{\tilde{h}^{0,b}} \, .
	\end{align*}
	We use rough bounds and obtain by Cauchy-Scwharz inequality as $s>d/2$ and $b>1/2$ :
	\begin{align*}
		\| \ind_{|\tau|\ge L^{\theta^{-1}}}(\mathcal{P}v\ind_{|\s_2|\ge L^{\theta^{-1}}})\|_{\Tilde{h}^{0,1-b}} &\les L^{-2\theta^{-1}(b-1/2)}  \|v\|_{\tilde{h}^{0,b}} \frac{\al T_{max}}{L^{d/2}} \sup_{\tau,\s_2 \in \R}\|\int_{\R}|\Ft_{t}(J_\mathcal{T})(\s_1,k_1)|d\s_1  \|_{l_k^1} \\
		& \les  L^{-2\theta^{-1}(b-1/2)}  \|v\|_{\tilde{h}^{0,b}} \frac{\al T_{max}}{L^{d/2}}L^{d/2} \|J_\mathcal{T} \|_{\Tilde{h}^{s,b}} \,.
	\end{align*}
	We finally use Proposition \ref{bound1} and the bound $T_{max}\le L^{C(\gamma,d)}$ to get :
	\begin{align*}
		\| \ind_{|\tau|\ge L^{\theta^{-1}}}(\mathcal{P}v\ind_{|\s_2|\ge L^{\theta^{-1}}})\|_{\Tilde{h}^{0,1-b}} &\les  L^{-2\theta^{-1}(b-1/2)}  \|v\|_{\tilde{h}^{0,b}}  L^{C(\gamma,d)} L^{\theta+C(b-1/2)}L^{-\delta} \,.
	\end{align*}
	Similar bounds hold true if only $|\tau|$ or $|\s_2|$ is greater than $L^{\theta^{-1}}$. These bounds are better than the ones we will obtain when $|\tau|,|\s_2|\le L^{\theta^{-1}}$ as $\theta >0$ is arbitrarily small. Thus, we only need to bound $ \sup_{\tau,\s_2 \in \R}\|\ind_{|\tau|\le L^{\theta^{-1}}}\ind_{|\s_2|\le L^{\theta^{-1}}}\chi (\tau,\s_2)  \|_{l_k^2\xrightarrow[]{}l_k^2}$ to close the argument. \\
	
	\noi To go from a point-wise L-certain estimate in $\tau,\s_2$ to a uniform L-certain estimate, we use the same technique as in Section \ref{proofbound1} and cut $[-L^{\theta^{-1}},L^{\theta^{-1}}]$ into smaller intervals, using the bounds on the differential of $H$, $\partial_1 H$ and $\partial_{2}H$ implied by the bounds on the differential of $I_1$ in Lemma \ref{FT}. \\
	\noi Therefore, one can fix $\tau$ and $\s_2$. We then set $$\zeta:=\tau-\s_2 \, .$$ We will make the change of variable $$\chi(\tau,s_2)=\chi(\tau,\zeta) \, .$$ \\
	Lemma \ref{lemmaK} also implies that $|\mathcal{K}_\mathcal{T}|\le L^\theta$ and $|\partial_\tau \mathcal{K}_\mathcal{T}|\le L^\theta T_{max}$. Indeed, one can write :
	\begin{align*}
		&\left|\mathcal{K}_\mathcal{T} (\tau,\s_2,(k-k'\cdot k'), \{k_l, l\neq r'\})\right| \\
		&= \left| \jb{T_{max}q_r-\zeta}^5\prod_{n\in \mathcal{N}-\{r\}} \jb{T_{max}q_n} \int_{\R}H(\tau,\tau_1+\s_2-2T_{max}k_1\cdot k_2) K_\mathcal{T}(\tau_1,k_n \ : \ n \in \mathcal{T})d \tau_1 \right| \\
		&\les \jb{T_{max}q_r-\zeta}^5 \int_\R \jb{\zeta-\tau_1+2T_{max}k_1\cdot k_2}^{-10}\jb{\tau_1-T_{max}q_{\text{root of $\mathcal{T}$}}}^{-10}d\tau_1  \, .
	\end{align*}
	We do a change of variables and estimate the integral :
	\begin{align*}
		&\left|\mathcal{K}_\mathcal{T} (\tau,\s_2,(k-k'\cdot k'), \{k_l, l\neq r'\})\right| \\
		&\hspace{0.5cm} \les \jb{T_{max}q_r-\zeta}^5 \int_\R \jb{\zeta-\tau_1-T_{max}q_{\text{root of $\mathcal{T}$}}+2T_{max}k_1\cdot k_2}^{-10}\jb{\tau_1}^{-10}d\tau_1  \\
		& \hspace{0.5cm}\les \jb{T_{max}q_r-\zeta}^5  \jb{\zeta-T_{max}q_{\text{root of $\mathcal{T}$}}+2T_{max}k_1\cdot k_2}^{-5} \\
		& \hspace{0.5cm}= \jb{T_{max}q_r-\zeta}^5  \jb{\zeta-T_{max}q_r}^{-5}\\
		&\hspace{0.5cm}\les L^\theta \, .
	\end{align*}
	
	\noi For $\partial_\tau \mathcal{K}_\mathcal{T}$, we proceed in a similar manner.  \\
	
	\noi We make further reductions on \eqref{formulaofchi} (similarly to Section \ref{proofbound1}). \\

	\noi We fix the integer part of $T_{max}q_n$ for $n\in \mathcal{N}-\{r\}$ and of $T_{max}q_r$ and assume them to be $\les L^{\theta^{-1}}$ by decay of $\jb{T_{max}q_r-\zeta}^{-5}$ and because of the bound $|\zeta|\les L^{\theta^{-1}}$. \\ We justify we can do this up to a small enough loss. We denote by $[x]$ the integer part of $x$. If $|[T_{max}q_r-\zeta] |> L^{\theta^{-1}}$, then we use Lemma \ref{hypercontractivity} with $A=L^\theta$, we denote by $\mathcal{T}'$ a copy of the tree $\mathcal{T}$ where the wave numbers of the leaves have been permuted by some $\s \in S_{n+1}$ ; we will use the following notations for the tree $\mathcal{T}'$ : $ \mathcal{N}'=:\{\s(n), \ n\in \mathcal{N}\}$, $ \mathcal{L}'=:\{\s(l), \ l\in \mathcal{L}\}$, $\s(r)$ its root, $\s(r')$ its placeholder leaf. \\
	\noi We get then that L-certainly for every $\zeta$ and $\tau$ :
	\begin{multline*}
		|\chi_{kk'}(\tau,\zeta)|^2 \les  \sum_{\s \in \mathcal{S}_{n+1}}\left(\frac{\al T_{max}}{L^{d/2}} \right)^{2(n+1)} \sum_{\substack{\nu \in \Z \\ |\nu| \ge L^{\theta^{-1}}}}\sum_{k_n, k_{\s(n)} \ \text{admissible}}\ind([T_{max}q_r-\zeta]=\nu)\\ \times \left|\mathcal{K}(\tau,\zeta,(k-k'\cdot k'), \{k_l, l\neq r'\})\right| 
		\left|\mathcal{K}(\tau,\zeta,(k-k'\cdot k'), \{k_{\s(l)}, l\neq r'\})\right|\frac{1}{\jb{T_{max}q_r-\zeta}^5} \\ \times \frac{1}{\jb{T_{max}q_{\s(r)}-\zeta}^5} \prod_{n\in \mathcal{N}-\{r\}} \frac{1}{\jb{T_{max}q_n}} \frac{1}{\jb{T_{max}q_{\s(n)}}} \, .
	\end{multline*}
	We use the indicator function $[T_{max}q_r-\zeta]=\nu$ and get this is smaller than :
	\begin{multline*}
		\sum_{\s \in \mathcal{S}_{n+1}}\left(\frac{\al T_{max}}{L^{d/2}} \right)^{2(n+1)} \sum_{\substack{\nu \in \Z \\ |\nu| \ge L^{\theta^{-1}}}}\frac{1}{\jb{\nu}^5}\sum_{k_n, k_{\s(n)} \ \text{admissible}}\ind([T_{max}q_r-\zeta]=\nu) \\
		\times \left|\mathcal{K}(\tau, \zeta,(k-k'\cdot k'), \{k_l, l\neq r'\})\right| \left|\mathcal{K}(\tau, \zeta,(k-k'\cdot k'), \{k_{\s(l)}, l\neq r'\})\right| \frac{1}{\jb{T_{max}q_{\s(r)}-\zeta}^5} \\
		\times \prod_{n\in \mathcal{N}-\{r\}} \frac{1}{\jb{T_{max}q_n}} \frac{1}{\jb{T_{max}q_{\s(n)}}} \, .
	\end{multline*}
	We sum on $\nu$ and bound the previous sum by :
	\begin{multline*}
		\sum_{\s \in \mathcal{S}_{n+1}}\left(\frac{\al T_{max}}{L^{d/2}} \right)^{2(n+1)} \log(L)\sup_{\nu \in \Z}\sum_{k_n, k_{\s(n)} \ \text{admissible}}\ind([T_{max}q_r-\zeta]=\nu) \\
		\times \left|\mathcal{K}(\tau,\zeta,(k-k'\cdot k'), \{k_l, l\neq r'\})\right| \left|\mathcal{K}(\tau,\zeta,(k-k'\cdot k'), \{k_{\s(l)}, l\neq r'\})\right|\frac{1}{\jb{T_{max}q_{\s(r)}-\zeta}^5} \\
		\times \prod_{n\in \mathcal{N}-\{r\}} \frac{1}{\jb{T_{max}q_n}}\prod_{n\in \mathcal{N}-\{r\}} \frac{1}{\jb{T_{max}q_{\s(n)}}} \, .
	\end{multline*}
	We have the same result if we impose $|[T_{max}q_n]|\ge L^{\theta^{-1}}$ for some $n \in \mathcal{N}-\{r\}$. Thus, the part of the operator where such conditions are imposed can be bound in the same way as the rest of the operator (see below) and using rough bounds, up to a loss of $\log(L)^c$. \\
	
	\noi If we try to bound $ \|\chi (\tau,\zeta)  \|_{l_k^2\xrightarrow[]{}l_k^2}$ directly, the bound will be too large for us to close the argument. So we use a $TT^*$ argument and try to bound $\|(\chi\chi^*)^D  \|_{ l_k^2\xrightarrow[]{}l_k^2}$ to some high power $D\in \N$ where we omit the dependence on $\tau$ and $\zeta$. \\
	
	\begin{remark}
		One could also try to bound the spectral radius $\rho(\chi)$ but this would not work with the current structure of the proof as using this information would require the cutoff in the Feynmann tree expansions $N$ to depend on $L$.
	\end{remark}
	
	\noi Now we compute the adjoint of $\chi$. Let $w\in l_k^2$.
	\begin{align*}
		\left( \chi v,w\right)&=\sum_{k\in \Z_L^d}\frac{\al T_{max}}{L^{d/2}}\sum_{\substack{k_1,k_2 \in \Z_L^d\\k_1+k_2=k}}v_{k_2}(\s_2) \int_{\R}H(\tau,\s_1+\s_2-2T_{max}k_1\cdot k_2)\Ft_{t}(J_\mathcal{T})(\s_1,k_1)d\s_1 \cj{w_k} \\
		&=\sum_{k_2\in \Z_L^d}v_{k_2}(\s_2)\cj{\frac{\al T_{max}}{L^{d/2}}\sum_{\substack{k,k_1 \in \Z_L^d \\k_1+k_2=k}} \int_{\R}\cj{H(\tau,\s_1+\s_2-2T_{max}k_1\cdot k_2)\Ft_{t}(J_\mathcal{T})(\s_1,k_1)}d\s_1 w_k} \, .
	\end{align*}
	So changing the indexation in the sum :
	\begin{align}
		(\chi^*w)(k') &= \frac{\al T_{max}}{L^{d/2}}\sum_{\substack{k,k_1 \in \Z_L^d \\k'=k-k_1}} \int_{\R}\cj{H(\tau,\s_1+\s_2-2T_{max}k_1\cdot (k-k_1))\Ft_{t}(J_\mathcal{T})(\s_1,k_1)}d\s_1 w_k \\ \label{chistar}
		& =\frac{\al T_{max}}{L^{d/2}}\sum_{\substack{k,k_1 \in \Z_L^d \\k'=k-k_1}} \int_{\R}\cj{H(\tau,\s_1+\s_2-2T_{max}k'\cdot (k-k'))\Ft_{t}(J_\mathcal{T})(\s_1,k_1)}d\s_1 w_k \, . \\ \notag 
	\end{align}
	We want to compute $(\chi \chi^*)^{D}$. 
	
	\noi We recall that $(q_n)_{n\in \mathcal 
		N}$ and $(\O_n)_{n\in \mathcal{N}}$ are linear combinations one of the other. Thus we can assume we have fixed $ (\s_n)_{n\in \mathcal{N}} \in \R^\mathcal{N}$ and look at the restriction on the sum when $|\O_n-\sigma_n|\le T_{max}^{-1}$ for all $n\in \mathcal{N}$ and $\s_n~=~O(L^{\theta^{-1}})$ (this last assumption, which was not present in Section \ref{proofbound1}, is needed to apply Lemma \ref{test} below). We henceforth assume the operator $\chi$  takes into account these restrictions.\\
	\noi We will see that the operator $(\chi\chi^*)^D$ corresponds to a concatenation of trees. However, because of the complex conjugate in the formula of $\chi^*$ \eqref{chistar}, some leaves corresponding to complex Gaussians will be conjugated. We will use the notation $z^+:=z$ and $z^-:=\cj{z}$ for $z\in \C$. \\
	\noi To illustrate this, we choose to represent the tree corresponding to $(\chi \chi^*)^2$  as an example : 
	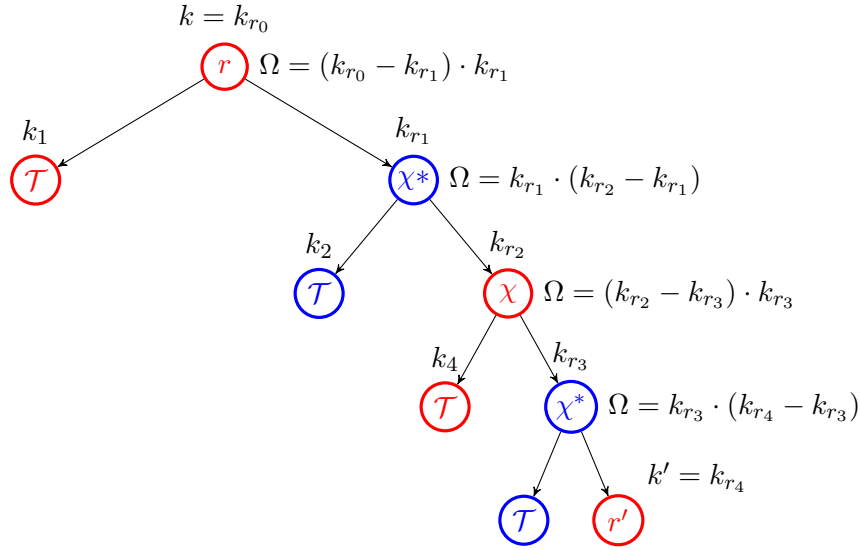
\begin{figure}[H]
		\centering
		\begin{tikzpicture} [->,>=stealth',level/.style={sibling distance = 5cm/#1,
				level distance = 1.5cm}] 
			\node[arn_r, label=90:{$k=k_{r_0}$}, label=0:{$\O=(k_{r_0}-k_{r_1})\cdot k_{r_1}$}] {$r$} 
			child{ node [arn_r, label=90:{$k_1$}] {$\mathcal{T}$} } 
			child{node [arn_b,label=0:{$\Omega=k_{r_1}\cdot(k_{r_2}-k_{r_1})$},label=90:{$k_{r_1}$}] {$\chi*$}
				child{node [arn_b,label=90:{$k_2$}] {$\mathcal{T}$}}
				child{node [arn_r, label=0:{$\Omega=(k_{r_2}-k_{r_3})\cdot k_{r_3}$},label=90:{$k_{r_2}$}] {$\chi$}
					child{node [arn_r, label=90:{$k_4$}] {$\mathcal{T}$}}
					child{node [arn_b, label=90:{$k_{r_3}$}, label=0:{$\O=k_{r_3}\cdot(k_{r_4}-k_{r_3})$}] {$\chi^*$}
						child{node [arn_b] {$\mathcal{T}$}}
						child{node [arn_r, label=45:{$k'=k_{r_4}$}] {$r'$}}
					} 
				}
			}
			;
		\end{tikzpicture}
		\caption{Tree representation of $(\chi\chi^*)^2$. Colours correspond to different kind of nodes : blue for $\chi^*$ and red for $\chi$, each having a different resonance relation represented by $\Omega$. The trees in blue have conjugated leaves.}
		\label{tree}
	\end{figure}
	\noi In the end, the operator depends on :
	\begin{align*}
		K_{set} :=\left\{ (k_{r_{2j}}-k_{r_{2j+1}})\cdot k_{r_{2j+1}}, \ k_{r_{2j+1}}\cdot (k_{r_{2j+2}}-k_{r_{2j+1}}) \ \ for \ 0 \le j \le D-1 \right\}
	\end{align*}
	and so we write :
	\begin{align*}
		\Big( (\chi\chi^*)^D(\tau,\zeta)\Big)_{kk'}=\sum_{k_n \ \text{admissible}, \ n \in \mathcal{T}^D}\mathcal{K}^{(D)}(\tau,\zeta,K_{set},\ \{k_l, \ l\neq r'\})\left(\frac{\al T_{max}}{L^{d/2}}\right)^{2D(n+1)}\prod_{\substack{l\in \mathcal{L}^D\\ l\neq r'}}g_{Lk_l}^{i_l}
	\end{align*}
	where $\mathcal{T}^{D}$ is the tree formed on the model of Figure \ref{tree} with $D$ concatenations and leaves in the blue trees conjugated (i.e. $i_l=-$ iff $l\in \text{blue tree}$), $\mathcal{N}^D$ its set of branching nodes, $\mathcal{L}^D$ its set of leaves and with decoration verifying : $|k_l|\le L^\theta$ for $l\neq r'$, $|\O_{n}-\s_{n}|\le T_{max}^{-1}$ and $\s_n = O(L^{\theta^{-1}})$ ; where $\mathcal{K}^{(D)}$ is some function with $|\mathcal{K}^{(D)}|\le L^\theta$, $|\partial\mathcal{K}^{(D)}|\le L^\theta T_{max}$ (these properties are a direct consequence of the bounds on $\mathcal{K}_\mathcal{T}$) and where the restrictions on $\chi$ still apply ; where $k_r=k$ and $k_{r'}=k'$. \\
	
	\noi We use the fact that $k'+\sum_{l\in \mathcal{L}^D}k_l=k$ in the tree $\mathcal{T}^D$. First, let $\mathcal{O}$ be an operator with the same support as $(\chi\chi^*)^D$. By Schur test estimate :
	\begin{align*}
		\|\mathcal{O}\|_{l^2\xrightarrow[]{}l^2}^2 &\les \sup_{k'}(\sum_k |\mathcal{O}_{kk'}|)\sup_{k}(\sum_{k'} |\mathcal{O}_{kk'}|)  \, .
	\end{align*}
	And for $k$ (respectively $k'$) fixed there are at most $L^{d+\theta}$ values for $k'$ (respectively $k$). This yields :
	\begin{align*}
		\|\mathcal{O}\|_{l^2\xrightarrow[]{}l^2}^2& \les L^{d+\theta} \sup_{k,k'}(|\mathcal{O}_{kk'}|)L^{d+\theta} \sup_{k,k'}(|\mathcal{O}_{kk'}|) \, .
	\end{align*} 
	
	\noi Hence :
	\begin{align*}
		\|\mathcal{O}\|_{l^2\xrightarrow[]{}l^2} \les L^{d+\theta}\|\mathcal{O}\|_{l^\infty} \, .
	\end{align*}
	So we can bound :
	\begin{align}
		\|(\chi\chi^*)^D\|_{l^2\xrightarrow{}l^2} \les L^{d+\theta}\|(\chi\chi^*)^D\|_{l^\infty} \label{Schur} \, .
	\end{align}
	Now we only to prove bounds $L$-certainly and uniformly in $k$ and $k'$. Let us fix $k$ and $k'$. \\
	\noi In order to bound $ \left|(\chi\chi^*)_{kk'}^D\right|^2 $, we write a conjugated copy of $\mathcal{T}^D$, which we denote by $\mathcal{T}^{D-}$. We will use the notations for the tree $\mathcal{T}^{D-}$ : $\mathcal{N}^{D-}$, $\mathcal{L}^{D-}$, $r^-$ root, $r'^-$ placeholder. This tree will be used to represent the conjugate term when we compute the $L^2(\O)$ norm. This copy has all the same constraints for its decoration, as it copies the structure of $\mathcal{T}^D$, its only difference being the conjugation i.e. $i_{l^-}=-i_l$ for all $l\in \mathcal{L}^D$. \\
	\noi Just as in Section \ref{proofbound1}, we will need to make pairs between conjugated and non-conjugated leaves to compute the norm. We denote by $\s \in Bij\big((\mathcal{L}^D-\{r'\} )\cup(\mathcal{L}^{D-}-\{r'^-\})\big)$ the bijective involution which to each leaf in $(\mathcal{L}^D-\{r'\} )\cup(\mathcal{L}^{D-}-\{r'^-\})$ gives its paired leaf meaning $ i_{\s(l)}=-i_l$ and $k_l=k_{\s(l)}$. We will call this condition the $\s$-pairing condition. We denote by $S^D$ the set of bijections verifying these properties.    \\
	\noi We use a variant of Lemma \ref{hypercontractivity} with $A=L^\theta$ and notice there are at most $(2(n+1)D)!$ possible involutions in $S^D$, and get L-certainly :
	\begin{multline*}
		\left|(\chi\chi^*)_{kk'}^D\right|^2 \les L^\theta  \left( \frac{\al T_{max}}{L^{d/2}}\right)^{4D(n+1)}  \sum_{\s \in S^D} \sum_{k_n,  \ n\in \mathcal{T}^D\cup\mathcal{T}^{D-}}\ind (k_n \ \text{admissible}) \ind(\s-\text{pairing condition}) \, .
	\end{multline*}
	\noi First we notice that the counting problems associated to the roots of the $\chi^*$ operators in the tree representation of Figure \ref{tree} (labeled by $k_{r_{2j-1}}$ for $1\le j \le D$) are different as their associated resonance relation is different. In particular, if for instance $k_{r_1}=0$ in Figure \ref{tree}, one cannot hope to gain anything from the condition $|\O-\s|\le T_{max}^{-1}$ for some $\s \in \R$. Not only that, but $\O$ is not quadratic in $k_{r_2}$ which means we cannot use the divisor method previously used in Section \ref{proofbound1}. This comes from the fact that these counting problems are in a way associated to the dispersion relation of equation $i\dt u -\Dl u=|u|^2$. \\
	
	\noi We will use the following counting lemma :
	\begin{lemma} \label{countinglemmasimple}
		Let $k\in \Z^d-\{0\}$, $\s\in\R$ then :
		\begin{equation*}
			\#\{N\in \Z^d, \ |N|\le L^{1+\theta}, \ |k\cdot(k-N)-\s|\le L^{\theta}L^2T_{max}^{-1}\} \le \begin{cases}
				L^\theta L^{d+1}T_{max}^{-1}\|k\|_\infty^{-1} \ \text{if} \ \|k\|_\infty \ge LT_{max}^{-1}\, . \\
				L^\theta L^d \ \text{otherwise}  \, .
			\end{cases}
		\end{equation*}
	\end{lemma}
	\begin{proof}
		If $ \|k\|_\infty \le LT_{max}^{-1} $, we do not use the condition on the dispersion relation and the proof is straightforward. \\
		\noi Otherwise, one can assume $k_d=\|k\|_{\infty} \neq 0$ up to permuting the coordinates as $k\neq 0$.\\
		We fix the first $d-1$ coordinates and then use the condition $$|k_d(k_d-N_d)-\s'|\le L^2T_{max}^{-1}$$ to get a bound of size $$\les L^2T_{max}^{-1}|k_d|^{-1}$$ on the last coordinate.
	\end{proof}
	
	\noi We also need to use a variant of the counting lemma previously used in Section \ref{proofbound1} as one difference is that $k$ and $k'$ are not bounded a priori. This prevents us from using the same divisor bound estimate, as the estimate obtained would not be uniform in $k$ and $k'$. Instead, we use the following lemma, taken as Lemma 4.4 from \cite{countinglemmapaper} :
	\begin{lemma} \label{countinglemmahard}
		Let $R=\Z$ or $\Z[i]$, $c_1,c_2 \in \C$, $r\in R-\{0\}$, $B_1,B_2 \ge 0$. Then we have for $\nu>0$~~:
		\begin{equation*}
			\# \left\{ N_1,N_2 \in R , \ r=N_1N_2, \ |N_1-c_1|\le B_1, |N_2-c_2|\le B_2\right\} \le C(\nu)B_1^\nu B_2^\nu \, .
		\end{equation*}
	\end{lemma}
	\begin{proof}See the proof of Lemma 4.4 from \cite{countinglemmapaper}. \\
	\end{proof}
	Applying this lemma with $R=\C$, to our case, we get the following estimate for some $r\in \Z$, $r\neq 0$ (the estimate is trivial for $r=0$) :
	\begin{multline*}
		\# \left\{ N_1, N_2 \in \Z, \ -(2N_1-k_1)^2-(2N_2-k_2)^2=r, \ |N_1|,|N_2|\le L^{1+\theta} \right\} \\\le \# \left\{ N_1, N_2 \in \Z, \ N_1^2+N_2^2=r, \ |N_1+k_1|,|N_2+k_2|\le L^{1+\theta}  \right\}\les_\nu L^{(1+\theta)2\nu}
	\end{multline*}
	for any $\nu>0$ small. \\
	\noi One notices that the estimate from Lemma \ref{countinglemmahard} is better than the estimate from Lemma \ref{countinglemmasimple} which is good news as they are more nodes corresponding to counting problems associated to Lemma \ref{countinglemmahard} than to Lemma \ref{countinglemmasimple}. \\
	
	\noi We first write :
	\begin{multline*}
		\sum_{\s \in S^D} \sum_{k_n, \ n\in \mathcal{T}^D\cup\mathcal{T}^{D-}}\ind (k_n \ \text{admissible}) \ind(\s-\text{pairing  condition}) \\\les (2(n+1)D)!\sup_{\s \in S^D}\sum_{k_n, \ n\in \mathcal{T}^D\cup\mathcal{T}^{D-}}\ind (k_n \ \text{admissible}) \ind(\s-\text{pairing condition}) \, .
	\end{multline*}
	Let us fix $\s \in S^D$. We will proceed as in Section \ref{proofbound1}. One difference is that self-coupling is now possible i.e. some leaves in the tree $\mathcal{T}^D$ might be paired together. Let us denote by $0\le p\le 2D(n+1)$ the amount of pairs within the tree $\mathcal{T}^D$. Since all leaves must be paired, this means that $2D(n+1)-2p$ leaves in $\mathcal{T}^{D-}$ are paired with leaves in $\mathcal{T}^D$ and therefore that there are $p$ pairs in $\mathcal{T}^{D-}$ as well. \\
	\noi We will use Lemma \ref{countinglemmahard} and Lemma \ref{countinglemmasimple} depending on the situation. We can use Lemma \ref{countinglemmahard} whenever the counting problem is associated to a node within a copy of the tree $\mathcal{T}$ or the root of a $\chi$ operator (see Figure \ref{tree}). We will use Lemma \ref{countinglemmasimple} for counting problems associated to the roots of the $\chi^*$ operator in Figure \ref{tree}. Finally, one notices that the leaf $k'$ corresponding to the placeholder is fixed which means one can avoid performing the last $\chi^*$ related counting problem, i.e. the one labeled by $k_{r_{2D-1}}$.\\
	\noi We will proceed by using an algorithm that goes down the tree $\mathcal{T}^D$, doing counting estimates on the left hand side children which are copies of $\mathcal{T}$ first. The algorithm will then fix the values of the leaves paired to the leaves present in the child tree treated by the algorithm each time. The key fact is that if $r$ leaves are fixed for a tree with $n$ nodes, then one has to do only $n-r$ counting steps of the algorithm. \\
	Let us denote the global counting problem by :
	\begin{align*}
		M:= \sum_{k_n, \ n\in \mathcal{T}^D\cup\mathcal{T}^{D-}}\ind (k_n \ \text{admissible}) \ind(\s-\text{pairing condition}) \, .
	\end{align*}
	For a generic tree $\mathcal{T}_0$, we will denote by 
	\begin{align*}
		M_{\mathcal{T}_0}
	\end{align*}
	the counting problem associated to the tree $\mathcal{T}_0$. If furthermore, we impose a set of conditions $E$ on the counting, we will denote its result by
	$$M_{\mathcal{T}_0,E} \, .$$
	We apply our algorithm which will first treat the root $r$ of the tree $\mathcal{T}^D$, then treat the nodes of the copy of the tree $\mathcal{T}$ which is the left hand child of the root $r$. Then we will treat the root of the operator $\chi^*$ which is the right hand child of the root $r$ and then the nodes of the copy of the tree which is the left hand child of the root of $\chi^*$. Then we can repeat this process going down the tree $\mathcal{T}^D$. Each time we treat a counting problem associated with some leaves, we fix the leaves paired with the leaves we have treated and keep in mind these conditions for the rest of the counting. \\
	Now let us consider that we have done our algorithm up to the root of a tree $\mathcal{T}_0$ with some conditions imposed we will not write for the sake of brevity. Then if $\mathcal{T}_0$ has children $\mathcal{T}_1$ and $\mathcal{T}_2$ (left and right hand respectively) with respective labels for their roots $k_0$, $k_1$ and $k_2$ and with $p$ pairs between the leaves of $\mathcal{T}_1$ and $\mathcal{T}_2$, three cases arise. \\
	\begin{itemize}

		\item If the root is not the root of a $\chi^*$ operator, we use Lemma \ref{countinglemmahard}. We get :
		\begin{equation}\label{firstcasebound}
			M_{\mathcal{T}_{0}}\le L^\theta L^{d}T_{max}^{-1}\sup_{\substack{k_1}}M_{\mathcal{T}_1,\text{root is indexed by }k_1}\sup_{k_2, \ p \ \text{fixed leaves}}M_{\mathcal{T}_2,\text{root is indexed by }k_2, p \text{ fixed leaves}} \, .
		\end{equation}
		
		\item If the root is the root of  a $\chi^*$ operator i.e. is labeled by $k_{r_{2j-1}} \neq 0$ for $1\le j\le D$  then it corresponds to the counting problem treated in Lemma \ref{countinglemmasimple}. We get :
		\begin{multline*}
			M_{\mathcal{T}_0}\le L^\theta \min(L^d,L^{d+1}T_{max}^{-1}\|k_{r_{2j-1}}\|_\infty^{-1})\sup_{\substack{k_1}}M_{\mathcal{T}_1, \text{root is indexed by }k_1} \\
			\times \sup_{k_2, \ p \ \text{fixed leaves}}M_{\mathcal{T}_2,\text{root is indexed by }k_2,p \ \text{fixed leaves}} \, .
		\end{multline*}
		Now we can use the fact that every root of a $\chi^*$ operator is the child of a $\chi$ operator. From the first case \eqref{firstcasebound}, we know $k_{r_{2j-1}}$ can take up to $L^\theta L^d T_{max}^{-1}$ values. The worst case is when these values are the smallest. We denote by $M_{\mathcal{T}_{-1}}$ the counting problem associated to the parent of the root of $\chi^*$ and $M_{\mathcal{T}_{0'}}$ the counting problem associated to the left hand sibling of the root of operator $\chi^*$. Thus, we get, omitting the conditions relative to the fixed leaves :
		\begin{align*}
			M_{\mathcal{T}_{-1}}\les \sup_{k_{0'}}M_{\mathcal{T}_{0'}, \text{root is indexed by $k_{0'}$}}\sum_{\kappa=1}^{L^{\theta/d}LT_{max}^{-1/d}}\sum_{\kappa_2,...,\kappa_d=1}^{|\kappa|}M_{\mathcal{T}_0, (k_{r_{2j-1}}=(\kappa,\kappa_2,...,\kappa_d))} \, .
		\end{align*}
		Hence we get that the double sum is less than :
		\begin{multline*}
			\sum_{\kappa=1}^{L^{\theta/d}LT_{max}^{-1/d}}\sum_{\kappa_2,...,\kappa_d=1}^{|\kappa|}L^\theta \min(L^d,L^{d+1}T_{max}^{-1}\|k_{r_{2j-1}}\|_\infty^{-1})\sup_{\substack{k_1}}M_{\mathcal{T}_1, \text{root is indexed by $k_1$}} \\\times\sup_{k_2,\ p \ \text{fixed leaves}}M_{\mathcal{T}_2,\text{root is indexed by $k_2$}} \, .
		\end{multline*}
		We factorize by the terms depending on $\mathcal{T}_1$ and $\mathcal{T}_2$ and look at the remaining double sum :
		\begin{align*}
			\sum_{\kappa=1}^{L^{\theta/d}LT_{max}^{-1/d}}\sum_{\kappa_2,...,\kappa_d=1}^{|\kappa|}L^\theta \min(L^d,L^{d+1}T_{max}^{-1}\|k_{r_{2j-1}}\|_\infty^{-1})& = \sum_{\kappa=1}^{L^{\theta/d}LT_{max}^{-1/d}}|\kappa|^{d-1} \min(L^d,L^{d+1}T_{max}^{-1}\kappa^{-1}) \\
			& \les L^\theta(L^{2d}T_{max}^{-d}+L^{2d}T_{max}^{-2+1/d}) \\& \les L^\theta L^{2d}T_{max}^{-2+1/d} \, .
		\end{align*}
		Hence, we get :
		\begin{multline*}
			M_{\mathcal{T}_{-1}}\les \sup_{k_{0'}}M_{\mathcal{T}_{0'}, \text{root is indexed by $k_{0'}$}}\sup_{\substack{k_1}}M_{\mathcal{T}_1, \text{root is indexed by $k_1$}} \\
			\times \sup_{k_2,\ p \ \text{fixed leaves}}M_{\mathcal{T}_2,\text{root is indexed by $k_2$}}L^\theta L^{2d}T_{max}^{-2+1/d}  \, .
		\end{multline*}
		And recalling the bound in the first case \eqref{firstcasebound}, it is as if the counting problem of $\chi^*$ had yielded the bound
		\begin{equation} \label{secondcasebound}
			L^dT_{max}^{-1+1/d} \, .
		\end{equation}
		For the rest of the counting, we can do all our computations as if this were the case. 
		
		\begin{remark}\label{remarkk-1}
			The use of $\|k\|_{\infty}^{-1}$ in Lemma\ref{countinglemmasimple} to improve the trivial counting is crucial. It allows us to exploit the good divisor bound of the counting problem associated to the parent node of the root of $\chi^*$ to get a smaller bound instead of using a trivial bound on the root of $\chi^*$ directly.
		\end{remark}

		\item Another case is when the root is the root of a $\chi^*$ operator is labeled by $0$. Recall that $1~\le~ T_{max}~ \le~ L^{\frac{d}{2-1/d}}$. We want to use Lemma \ref{countinglemmasimple} but in this case, one can only get the trivial bound $L^{d+\theta}$. However, one can split the whole counting problem for the whole tree $\mathcal{T}^D$ in two problems : one where one (or more) labels of roots of $\chi^*$ operators are set to $0$ and one where they are not $0$. \\
		In the case where they are not $0$, one can refer to the two cases above. \\
		In the case where at least one of these labels is $0$, let us say $k_{r_1}=0$ for simplicity, one has that the first counting problem of the tree $\mathcal{T}^D$ is trivial as two labels are fixed among the three in the counting problem. We will denote by $M_\mathcal{T}$ the counting problem associated to the tree $\mathcal{T}$ linked to the left hand side child of the root in Figure \ref{tree} and $M_{r_1}$ the counting problem associated to the right hand side child of the root with fixed label for the root $0$. Thus, one gets :
		\begin{equation*}
			M\le 1 \sup_{\text{root of }\mathcal{T} } M_{\mathcal{T}, k_{r_1}=0}\sup_{p \ \text{fixed leaves}}M_{r_{1}}
		\end{equation*}
		Now one can perform the counting from Lemma \ref{countinglemmasimple} on the root of the right hand side child and get the trivial bound $L^{d+\theta}$, i.e. one has taking similar notations and denoting by $p'$ the number of pairs between the left hand and right hand children of $r_1$ and $p_1$ the amount of leaves among the $p$ fixed that were on the left hand side child of $r_1$ (i.e. the blue tree), setting $p_2:=p-p_1$  :
		\begin{equation*}
			M\le L^{d+\theta}\sup_{\text{root of }\mathcal{T}, \ k_{r_1}=0} M_{\mathcal{T}} \sup_{\text{root of blue }\mathcal{T}, \ p_1 \ \text{fixed leaves}} M_{blue \ \mathcal{T}}\sup_{k_{r_2}, \ p_2+p' \ \text{fixed leaves}}M_{r_{2}}
		\end{equation*}
		where $M_{blue \ \mathcal{T}}$ is the counting problem of the left hand side child of $r_1$ (i.e. the rest of the blue in tree in Figure \ref{tree}) and $M_{r_2}$ is the counting problem of the right hand side child of $r_1$.\\
		Without the condition $k_{r_1}=0$, one would get :
		\begin{equation*}
			M\le L^{d+\theta}T_{max}^{-1}L^{d+\theta}T_{max}^{-1+1/d}\sup_{\text{root of }\mathcal{T}} M_{\mathcal{T}} \sup_{\text{root of blue }\mathcal{T}, \ p_1 \ \text{fixed leaves}} M_{blue \ \mathcal{T}}\sup_{k_{r_2}, \ p_2+p' \ \text{fixed leaves}}M_{r_{2}}
		\end{equation*}
		but one notices that :
		\begin{equation*}
			L^d \le L^{2d}T_{max}^{-2+1/d} \iff T_{max}\le L^{\frac{d}{2-1/d}}
		\end{equation*}
		which is verified. This reasoning can be applied to any root of a $\chi^*$ tree at any step of the overall counting algorithm as they are always a child of a root of a $\chi$ tree. This means that the worst case is when none of the roots of the $\chi^*$ operators are $0$.

	\end{itemize}

	\noi Now we just have to assess how many times we have to apply each counting estimate. \\
	
	\noi First we justify that if a tree $\mathcal{T}$ has $p$ fixed leaves, $n$ branching nodes and a fixed root, then one has to perform $n-p$ counting lemmas to count its admissible decorations. To prove this we consider the following algorithms. \\
	Whenever the two leaves of a node are fixed, we erase the two leaves and transform the node into a fixed leaf (whose wave number is determined by the wave numbers of its two children leaves). This transformation does not change the counting problem and diminishes by one the total amount of fixed leaves and branching nodes. We do this operation as long as there are such fixed pairs of leaves. Let us say we have done this operation $w$ times. Then one has in the end a tree with no pairs of fixed leaves, with $n-w$ branching nodes and $p-w$ fixed leaves. For such trees with no pairs, one proceeds by induction, the proof being trivial for size 1 trees and propagating to larger trees. \\
	
	\noi Now let us go back to the estimate of $M$. We first apply our counting algorithms going down from the root to the nodes $r_i$ as in Figure \ref{tree}. We first count the left hand size child trees of these nodes, each time fixing more leaves on the right hand size until the counting potentially becomes trivial on the right of $\mathcal{T}^D$. In total, for each pair of leaves, one of the leaf will eventually be fixed by the algorithm (the one the most on the right in the pair), once its paired leaf has been counted. This means that the sum of counting estimates avoided by the pairs is equal to $p$. Now these avoided counting estimates can concern regular nodes or roots of $\chi^*$ operators. However, the worst case is when they concern regular  branching nodes since we can apply better counting estimates to them. This means one will apply in the worst case the estimate \eqref{secondcasebound} $D-1$ times for each node of $\chi^*$ trees (i.e. blue tress) except the last and the estimate \eqref{firstcasebound} $2DN+D-p$ times. This concludes the counting on $\mathcal{T}^D$. \\
	Now remains the counting on $\mathcal{T}^{D-}$ but recall that all of its leaves with pairs in $\mathcal{T}^D$ are now fixed and it has $p$ pairs. Applying the same algorithm as for $\mathcal{T}^D$, we get that we have to apply estimate \eqref{secondcasebound} $D-1$ times and estimate \eqref{firstcasebound} $2nD+D-(2D(n+1)-2p)-p=p-D$ if $p\ge D$ times. Otherwise, one applies estimate \eqref{secondcasebound} $p$ times if $p\le D$.
	\noi In the end, we have~~: 
	\begin{align*}
		M&\le L^\theta (L^dT_{max}^{-1+1/d})^D (L^dT_{max}^{-1})^{2nD+D-p} \\
		& \hspace{2.5cm} \times \Big(\ind(p< D)(L^dT_{max}^{-1+1/d})^{p}+\ind(p\ge D)(L^dT_{max}^{-1})^{p-D}(L^dT_{max}^{-1+1/d})^D\Big)  \, .
	\end{align*}
	We simplify the terms between the parenthesis and get
	\begin{align*}
		M&\le L^\theta (L^dT_{max}^{-1+1/d})^D (L^dT_{max}^{-1})^{2nD+D}\Big(\ind(p<D)(T_{max}^{1/d})^{p}+ \ind(p\ge D)(T_{max}^{1/d})^D\Big) \\
		& \le L^\theta (T_{max}^{1/d})^D (L^dT_{max}^{-1})^{2(n+1)D}\Big(\ind(p<D)(T_{max}^{1/d})^{D}+ \ind(p\ge D)(T_{max}^{1/d})^D\Big) \\
		&=L^\theta (L^dT_{max}^{-1})^{2(n+1)D}(T_{max}^{1/d})^{2D}\, . 
	\end{align*}
	
	\noi Hence :
	\begin{align*}
		\left|(\chi\chi^*)_{kk'}^D\right|^2 &\les L^\theta(2(n+1)D)! \left( \frac{\al T_{max}}{L^{d/2}}\right)^{4D(n+1)}  (L^dT_{max}^{-1})^{2(n+1)D}(T_{max}^{1/d})^{2D} \\
		&= L^\theta(2(n+1)D)! (\al T_{max}^{1/2})^{4D(n+1)}(T_{max}^{1/d})^{2D} \, .
	\end{align*}
	Here, we notice that because of the affine estimates on the nodes of $\chi^*$, the worst case is when the tree $\mathcal{T}$ is the smallest i.e. $n=0$. To get a geometric bound in $n$, we have to write :
	\begin{align*}
		\left|(\chi\chi^*)_{kk'}^D\right|^2& \le L^{\theta} (2(n+1)D)! (\al T_{max}^{1/2}(T_{max}^{1/d})^{\frac{1}{2(n+1)}})^{4D(n+1)} \\
		&\le L^\theta (2(n+1)D)! (\al T_{max}^{1/2}(T_{max}^{1/d})^{1/2})^{4D(n+1)} \\
		&=L^\theta (2(n+1)D)! \rho^{4D(n+1)} \, .
	\end{align*}
	

	\noi But now we have to justify that this is L-certain uniformly in $k$ and $k'$ which requires us, a priori, to remove an infinite number of L-negligible probability sets. We will use Lemma \ref{test} below to go back to a finite controlled number of removed sets. We notice that $\mathcal{K}^{(D)}$ only depends on $k$ and $k'$ through the expressions $(k-k_{r_1})\cdot k_{r_1}$ and $k_{r_{2D-1}}\cdot (k'-k_{r_{2D-1}})$.
	Now we state the aforementioned Lemma taken as Claim 3.7  from \cite{Deng_2021} :
	\begin{lemma} \label{test}
		Define for $k \in \_L^d$ :
		\begin{equation*}
			f_{(k)}:m \in \Z_L^d \mapsto m\cdot (k+ m) \in \Z_{L^2}^d
		\end{equation*}
		and
		\begin{align*}
			Dom(f_{(k)}):= \left\{ m \in \Z_L^d, \ |m|\le L^\theta , \ |f_{(k)}(m)| \le L^{\theta^{-1}} \right\} \, .
		\end{align*}
		Then there exists $C>0$, $A\in \N$ and finitely many functions $f_1, ...,f_A\in\{f_{(k)},\ k\in \Z_L^d\}$ with $A \le L^{C\theta^{-1}}$ such that :
		\begin{align*}
			\forall \, k \in \Z_L^d, \ \exists \, 1 \le j \le A \ s.t \ |f_{(k)}-f_j| \le L^{-\theta^{-1}} \ \text{on $Dom(f_{(k)})$}.
		\end{align*}
	\end{lemma}
	\begin{proof} See Claim 3.7 from \cite{Deng_2021}.

	\end{proof}
	\noi This lemma allows us to get uniformity in $k$ and $k'$ up to losing $L^{C\theta^{-1}}$ $L$-negligible sets. We thus get in the end L-certainly :
	\begin{align*}
		\|(\chi\chi^*)^D\|_{l^2\xrightarrow{}l^2} &\les L^{(d+\theta)} L^\theta(2(n+1)D)!^{1/2}  \rho^{2D(n+1)} \, . \\
	\end{align*}
	So :
	\begin{align*}
		\|\chi\|_{l^2\xrightarrow{}l^2} &\les L^{\frac{d+\theta}{2D}} (2(n+1)D)!^{1/2D} \rho^{(n+1)}  \\
		& \le L^{\frac{d+\theta}{2D}} (4ND)^N \rho^{(n+1)} 
	\end{align*}
	hence we get our result taking $D$ large enough and then $L$ large enough. \\
	
	\begin{remark}
		\noi One can find an alternative proof of Theorem \ref{theorem1} by proceeding similarly to Section \ref{proofbound1}, as self-coupling does not change the amount of counting we have to do, and applying $4D(n+1)$ times the trivial counting bound $L^\theta L^d$. This yields :
		\begin{equation} \label{skip2}
			M \le L^\theta L^{d4D(n+1)}
		\end{equation}
		hence :
		\begin{align*}
			\left|(\chi\chi^*)_{kk'}^D\right|^2 &\les L^\theta(2(n+1)D)! \left( \frac{\al T_{max}}{L^{d/2}}\right)^{4D(n+1)} (L^d)^{2(n+1)D} \\
			&= L^\theta(2(n+1)D)! (\al T_{max})^{4D(n+1)} \\
			& \le L^\theta(2(n+1)D)! \rho^{4D(n+1)} \, ,
		\end{align*}
		which gives the same result as the Bourgain analysis.
	\end{remark}
	
	\section{Sum to integral}
	\noi This section aims at proving the following lemma :
	\begin{lemma} \label{sumtointegral}
		Let $B(x)\in \mathcal{S}(\R), \ g(x):=\left|\frac{\sin(x)}{x}\right|^2$, $\gamma \in (0,1)$, $d\ge 5$, $0\le t\le L^{2\gamma}$, $\Tilde{k}~\in~\Z^d$ fixed and define $k:=\Tilde{k}/L$ and :
		\begin{align*}
			S_k&:=\frac{t}{L^d}\sum_{x\in \Z_L^d}B(x)g(tx\cdot(k-x)) \\
			I_k &:= t \int_{\R^d}B(x)g(tx\cdot(k-x))dx \, .
		\end{align*}
		Then we have :
		\begin{align*}
			S_k=I_k+o_L(1) \, .
		\end{align*}
	\end{lemma}
	\begin{proof} We set $\delta,  \eps >0$, two small parameters we will determine later.
		We use a smooth non-negative cut-off function $\chi(x)$ verifying $\chi = 1$ on $[-1,1]$ and $\chi =0$ on $[-2,2]^c$ and get by decay of $B$ and because $g$ is bounded :
		\begin{align*}
			S_k&=\frac{t}{L^d}\sum_{x\in \Z^d}\chi(x/L^{1+\delta})B(x/L)g(\mu x\cdot(k-x))+o_L(1) \\
			I_k&=\frac{t}{L^d} \int_{\R^d}\chi(x/L^{1+\delta})B(x/L)g(\mu x\cdot(k-x))dx+o_L(1)
		\end{align*}
		where $\mu:=\frac{t}{L^2}$. From now on, we will consider the truncated version of $S_k$ and $I_k$ but keep their original notations.\\
		By inverse Fourier transform,
		\begin{align*}
			S_k-I_k=\int_{-\mu}^\mu S_k(\tau)-I_k(\tau)d\tau
		\end{align*}
		where :
		\begin{align*}
			S_k(\tau)&:=\frac{t}{\mu L^d}\sum_{x\in \Z^d}\chi(x/L^{1+\delta})B(x/L)\Ft(g)(\tau /\mu)e^{i\tau x \cdot (k-x)} \\
			I_k(\tau)&:=t\mu^{-1}L^{-d} \int_{\R^d}B(x/L)\chi(x/L^{1+\delta})\Ft(g)(\tau/\mu)e^{i\tau x\cdot (k-x)} \, .
		\end{align*}
		So :
		\begin{align*}
			\left|S_k-I_k \right| & \le \int_{|\tau|<\frac{1}{4}L^{-1-\delta}}|S_k(\tau)-I_k(\tau)|d\tau + \int_{\frac{1}{4}L^{-1-\delta}\le |\tau| <\mu}|S_k(\tau)|d\tau + \int_{\frac{1}{4}L^{-1-\delta}\le |\tau| <\mu}|I_k(\tau)|d\tau \\
			&=: A+B+C \, .
		\end{align*}
		\begin{itemize}
			\item \textbf{A term :} By Poisson summation :
			\begin{align*}
				\int_{|\tau|<\frac{1}{4}L^{-1-\delta}}S_k(\tau)d\tau=L^{-d}t\mu^{-1}\int_{|\tau|<\frac{1}{4}L^{-1-\delta}} \sum_{c\in\Z^d}\int_{\R^d}B(z/L)\chi(z/L^{1+\delta})\Ft(g)(\tau/\mu)e^{-ic\cdot z+i\tau z \cdot (k-z)}dzd\tau \, .
			\end{align*}
			For $c=0$, we get :
			\begin{align*}
				L^{-d}t\mu^{-1}\int_{|\tau|<\frac{1}{4}L^{-1-\delta}} \int_{\R^d}B(z/L)\chi(z/L^{1+\delta})\Ft(g)(\tau/\mu)e^{i\tau z \cdot (k-z)}dzd\tau&=\int_{|\tau|<\frac{1}{4}L^{-1-\delta}}I_k(\tau)d\tau \, .
			\end{align*}
			For $c\neq 0$, we get :
			\begin{align*}
				L^{-d}t\mu^{-1}&\int_{|\tau|<\frac{1}{4}L^{-1-\delta}} \sum_{c\in\Z^d-\{0\}}\int_{\R^d}B(z/L)\chi(z/L^{1+\delta})\Ft(g)(\tau/\mu)e^{-ic\cdot z+i\tau z \cdot (k-z)}dzd\tau\\ & = t\mu^{-1}\int_{|\tau|<\frac{1}{4}L^{-1-\delta}} \sum_{c\in\Z^d-\{0\}}\int_{\R^d}B(z)\chi(z/L^{\delta})\Ft(g)(\tau/\mu)e^{-i\phi_L(z)}dzd\tau 
			\end{align*}
			where $\phi_L(z)=Lc\cdot z-L^2\tau z \cdot (k-z)$, $\nabla \phi_L(z)=Lc-L^2\tau (k-2z)$, $\text{Hess} \phi_L(z)=2\tau L^2Id$. Now for  $L$ large enough, we have that $|k-z|\le 2L^\delta$ and so if $\tau L^{1+\delta}\le 1/4$, then $\left|\nabla \phi_L(z)\right|\ge L$ as $c \neq 0$. Furthermore, one has $\left| \frac{\text{Hess} \phi_L(x)}{\left|\nabla \phi_L(x)\right|^2}\right|\les \tau \les L^{-1-\delta}$ where we integrate, so by the non-stationary phase argument, one has :
			\begin{align*}
				&\left|t\mu^{-1}\int_{|\tau|<\frac{1}{4}L^{-1-\delta}} \sum_{c\in\Z^d-\{0\}}\int_{\R^d}B(z)\chi(z/L^{\delta})\Ft(g)(\tau/\mu)e^{-i\phi_L(z)}dzd\tau \right| \\ 
				& \hspace{1cm} \les t\mu^{-1}\int_{|\tau|<\frac{1}{4}L^{-1-\delta}}|\Ft(g)(\tau/\mu) |\sum_{c\in\Z^d-\{0\}}\left|\int_{\R^d}f(z)e^{-ic\cdot z}dz\right|d\tau  \\
				& \hspace{1cm} \text{where $f(z):=ie^{i\tau z\cdot(k-z)}\left[\frac{\sum_{j=1}^d\partial_j(B\chi(\cdot/L^{\delta}))(z)}{\sum_{j=1}^d\partial_j \phi(z)}-\frac{B(z)\chi(z/L^{\delta})\sum_{j=1}\sum_{l=1}^d\partial_{j,l}\phi(z)}{|\sum_{j=1}^d\partial_j \phi(z)|^2}\right]$} \\
				&\hspace{1cm} \les t\mu^{-1}\int_{|\tau|<\frac{1}{4}L^{-1-\delta}}|\Ft(g)(\tau/\mu) |L^{-1}\sum_{c\in\Z^d-\{0\}}\left|L\Ft(f)(c)\right|d\tau \\
				&\hspace{1cm} \les L^2 L^{-1-\delta}L^{-1} \ \ \ \text{because $Lf\in \mathcal{S}(\R)$} \\
				&\hspace{1cm}=o_L(1) \, .
			\end{align*}
			\item \textbf{C term :} We have :
			\begin{align*}
				C=\int_{\frac{1}{4}L^{-1-\delta}\le |\tau| <\mu}\left|t\mu^{-1} \int_{\R^d}B(x)\chi(x/L^{\delta})\Ft(g)(\tau/\mu)e^{i\psi_L(x)}dx\right|d\tau
			\end{align*}
			where $\psi_L(x)=\tau L^2x\cdot(k-x)$, $\nabla \psi_L(x)=\tau L^2 (k-2x)$, $\text{Hess} \psi_L(x)=-2\tau L^2 Id$. So, $\nabla\psi_L(x)=0 \iff x=k/2$ which is compatible with $|x|\les L^\delta$ for $L$ large enough (remember that $Lk$ is fixed). So by the stationary phase argument, we have :
			\begin{align*}
				\left|t\mu^{-1} \int_{\R^d}B(x)\chi(x/L^{\delta})\Ft(g)(\tau/\mu)e^{i\psi_L(x)}dx\right| \les t\mu^{-1} \left|\tau L^2\right|^{-d/2}
			\end{align*}
			hence :
			\begin{align*}
				C&\les\int_{\frac{1}{4}L^{-1-\delta}\le |\tau| <\mu}\left|\Ft(g)(\tau/\mu)\right|t\mu^{-1} \left|\tau L^2\right|^{-d/2}d\tau \\
				&\les \begin{cases}
					L^2 L^{-d}L^{(-1-\delta)(-d/2+1)}=L^{d(-1/2+\delta/2)}L^{1-\delta}\les L^{-1/2+d\delta} \ \ \text{for $d\ge3$}. \\
					L^2 L^{-d}\log\left(\frac{tL^\delta}{L}\right)=\log\left(\frac{tL^\delta}{L}\right) \ \ \text{for $d=2$}. \\
					L^2 L^{-d}t^{1/2}L^{-1}=t^{1/2} \ \ \text{for $d=1$}.
				\end{cases}
			\end{align*}
			For this technique to work in the case $d=2$, we would need $t=L^{1-\delta}+o(L^{1-\delta})$. \\
			
			\item \textbf{B term :} We have first by forgetting about the absolute value :
			\begin{align*}
				B=\mu^{-1}\int_{\frac{1}{4}L^{-1-\delta}<|\tau|<\mu}\Ft(g)(\tau/\mu)F(\tau)d\tau
			\end{align*}
			where
			\begin{align*}
				F(\tau)=L^{-d}t\sum_{x\in \Z^d}B(x/L)\chi(x/L^{1+\delta})e^{i\tau x \cdot (k-x)}\, .
			\end{align*}
			We set $p_i(x):=x_i(k_i-x_i)$ for $1\le i \le d$ and $G(s,n):=\sum_{q=0}^n e^{isq(k-q)}$ for $n\in \N$, $G(s,x):=G(s,[x])$ for $x\in \R$ where $[.]$ denotes the integer part function. \\
			Now by integration by parts, we get, forgetting about negligible boundary terms (B.T) (each boundary is evaluated at $0$ for some of the variables $(x_i)_{1\le i\le d}$ so it involves less integration and thus is smaller), and considering only the case $x\ge 0$ since the other cases are similar :
			\begin{align}
				F(\tau)=\frac{t}{L^d}\int_{x_i \ge 0}\partial_{x_1}...\partial_{x_d}\left[B(x/L)\chi(x/L^{1+\delta})\right]\prod_{j=1}^d G(\tau,x_j)dx_j + \textbf{B.T} \, .
			\end{align}
			We first estimate G. \\
			By Hua's lemma (\cite{Hualemma}) :
			\begin{align*}
				\|G(\tau,x_j)\|_{L_\tau^4([0,1])}, \ \|G(\tau,x_j)\|_{L_\tau^3([0,1])} \les |x_j|^{1/2+\delta}\les L^{(1+\delta)(1/2+\delta)} \, .
			\end{align*}
			Let $s\in [0,1]$ and $0\le a < q \le n$ be integers such that $(a,q)=1$ and $\left|s-\frac{a}{q}\right|<\frac{1}{qn}$ which exist by Dirichlet's approximation theorem. Then by Weyl's sum estimate :
			\begin{align*}
				|G(s,n)|\le \frac{n}{\sqrt{q}(1+n\|s-a/q\|^{1/2})}\le \frac{n}{\sqrt{q}} \, .
			\end{align*}
			In our case, $s\in [L^{-1-\delta},L^{-\eps}]$ where $\eps:=2-2\gamma$ as $t\le L^{2\gamma}$. So either $|n|<L^{2\eps}$ or $a/q \les L^{-\eps} $. \\
			\noi If $a=0$, then $|G(s,n)|\les |s|^{-1/2}\les L^{3/4}$. \\
			\noi If $a\neq 0$, then $|G(s,n)|\les L^{1+\delta-\eps/2}$ in both cases listed above, using $|n|\les L^{1+\delta}$. \\
			We finally use Hölder in $\tau$ and get for $d\ge 5$, 
			\begin{align*}
				|B| & \les \|G\|_{L^4}^{4}\|G\|_{L^\infty}^{d-4}L^{-d}\mu^{-1}t\int_{x_i \ge 0}\left|\partial_{x_1}...\partial_{x_d}\left[B(x/L)\chi(x/L^{1+\delta})\right]\right|dx_j \\ 
				& \les L^{4(1+\delta)(1/2+\delta)}\|G\|_{L^\infty}^{d-4}L^{-d}\mu^{-1}t\int_{x_i \ge 0}L^{-d}\left|\partial_{x_1}...\partial_{x_d}\left[B(x)\chi(x/L^{\delta})\right]\right|L^ddx_j \\ 
				& \les  L^{4(1+\delta)(1/2+\delta)}L^{(d-4)(1+\delta-\eps/2)}L^{-d}L^2 \\
				&\les L^{2\delta+4\delta^2+2\eps}L^{d(\delta - \eps/2)} = o_L(1) \ \ \ \text{for $\delta$ and $\eps$ small enough, well chosen.}
			\end{align*}
			For $d=4$, we obtain $L^{6\delta +4\delta^2}\neq o_L(1)$. \\
			
		\end{itemize}
		Hence our result for $d\ge 5$.
	\end{proof}
	\section{Proof of Theorem \ref{wavetheorem}}
	We present the proof of the kinetic derivation in Theorem \ref{wavetheorem}.
	We will denote by $E_L$ the set as in Theorem \ref{wavetheorem} where Propositions \ref{bound1} and \ref{bound2} hold true. We use the fact that for $0\le n\neq n' \le 2N$, we have :
	\begin{equation*}
		\E[J_n\cj{J_{n'}}]=0
	\end{equation*}
	as there are not the same amount of conjugated and non conjugated Gaussian random variables. \\
	\noi Then we have for $0\le t=sT_{max} \le T_{max}$ and $k\in \Z_L^d$ :
	\begin{multline*}
		\E\left[|\Ft_{x}(u)(t,k)|^2 \ind_{E_L}\right]  = \E\left[|J_0(t/T_{max},k)|^2 \ind_{E_L}\right] + \E\left[|J_1(t/T_{max},k)|^2 \ind_{E_L}\right] \\  + \sum_{\substack{2 \le n \le N}}\E\left[|J_n(t/T_{max},k)|^2 \ind_{E_L}\right] + 2\sum_{0 \le n \le N}\E\left[\Re \left(\cj{R_{N+1}(t/T_{max},k)}J_n(t/T_{max},k)\right)\ind_{E_L}\right]  \\
		+ \E\left[|R_{N+1}(t/T_{max},k)|^2 \ind_{E_L}\right] \, .
	\end{multline*}
	We set for $0\le t \le T_{max}$ and $k\in \Z^d$ :
	\begin{multline*}
		r_1(k,\frac{t}{T_{kin}},L):=  \sum_{\substack{2 \le n \le 2N}}\E\left[|J_n(t/T_{max},k/L)|^2 \ind_{E_L}\right] \\+ 2\sum_{0 \le n \le N}\E\left[\Re \left(\cj{R_{N+1}(t/T_{max},k/L)}J_n(t/T_{max},k/L)\right)\ind_{E_L}\right]  + \E\left[|R_{N+1}(t/T_{max},k/L)|^2 \ind_{E_L}\right] \, .
	\end{multline*}
	We use Propositions \ref{bound1} and \ref{bound2}, Cauchy-Schwarz inequality and Sobolev injections to get for $s$ large enough, $b>1/2$ :
	\begin{align*}
		\lim_{L\xrightarrow{}+\infty}\sup_{k\in \Z^d}\big|r_1(k,\frac{t}{T_{kin}},L)\big|L^{0+}=0 \, .
	\end{align*}
	\noi We set for $0\le t \le T_{max}$ and $k\in \Z^d$ :
	\begin{align*}
		r_2(k,\frac{t}{T_{kin}},L)&:=\E\left[|J_0(t/T_{max},k/L)|^2 \ind_{E_L}\right] - \E\left[|J_0(t/T_{max},k/L)|^2 \right] 
	\end{align*}
	\noi and get by similar arguments for $0\le t \le T_{max}$ and $k\in \Z_L^d$ :
	\begin{align*}
		\E\left[|J_0(t/T_{max},k)|^2 \ind_{E_L}\right] = \E\left[|J_0(t/T_{max},k)|^2 \right] + r_2(Lk,\frac{t}{T_{kin}},L) = n_{in}(k)+r_2(Lk,\frac{t}{T_{kin}},L)
	\end{align*}
	where $r_2(Lk,\frac{t}{T_{kin}},L)= O_L(e^{-cL^{\theta}})$. \\
	
	\noi We define $S_k$ and $I_k$ as in Lemma \ref{sumtointegral}. We write for $0\le t \le T_{max}$ and $k\in \Z_L^d$ :
	\begin{align*}
		\E\left[|J_1(t/T_{max},k)|^2 \ind_{E_L}\right]& =2\frac{\al^2 t^2}{L^d}\sum_{\substack{k_1+k_2=k \\ k_1,k_2 \in \Z_L^d}}n_{in}(k_1)n_{in}(k_2)\left(\frac{\sin(t\pi 2k_1 \cdot k_2)}{t\pi 2 k_1 \cdot k_2} \right)^2 \\
		&= 2 \al^2 t S_k   \, .
	\end{align*}
	We set for $0\le t \le T_{max}$ and $k\in \Z^d$ :
	\begin{align*}
		r_3(k,\frac{t}{T_{kin}},L):=2 \al^2 t S_{k/L} -2 \al^2 t I_{k/L} \, ,
	\end{align*}
	and get for $0\le t \le T_{max}$ and $k\in \Z_L^d$ :
	\begin{align*}
		\E\left[|J_1(t/T_{max},k)|^2 \ind_{E_L}\right]& =2 \al^2 t I_k +r_3(Lk,\frac{t}{T_{kin}},L) 
	\end{align*}
	where $r_3(Lk,\frac{t}{T_{kin}},L)=o_L(\al^2 t)$ by Lemma \ref{sumtointegral}. \\
	We finally set for $0\le t \le T_{max}$ and $k\in \Z^d$ :
	\begin{multline*}
		r_4(k,\frac{t}{T_{kin}},L):=2 \al^2 t \int_{\R^d}n_{in}(x)n_{in}(k/L-x)\left(\frac{\sin(t\pi 2x \cdot (k/L-x))}{t\pi 2 x \cdot (k/L-x)} \right)^2dx \\- 2 \al^2t \int_{\R^d}n_{in}(x)n_{in}(k/L-x)\delta(x \cdot (k/L-x))dx \,.
	\end{multline*}
	We finally have for $0\le t \le T_{max}$ and $k\in \Z_L^d$ :
	\begin{align*}
		\E\left[|J_1(t/T_{max},k)|^2 \ind_{E_L}\right]&=2 \al^2 t \int_{\R^d}n_{in}(x)n_{in}(k-x)\left(\frac{\sin(t\pi 2x \cdot (k-x))}{t\pi 2 x \cdot (k-x)} \right)^2dx + r_3(Lk,\frac{t}{T_{kin}},L)  \\
		&= 2 \al^2t \int_{\R^d}n_{in}(x)n_{in}(k-x)\delta(x \cdot (k-x))dx + r_3(Lk,\frac{t}{T_{kin}},L) \\
		&+r_4(Lk,\frac{t}{T_{kin}},L)
	\end{align*}
	where $r_4(Lk,\frac{t}{T_{kin}},L)=o_L(\al^2)$ for $t\ge L^{0+}\xrightarrow[]{}+\infty$ by a classical approximate identity result as the test function is a Schwartz function.  \\
	
	\noi Hence our result as $\al^2=T_{kin}^{-1}$, $\al^2t \le T_{max}/T_{kin}\le L^{-\delta}$ and setting $r~:=~r_1~+~r_2~+~r_4$.

	\appendix
	\section{Empirical measure interpretation}\label{empirical measure}
	
	First we give the heuristic equivalent of the convergence of an empirical measure. This approach is inspired by the many particles case. 
	Set for $k\in \Z^d$ :
	\begin{align*}
		\mu_k(t) := \frac{1}{L^d}\sum_{k\in \Z^d}|\Ft_{x}(u)(t,k/L)|^2\delta_{k/L}
	\end{align*}
	where $\delta_x$ denotes the Dirac measure supported on $x\in \R$. \\
	
	\noi Then, we wish to show for some class of $g$ :
	\begin{align*}
		\left(\E[\mu_k(t)],g \right)&=\frac{1}{L^d}\sum_{k\in \Z_L^d}\E\big[|\Ft_{x}(u)(t,k)|^2\big]g(k) \xrightarrow[L\xrightarrow{}+\infty]{}\int_{\R^d}g(k)f(t/T_{kin},k)dk
	\end{align*}
	meaning we would get a convergence :
	\begin{align*}
		\E[\mu_k(t)]\xrightarrow[L\xrightarrow{}+\infty]{\text{in some weak sense}}f(t/T_{kin},k)dk
	\end{align*}
	where $f$ is the solution to \eqref{limitequation}. \\
	\noi This can be somewhat justified by the following proposition :
	\begin{proposition} Let $f$ be the solution to \eqref{limitequation} given by Proposition \ref{LWP} up to a maximal time of existence $T^*$. Let $0<\delta < T^*$ such that we have a stronger convergence result i.e.~\eqref{strongerconvergence}. Then, we have for $0 \le t\le \delta$ :
		\begin{align*}
			\E[\mu_k(t)]\xrightarrow[L\xrightarrow{}+\infty]{\text{tempered distribution}}f(t,k)dk \, .
		\end{align*}
	\end{proposition}
	
	\begin{proof} \label{proofofempiricalmeasure}
		Let $g\in \mathcal{S}(\R^d)$ and $t\in [0,\delta T_{kin}[$ :
		\begin{multline*}
			\Bigg| \frac{1}{L^d}\sum_{k\in \Z_L^d}|\E[\Ft_{x}(u)(t,k)|^2]g(k) -\int_{\R^d}g(k)f(t/T_{kin},k)dk\Bigg|  \\
			\le \Bigg| \frac{1}{L^d}\sum_{k\in \Z_L^d}|\E[\Ft_{x}(u)(t,k)|^2]g(k) -\frac{1}{L^d}\sum_{k\in \Z_L^d}f(t/T_{kin},k)g(k)\Bigg| \\ +\Bigg| \int_{\R^d}g(k)f(t/T_{kin},k)dk -\frac{1}{L^d}\sum_{k\in \Z_L^d}f(t/T_{kin},k)g(k)\Bigg| \, .
		\end{multline*}
		We use Riemann summation and the convergence assumption to show this is 
		\begin{align*}
			&\le \sup_{t \in [0,\delta ]} \sup_{k\in \Z_L^d}\big|\E[\Ft_{x}(u)(t,k)|^2] -f(t/T_{kin},k)\big|\frac{1}{L^d}\sum_{k\in \Z_L^d}|g(k)|  +o_L(1)  \\
			& = o_L(1) \, .
		\end{align*}
		Therefore, we have convergence at least in the tempered distribution sense i.e. in $\mathcal{S}'$.
	\end{proof}
	\begin{remark}
		Now formally, one could approach $\E[\mu_k]$ by a Law of Large Numbers. This would translate our result into a Law of Large Numbers. 
	\end{remark}
	
	\section{Proof of Lemma \ref{time of existence}}\label{proofoflemma}
	\noi We first take $u_0$ as in Theorem \ref{theorem1} and compute :
	\begin{align*}
		\|u_0\|_{H^s}^2&=\al^2 \int_{\T_L^d}\Big|\frac{1}{L^{d/2}}\sum_{k\in\Z^d}\jb{k/L}^s\sqrt{n_{in}(k/L)}g_ke^{i2\pi (k/L)\cdot x}\Big|^2dx \\
		&=\al^2 \frac{1}{L^d}\int_{\T_L^d}\Big(\sum_{k\in\Z^d}\jb{k/L}^s\sqrt{n_{in}(k/L)}g_ke^{i2\pi (k/L)\cdot x}\Big)\cj{\Big(\sum_{k\in\Z^d}\jb{k/L}^s\sqrt{n_{in}(k/L)}g_ke^{i2\pi (k/L)\cdot x}\Big)}dx\\
		&=\al^2 \sum_{k\in\Z^d}\jb{k/L}^{2s}n_{in}(k/L)|g_k|^2 
	\end{align*}
	hence
	\begin{align*}
		\E\big[\|u_0\|_{H^s}^2\big]=\al^2 \sum_{k\in\Z_L^d}\jb{k}^{2s}n_{in}(k)
	\end{align*}
	which yields $\E\big[\|u_0\|_{H^s}^2\big]^{1/2}\sim \al C_s^{1/2}L^{d/2}$ where $C_s=\int_\R n_{in}(x)\jb{x}^{2s}dx$ is independent of $L$.
	By independence,
	\begin{align*}
		V\big(\|u_0\|_{H^s}^2\big)&=\al^4 \sum_{k\in\Z^d}\jb{k/L}^{4s}n_{in}^2(k/L)V\big(|g_k|^2\big )  \\
		&=\al^4 \sum_{k\in\Z_L^d}\jb{k}^{4s}n_{in}^2(k) \\
		&\sim \al^4 L^d c_s \ \ \ \text{where} \ c_s:=\int_\R \jb{x}^{4s}n_{in}^2(x)dx \, .
	\end{align*}
	So
	\begin{align*}
		\P\Big(\big|\|u_0\|_{H^s}^2-\E\big[\|u_0\|_{H^s}^2\big]\big|\ge \frac{\E\big[\|u_0\|_{H^s}^2\big]}{2}\Big) &\le \frac{4V\big(\|u_0\|_{H^s}^2\big)}{\E\big[\|u_0\|_{H^s}^2\big]^2} \\
		& \sim \frac{4c_s}{C_s^2}L^{-d} \, .
	\end{align*}
	So for $L$ large enough, on an event of probability $\ges 1-\frac{4c_s}{C_s^2}L^{-d}$, we have that $\|u_0\|_{H^s}\asymp\E\big[\|u_0\|_{H^s}^2\big]^{1/2}\asymp \al C_s^{1/2}L^{d/2}$. \\

	\section{Results on the rescaled torus} \label{rescaledtorus}
	\noi We recall the classical quadratic NLS equation on the torus of size one :
	\begin{equation}
		\begin{cases} \label{truenls}
			i\dt v - \Dl v =  v^2  \\
			v(0,x)=v_0(x) 
		\end{cases}
	\end{equation}
	where $(t,x)\in \R\times \T^d$ and $\Dl:=\sum_{i=1}^d \partial_{x_i}^2$.
	Now we state a result derived directly from Theorems \ref{theorem1} and \ref{theorem2} :
	
	\begin{corollary}\label{truetheorem}
		Let $d\ge 1$, $\delta >0$, $\gamma > \delta$, $L^{-\gamma} \le \al \le L^{-\delta}$ and the well prepared initial data
		\begin{align*}
			v_0(x)=\frac{\al L^2}{L^{d/2}}\sum_{k\in\Z^d}\sqrt{n_{in}(k/L)}g_k(\omega) e^{i2\pi (k/L)\cdot x}
		\end{align*} with $n_{in}\in \mathcal{S}(\R^d)$ non-negative, $(g_k(\omega))_{k\in \Z^d}$ family of i.i.d centered normalized complex Gaussians. Then we have that there exists $\theta>0$, $K(\theta,\gamma ),c(\theta,\gamma)>0$, $s>d/2$ and $b>1/2$ close to $1/2$ such that for all $L$ large enough, there exists an event $E_L\subset \O$ verifying  $\P(E_L^c)\le Ke^{-cL^\theta}$, such that the solution to \eqref{truenls} in distributional (or even fixed point) sense is defined up to time $T_{max}$ on $E_L$ in $h^{s,b}$ where :
		\begin{equation*}
			\begin{cases}
				T_{max} =L^{-2}\al^{-\frac{2}{1+1/d}}L^{-\frac{2\delta}{1+1/d}} \ \ \ \text{if $L^{-\frac{d+1}{4-2/d}-\delta}\le\al \le L^{-\delta}$} \\
				T_{max} =L^{-2}\al^{-1}L^{-\delta} \ \ \ \text{if $L^{-\gamma}\le  \al \le L^{-\frac{d+1}{4-2/d}-\delta} \, .$} 
			\end{cases}
		\end{equation*}
	\end{corollary}

	\noi We now prove that Theorems \ref{theorem1} and \ref{theorem2} imply Corollary \ref{truetheorem}.
	\begin{proof}
		Take $u$ the solution as in Theorem \ref{theorem1} or \ref{theorem2}. Set $$v(t,x):= L^2 u(L^2t,Lx) \, .$$ Then $v$ is a solution of \eqref{truenls}.
	\end{proof}
	\noi Now to compare the time of existence of Corollary \ref{truetheorem} with Proposition \ref{basiclwp}, since $v(t,x)~:~=~L^2~u(L^2t,Lx)$, we get on an event of probability $\ges 1-\frac{4c_s}{C_s^2}L^{-d}$ that the time of existence given by Proposition \ref{basiclwp} on $\T^d$ is of order $L^{-2}\al^{-1}L^{-d/2}C_s^{-1/2}$ whereas on an event of probability $\ge 1-Ke^{-cL^\theta}$, the time of existence given by Corollary \ref{truetheorem} is of order $T_{max}$ hence an improvement as soon as 
	$$\begin{cases}
		\al^{\frac{1-1/d}{1+1/d}}\le C_s^{1/2}L^{d/2-\frac{2\delta}{1+1/d}} \ \ \ \text{if $L^{-\frac{d+1}{4-2/d}-\delta}\le\al \le L^{-\delta}$} \\
		L^{d/2-\delta}\ge C_s^{-1/2} \ \ \ \text{if $ L^{-\gamma}\le \al \le L^{-\frac{d+1}{4-2/d}-\delta} \, .$} 
	\end{cases}$$ which is covered by our range of $\al$. \\

	\noi But we can also assess the size of the initial data on $\T^d$ directly.
	\noi Let the setting be the same as in Corollary \ref{truetheorem}. Then for $L$ large enough, on an event of probability $\ges 1-\frac{4c_s}{C_s^2}L^{-d}$, we have that
	$$\|v_0\|_{\dot{H}^s}^2 \asymp C(s) L^{2s}L^4 \al^2 \, .$$
	\begin{proof}
		We adapt the proof in Section \ref{proofoflemma} of the Appendix : \\
		If we take $v_0$ as in Corollary \ref{truetheorem} and compute $\E\big[\|v_0\|_{H^s}^2\big]$, one would get :
		\begin{align*}
			\E\big[\|v_0\|_{H^s}^2\big]&=\frac{L^4 \al^2}{L^d}\sum_{k \in \Z^d}\jb{k}^{2s}n_{in}(k/L) \ \ \ \text{and} \\
			\E\big[\|v_0\|_{\dot{H}^s}^2\big]&=\frac{L^4 \al^2}{L^d}\sum_{k \in \Z^d}|k|^{2s}n_{in}(k/L) \asymp K_s L^{2s}L^4 \al
		\end{align*}
		where $K_s:=\int_\R |x|^{2s}n_{in}(x)dx$ and we can do the same reasoning as above for the $\dot{H}^s$ norm of $v_0$. 
		
	\end{proof}

	\noi This second estimate means that if we apply the result of Proposition \ref{basiclwp} directly on the rescaled torus $\T^d$, one would get a time of existence of order $C_s^{-1/2}L^{-s}L^{-2}\al^{-1}$ which is worse than if we apply Proposition \ref{basiclwp} directly on $\T_L^d$ without rescaling as long as $s>d/2$ . \\
	This means, that the initial condition has a peculiar property : we establish a better time of existence by scaling it on the torus $\T_L^d$, applying a classical well-posedness result and then rescaling it on $\T^d$ rather than simply applying a classical result directly.

	\section{Technical lemmas}
	\noi We present here some technical results used in the paper. The following lemma is used in Section \ref{proofbound1} at line \eqref{technicaluse}.
	\begin{lemma}\label{technical}
		Let $\al, \beta \in \R$, $b>0$ and $n\in \R$ such that $n>b+1/2$. Then :
		\begin{equation*}
			\int_\R \frac{\jb{\tau}^{2b}}{\jb{\tau-\al}^n\jb{\tau-\beta}^n}  d\tau 
			\les\left[\max \left(\jb{\al},\jb{\beta}\right)^{2b}\jb{\al-\beta}^{-n+1}\right] \, .
		\end{equation*}
	\end{lemma}
	\begin{proof} 
		\begin{itemize}
			\item If $|\al-\beta|\ge 1$. We write
			\begin{equation*}
				\tau =\frac{-\beta}{\al-\beta}(\tau-\al)+\frac{\al}{\al-\beta}(\tau-\beta) \, .
			\end{equation*}
			By triangular inequality, we have :
			\begin{align*}
				\int_\R \frac{\jb{\tau}^{2b}}{\jb{\tau-\al}^n\jb{\tau-\beta}^n}  d\tau & \le \jb{\frac{\beta}{\al-\beta}}^{2b}\int_\R \frac{\jb{\tau-\al}^{2b}}{\jb{\tau-\al}^n\jb{\tau-\beta}^n}  d\tau + \jb{\frac{\al}{\al-\beta}}^{2b}\int_\R \frac{\jb{\tau-\beta}^{2b}}{\jb{\tau-\al}^n\jb{\tau-\beta}^n}  d\tau
			\end{align*}
			and using $|\al-\beta|\ge 1$, we have :
			\begin{align*}
				\jb{\frac{\beta}{\al-\beta}}^{2b}\int_\R \frac{\jb{\tau-\al}^{2b}}{\jb{\tau-\al}^n\jb{\tau-\beta}^n}  d\tau & \les \frac{\jb{\beta}^{2b}}{\jb{\al-\beta}^{2b}}\int_\R \frac{\jb{\tau}^{2b}}{\jb{\tau}^n\jb{\tau+\al-\beta}^n}  d\tau \, .
			\end{align*}
			We split the integral in two and get it is smaller than :
			\begin{align*}
				& \les \frac{\jb{\beta}^{2b}}{\jb{\al-\beta}^{2b}}\Big(\int_{|\tau+\al-\beta|\le |\al-\beta|} \frac{\jb{\tau}^{2b}}{\jb{\tau}^n\jb{\tau+\al-\beta}^n}  d\tau+\int_{|\tau+\al-\beta|\ge |\al-\beta|} \frac{\jb{\tau}^{2b}}{\jb{\tau}^n\jb{\tau+\al-\beta}^n}  d\tau\Big) \\
				& \les \frac{\jb{\beta}^{2b}}{\jb{\al-\beta}^{2b}}\Big(\int_{|\tau+\al-\beta|\le |\al-\beta|} \frac{\jb{\tau}^{2b}}{\jb{\tau}^n}  d\tau+\int_{\R} \frac{\jb{\tau}^{2b}}{\jb{\tau}^n\jb{\al-\beta}^n}  d\tau\Big) \\
				& \les \frac{\jb{\beta}^{2b}}{\jb{\al-\beta}^{2b}} (\jb{\al-\beta}^{2b-n+1}+\jb{\al-\beta}^{-n}) \\
				& \les \frac{\jb{\beta}^{2b}}{\jb{\al-\beta}^{n-1}} \, .
			\end{align*}
			We use the same reasoning for 
			$$\jb{\frac{\al}{\al-\beta}}^{2b}\int_\R \frac{\jb{\tau-\beta}^{2b}}{\jb{\tau-\al}^n\jb{\tau-\beta}^n}  d\tau$$
			and obtain :
			\begin{align*}
				\jb{\frac{\al}{\al-\beta}}^{2b}\int_\R \frac{\jb{\tau-\beta}^{2b}}{\jb{\tau-\al}^n\jb{\tau-\beta}^n}  d\tau \les \frac{\jb{\al}^{2b}}{\jb{\al-\beta}^{n-1}} \, .
			\end{align*}
			Hence our result.
			\item If $|\al-\beta|\le 1$. We have :
			\begin{equation*}
				1+|\tau-\al| \le 2 +|\tau-\beta| \le 3 + |\tau-\al|
			\end{equation*}
			hence there exists $C>0$ such that
			\begin{equation*}
				\jb{\tau-\al}\sim C\jb{\tau-\beta} \, .
			\end{equation*}
			This yields
			\begin{align*}
				\int_\R \frac{\jb{\tau}^{2b}}{\jb{\tau-\al}^n\jb{\tau-\beta}^n}  d\tau & \les \int_\R \frac{\jb{\tau}^{2b}}{\jb{\tau-\al}^{2n}}d\tau \\
				& \les \int_\R \frac{\jb{\tau}^{2b}}{\jb{\tau}^{2n}}d\tau + \int_\R \frac{\jb{\al}^{2b}}{\jb{\tau}^{2n}}d\tau \, .
			\end{align*}
			Using the conditions $n>b+1/2$, we get :
			\begin{align*}
				\int_\R \frac{\jb{\tau}^{2b}}{\jb{\tau-\al}^n\jb{\tau-\beta}^n}  d\tau & \les 1+\jb{\al}^{2b} \, .
			\end{align*}
			Hence the result as $\jb{\al-\beta}^{-n}\ges 1$.
		\end{itemize}
		
	\end{proof}

	\section{Classical result of local well-posedness}
	We present here the proof of Proposition \ref{basiclwp}. Our goal is to get a time of existence as large as possible. \\
	More specifically, we will work with initial conditions as in Corollary \ref{truetheorem}. A computation similar to the one in Section \ref{proofoflemma} of the appendix yields :
	\begin{align*}
		\|v_0\|_{H^s}^2\asymp C_s \al^2 L^4
	\end{align*}
	meaning the size of $v_0$ depends on the regime of $\al$. However for the most interesting case that is $\al$ not too small, the initial data is big.
	\subsection{Classical proof}\label{classical proof}
	
	\label{basicprooflwp} First we recall the classical approach. \\
	\noi We use the fact that $H^s(\T_L^d)$ is an algebra for $s>d/2$. This means that for $u$ solution to \eqref{truenls}, we have formally :
	\begin{align*}
		\|\dt(e^{it\Dl}u)\|_{H^s(\T_L^d)} =\| e^{it\Dl}(u)^2\|_{H^s(\T_L^d)} \les  \|u\|_{H^s(\T_L^d)}^2 \, .
	\end{align*}
	More rigorously, we can run a fixed point argument on the Duhamel formulation in $C_t^0 ( H_x^s(\T_L^d))$.
	This shows that the equation is locally well-posed in $H^s(\T_L^d)$ with a time of existence of at least $T \sim \|u_0\|_{H^s(\T_L^d)}^{-1}$. 
	
	\subsection{Bourgain spaces technique}
	
	\noi We state the Proposition in $\T^d$ for the sake of simplicity, as nothing is modified if we look at the same well-posedness problem on $\T_L^d$.\\
	
	\noi In order to improve the regularity threshold given by Subsection \ref{classical proof}, we use the $X^{s,b}$ spaces as introduced by Bourgain in \cite{Bourgain}. From now on, the Fourier transform in space denotes the usual Fourier transform on the torus of size $1$. We set :
	\begin{align*}
		\| u\|_{X^{s,b}}^2&:=\sum_{n \in \Z}\int_\R \jb{n}^{2s} \jb{\tau-|n|^2}^{2b} |\Ft_{t,x}(u)(\tau,n)|^2 d\tau \\
		\| u\|_{X_T^{s,b}}&:=\inf\{\| v\|_{X^{s,b}}, \ v_{|]0,T[}= u_{|]0,T[} \ \}
	\end{align*}
	and use the same techniques as in \cite{TzvetkovLWP} or \cite{Ginibre}, for instance, to show that if $\psi$ is a $\mathcal{C}_c^\infty$ function with support in $[0,2]$ and such that $\psi \equiv 1$ on $[0,1]$ and $\psi_T:=\psi(\frac{.}{T})$ then
	\begin{align*}
		\Phi : u\in X_T^{s,b} \mapsto \psi_T(t)e^{-it\Dl}u_0 + \psi_T(t)\int_0^t e^{-i(t-s)\Dl}u^2(s)ds \in X_T^{s,b}
	\end{align*}
	is well defined and a contraction for suitable $s$ and $b>1/2$ we determine later. 
	\subsubsection{Key Propositions}
	The two key inequalities needed are : 
	\begin{itemize}
		\item the integral in time estimate :
		\begin{proposition} \label{duhamel} For $T\ge 1$, for any $0\le b'<1/2$, $b\ge b'$ such that $b+b' \le 1$, 
			\begin{equation*}
				\Big\|\psi_T(t)\int_0^t f(s)ds\Big\|_{H^b}  \le C T^{1}\|f\|_{H^{-b'}} \, .
			\end{equation*}
		\end{proposition}
		\item and the multilinear estimate :
		\begin{proposition} \label{multilinBourgain estimate}
			For any $0<b'<1/2<b$ such that $b+b' <1$, for $s>\frac{d-2}{2}$
			\begin{equation*}
				\|u_1 u_2\|_{X^{s,-b'}}\le C\|u_1\|_{X^{s,b}}\|u_2\|_{X^{s,b}} \, .
			\end{equation*}
		\end{proposition}
	\end{itemize}
	\noi In order to prove the later Proposition, one needs the key bilinear estimate first proved in~~\cite{Bourgain}~~:
	\begin{lemma} \label{lemmaBourgain}
		For all $\eps >0$, there exists $C_\eps >0$ such that for every $L,N$ dyadic vectors in $\R^d$, if
		\begin{align*}
			u_0(x)&=\sum_{N\le \|n\|_\infty \le 2N}c_n e^{2i \pi n\cdot x} \\
			v_0(x)&=\sum_{L\le \|n\|_\infty \le 2L}d_n e^{2i \pi n\cdot x}
		\end{align*}
		then
		\begin{equation*}
			\|e^{it\Dl}u_0 e^{it\Dl}v_0\|_{L^2([0,1]_t\times \T^d)} \le C_\eps \min(N,L)^{\frac{d-2}{2}+\eps}\|u_0\|_{L^2(\T^d)}\|v_0\|_{L^2(\T^d)} \, .
		\end{equation*}
	\end{lemma}
	\noi and the following Lemma which allows to go from a local in time bilinear estimate to a Bourgain space estimate :
	\begin{lemma}\label{lemmaL4toBourgain}
		Let $b>1/2$, $\eps >0$, there exists $C_\eps >0$ such that for every $L,N$ dyadic vectors in $\R^d$, if $u,v \in X^{0,b}$ such that
		\begin{align*}
			u(t,x)&=\sum_{N\le \|n\|_\infty \le 2N}c_n(t) e^{2i \pi n\cdot x} \\
			v(t,x)&=\sum_{L\le \|n\|_\infty \le 2L}d_n(t) e^{2i \pi n\cdot x}
		\end{align*}
		then
		\begin{equation*}
			\|u v\|_{L^2(\R_t\times \T^d)} \le C_\eps \min(N,L)^{\frac{d-2}{2}+\eps}\|u\|_{X^{0,b}}\|v\|_{X^{0,b}} \, .
		\end{equation*}
	\end{lemma}
	\noi The key difference with the usual setting is that we have $T\ge 1$ and even large. This will impact the first inequality i.e. Proposition \ref{duhamel} and we check it does not impact Lemma \ref{lemmaL4toBourgain}. The rest of the argument is a direct adaptation of Theorem 3 of \cite{TzvetkovLWP} in the case of $\T^d$ with a different numerology. Therefore we only mention the adapted proof of Proposition \ref{duhamel} based on \cite{Ginibre} and repeat the proof of Lemma \ref{lemmaL4toBourgain} based on \cite{TzvetkovLWP}. \\
	\subsubsection{Proof of Proposition \ref{duhamel}}
	We write for $w>0$ :
	\begin{align*}
		\psi_T \int_0^t f(s)ds &=\psi_T\int_\R \frac{e^{it\tau}-1}{i\tau}\Ft_{t}(f)(\tau)d\tau \\
		& = \psi_T\sum_{k=1}^\infty \frac{t^k}{k!}\int_{|\tau|T^w\le 1} (i\tau)^{k-1}\Ft_{t}(f)(\tau)d\tau - \psi_T\int_{|\tau|T^w\ge 1} (i\tau)^{-1}\Ft_{t}(f)(\tau)d\tau \\
		& + \psi_T\int_{|\tau|T^w\ge 1} e^{it\tau}(i\tau)^{-1}\Ft_{t}(f)d\tau \\
		& := I + II + III \, .
	\end{align*}
	And we can bound :
	\begin{align*}
		\|I\|_{H^b}&\le \sum_{k=1}^\infty\|t^k\psi_T\|_{H^b} \frac{1}{k!}T^{w(1-k)}\|f\|_{H^{-b'}}\sqrt{\int_{|\tau|T^w\le 1} \jb{\tau}^{2b'}d\tau} 
	\end{align*}
	and $$\|t^k\psi_T\|_{L^2}=\sqrt{\int_\R t^{2k}\psi^2(t/T)dt}=T^{k+1/2}\sqrt{\int_\R t^{2k}\psi^2(t)dt}$$
	which is larger for $T\ge 1$ than
	$$\|t^k\psi_T\|_{\dot{H}^b}\le CT^{k+1/2-b}\sqrt{\int_\R t^{2k}\psi^2(t)dt}$$
	so one gets in the end
	\begin{align*}
		\|I\|_{H^b}&\le C \sum_{k=1}^\infty  \frac{T^{k+1/2}}{k!}T^{w(1-k)}\|f\|_{H^{-b'}}T^{w(-1/2)} \\
		& \le C T^{\frac{1}{2}+\frac{w}{2}}  (e^{T^{1-w}}-1)\|f\|_{H^{-b'}}
	\end{align*}
	as $b'\ge 0$. \\
	
	\noi We also have :
	\begin{align*}
		\|II\|_{H^b} & \le \|\psi_T\|_{H^b}\|f\|_{H^{-b'}} \sqrt{\int_{|\tau|T^w \ge 1}|\tau|^{-2}\jb{\tau}^{2b' d\tau}} \\
		&\le C T^{1/2}\|f\|_{H^{-b'}}T^{-w(b'-1/2)} \\
		& \le CT^{\frac{1}{2}+w(\frac{1}{2}-b')}
	\end{align*}
	as $b' <1/2$. \\
	
	\noi And finally
	\begin{align*}
		\big\|\int_{|\tau|T^w\ge 1} e^{it\tau}(i\tau)^{-1}\Ft_{t}(f)(\tau)d\tau \big\|_{H^b} & = \big\|\Ft_{t}^{-1}\big(\ind_{|\tau|T^w\ge 1} e^{it\tau}(i\tau)^{-1}\Ft_{t}(f)(\tau)\big) \big\|_{H^b} \\
		&= \sqrt{\int_\R \ind_{|\tau|T^w\ge 1} \jb{\tau}^{2b}|\tau|^{-2}|\Ft_{t}(f)(\tau)|^2d\tau} \\
		& \le  \sup_{|\tau|T^w\ge 1}|\tau|^{-1}\jb{\tau}^{b+b'}\|f\|_{H^{-b'}} \\
		& \le C T^{w(1-b-b')}\|f\|_{H^{-b'}}
	\end{align*}
	as $b+b'\le 1$.
	And
	\begin{align*}
		\big\|\int_{|\tau|T^w\ge 1} e^{it\tau}(i\tau)^{-1}\Ft_{t}(f)(\tau)d\tau \big\|_{L^2} \le CT^{w(1-b')}\|f\|_{H^{-b'}}
	\end{align*}
	so
	\begin{align*}
		\|III\|_{H^b}&=\big\|\jb{\tau}^b\Ft_{t}(\psi_T)*\Ft_{t}\Big(\int_{|\tau|T^w\ge 1} e^{it\tau}(i\tau)^{-1}\Ft_{t}(f)(\tau)d\tau\Big)\big\|_{L^2} \\
		& \le C\Big(\big\||\tau^b| \Ft_{t}(\psi_T)\big\|_{L^1}\big\|\int_{|\tau|T^w\ge 1} e^{it\tau}(i\tau)^{-1}\Ft_{t}(f)(\tau)d\tau \big\|_{L^2} \\
		&+ \big\|\Ft_{t}(\psi_T)\big\|_{L^1}\big\|\int_{|\tau|T^w\ge 1} e^{it\tau}(i\tau)^{-1}\Ft_{t}(f)(\tau)d\tau \big\|_{H^b}\Big) \\
		& \le C(T^{w(1-b')}T^{-b}+T^{w(1-b-b')}) \\
		& \le CT^{w(1-b')-\min(1,w)b} \, .
	\end{align*}
	Gathering all the bounds, we get that the best case is when $w=1$ and in that case the worst bound is of order $T$. The case where $w<0$ gives no better result. Thus one gets
	\begin{align*}
		\Big\|\psi_T(t)\int_0^t f(s)ds\Big\|_{H^b}  \le C T^{1}\|f\|_{H^{-b'}} \, .
	\end{align*}
	This concludes the proof of Proposition \ref{duhamel}. \\
	
	\subsubsection{Proof of Lemma \ref{lemmaL4toBourgain}}
	We first assume that $u_0$ and $v_0$ are with support in time in $[0,1]$. In this case we write : 
	\begin{align*}
		u(t):=e^{-it\Dl}f(t), \ v(t):=e^{-it\Dl}g(t) \, .
	\end{align*}
	So by Fourier inversion formula
	\begin{align*}
		uv(t)=\frac{1}{(2\pi)^2}\int_{-\infty}^\infty \int_{-\infty}^\infty e^{it\tau}e^{it\s}e^{-it\Dl}\Ft_t(f)(\tau)e^{-it\Dl}\Ft_{t}(g)(\s)d\tau d\s \, .
	\end{align*}
	Using Cauchy-Schwarz inequality in $\s,\tau$ with $b>1/2$ as well as Lemma \ref{lemmaBourgain} :
	\begin{align*}
		\|uv\|_{L^2([0,1]\times \T^d)} &\le \frac{1}{(2\pi)^2}\int_{-\infty}^\infty \int_{-\infty}^\infty \|e^{-it\Dl}\Ft_{t}(f)(\tau)e^{-it\Dl}\Ft_{t}(g)(\s)\|_{L^2([0,1]\times \T^d)}d\tau d\s \\
		& \le \frac{1}{(2\pi)^2}\int_{-\infty}^\infty \int_{-\infty}^\infty C_\eps \min(N,L)^{\frac{d-2}{2}+\eps}\|\Ft_{t}(f)(\tau)\|_{L^2( \T^d)}\|\Ft_{t}(g)(\s)\|_{L^2( \T^d)}d\tau d\s \\
		& \le C_\eps \min(N,L)^{\frac{d-2}{2}+\eps} \|\jb{\tau}^{b}\Ft_{t}(f)(\tau)\|_{L^2(\R\times \T^d)}\|\jb{\s}^{b}\Ft_{t}(g)(\s)\|_{L^2( \R\times\T^d)} \\
		&=  C_\eps \min(N,L)^{\frac{d-2}{2}+\eps} \|u\|_{X^{0,b}}\|v\|_{X^{0,b}} \, .
	\end{align*}
	Now let us go to the case where $u$ and $v$ are not necessarily with support in $[0,1]$. We take $\psi \in \C_c^\infty$ non-negative with support included in $[0,1]$ such that $\|\psi\|_{L^\infty} \le 1$ and such that  $(\psi_n):=(\psi(.-n/2))_{n\in\Z}$ is a partition of unity and for $n\in \Z$, $\psi_n \psi_{n+1}$ has support in $[\frac{n+1}{2},\frac{n}{2}+1]$. We write
	\begin{align*}
		u(t)=\sum_{n\in \Z}\psi(t-n/2)u(t)=:\sum_{n\in \Z}u_n(t), \ v(t)=\sum_{m\in \Z}\psi(t-m/2)v(t)=:\sum_{m\in \Z}v_m(t) \, .
	\end{align*}
	We notice that the $u_n$ are almost orthogonal i.e. :
	\begin{align*}
		\|u\|_{L^2_t}^2&=\|\sum_{n\in\Z}u_n\|_{L^2_t}^2 \\
		&=\sum_{n,m\in \Z}\langle u_n, u_m \rangle_{L^2} \\
		&=\sum_{n\in \Z}\|u_n\|_{L^2}+\langle u_n, u_{n-1} \rangle_{L^2}+\langle u_n, u_{n+1} \rangle_{L^2} \\
		& \ge \sum_{n\in \Z}\|u_n\|_{L^2}-\int_{[n/2,(n+1)/2]}u^2-\int_{[(n+1)/2,(n+2)/2]}u^2 \\
		&\ge \sum_{n\in \Z}\|u_n\|_{L^2}-2\int_\R u^2 \, .
	\end{align*}
	
	And then by triangular inequality :
	\begin{align*}
		\|uv\|_{L^2(\R\times \T^d)} &\le \sum_{n,m\in \Z}\|u_nv_m\|_{L^2(\R\times \T^d)} \\
		& \le C_\eps \min(N,L)^{\frac{d-2}{2}+\eps} \sum_{n,m\in \Z}  \|u_n\|_{X^{0,b}}\|v_m\|_{X^{0,b}} \\
		&=C_\eps \min(N,L)^{\frac{d-2}{2}+\eps} \sum_{n\in \Z}  \|u_n\|_{X^{0,b}}\sum_{m\in \Z}\|v_m\|_{X^{0,b}}
	\end{align*}
	and by the previous computation :
	\begin{align*}
		\sum_{n\in \Z}  \|u_n\|_{X^{0,b}} &= \sum_{n\in \Z}  \|S(-t)u_n\|_{H^b_tL^2_x} \\
		&= \sum_{n\in \Z}  \|\jb{\nabla}^bS(-t)u_n\|_{L^2_tL^2_x} \\
		&\le 3   \|\jb{\nabla}^bS(-t)\sum_{n\in \Z}u_n\|_{L^2_tL^2_x} \ \ \ \text{by almost orthogonality of the $u_n$ functions} \\
		&=3\|u\|_{X^{0,b}}
	\end{align*}
	so
	\begin{align*}
		\|uv\|_{L^2(\R\times \T^d)}\le C_\eps \min(N,L)^{\frac{d-2}{2}+\eps}\|u\|_{X^{0,b}} \|v\|_{X^{0,b}} \, .
	\end{align*}
	This concludes the proof. \\
	
	\noi In the end, the fixed point theorem gives a time of existence of $T\sim \|u_0\|_{H^s}$ in $C([-T,T],H^s(\T^d))$ for $s>\frac{d}{2}-1$.

	\begin{ackno}\rm
		The author E. Brun was supported by a CDSN financed by Université Lyon 1, Lyon, France.
		I would like to thank my PhD supervisor Nikolay Tzvetkov who was of a considerable help.
	\end{ackno}

\end{document}